\documentclass[11pt]{amsart}
\usepackage[utf8]{inputenc}
\usepackage{amsmath}
\usepackage{bbm}

\usepackage[a4paper,top=2.5cm,bottom=3cm,left=2.5cm,right=2.5cm]{geometry}
\usepackage{latexsym}
\usepackage{xcolor}
\usepackage{enumerate}
\usepackage{xcolor}

 \usepackage[usenames,dvipsnames]{pstricks}
 \usepackage{epsfig}
 \usepackage{pst-grad} % For gradients
\usepackage{pst-plot} % For axes

\title{Geodesics and admissible-path spaces in Carnot Groups}
\author{A. A. Agrachev}
\author{A. Gentile}
\author{A. Lerario}

\newcommand{\sinkt}{\frac{1}{\sqrt{\pi}}\sin kt}
\newcommand{\coskt}{\frac{1}{\sqrt{\pi}}\cos kt}
\newcommand{\dd}{\Delta^{2}}
\newcommand{\dds}{\left(\Delta^{2}\right)^{*}}
\newcommand{\g}{\mathfrak{g}}

\newcommand{\sod}{\mathfrak{so}\mathnormal{(d)}}
\newcommand{\he}{\textrm{He}}
\newcommand{\R}{\mathbb{R}}
\newcommand{\omegat}{\widehat{\Omega}_{\epsilon p}}
\newcommand{\gammas}{\Gamma_{\mathnormal{k|m_{1},\ldots,m_{r}}}}
\newcommand{\gammasn}{\Gamma_{\mathnormal{k|m_{1},\ldots,m_{r}|\vec{n}}}}
\newcommand{\diag}{\textrm{Diag}}
\newcommand{\ii}{\textrm{i}}
\newcommand{\coker}{\textrm{coker}}

\newcommand{\Id}{\mathbbm{1}}

\newcommand{\xps}{\Omega_p^s}

\newcommand{\lpj}{l_{j}^{+}}
\newcommand{\sgn}{\textrm{sgn}}
\newcommand{\bra}{\left\langle}
\newcommand{\ket}{\right\rangle}

\newtheorem{teo}{Theorem}
\newtheorem*{teoi}{Theorem}
\newtheorem*{propoi}{Proposition}

\newtheorem{teoa}{Theorem}[section]
\newtheorem{propoa}[teoa]{Proposition}
\newtheorem{lemma}[teo]{Lemma}

\newtheorem{coro}[teo]{Corollary}
\newtheorem{propo}[teo]{Proposition}
\theoremstyle{remark} 
\newtheorem{remark}[]{Remark}
\newtheorem{example}[]{Example}

\begin{document}

\maketitle

\begin{abstract}
We study the topology of the space $\Omega_p$ of \emph{admissible} paths between two points $e$ (the origin) and $p$ on a step-two Carnot group $G$:
$$\Omega_p=\{\gamma:I\to G\,|\, \textrm{$\gamma$ admissible, $\gamma(0)=e,\,\gamma(1)=p$}\}.$$ As it turns out, $\Omega_p$ is homotopy equivalent to an infinite dimensional sphere and in particular it is contractible. The \emph{energy} function: 
$$J:\Omega_p\to \mathbb{R}$$
is defined by $J(\gamma)= \frac{1}{2}\int_I\|\dot{\gamma}\|^2$; critical points of this function are sub-Riemannian geodesics between $e$ and $p$. We study the asymptotic of the number of geodesics and the topology of the sublevel sets:
$$\Omega_p^s=\{\gamma\in \Omega_p\, |\, J(\gamma)\leq s\} \quad\textrm{as $s\to \infty$}.$$
If $p$ \emph{is not} a vertical point in $G$, the number of geodesics joining $e$ and $p$ is bounded and the homology of $\Omega_p^s$ stabilizes to zero for $s$ large enough. 

A completely different behavior is experienced for the generic \emph{vertical} $p$. In this case we show that $J$ is a Morse-Bott function: geodesics appear in isolated \emph{families} (critical manifolds), indexed by their energy. Denoting by $l$ the corank of the horizontal distribution on $G$, we prove that:
$$\textrm{Card}\{\textrm{Critical manifolds with energy less than $s$}\}\leq O(s)^l.$$ Despite this evidence, Morse-Bott inequalities $b(\Omega_p^s)\leq O(s)^l$ are far from being sharp and we show that the following stronger estimate holds:
$$b(\Omega_p^s)\leq O(s)^{l-1}.$$
Thus each single Betti number $b_i(\Omega_p^s)$ ($i>0$) becomes eventually zero as $s\to \infty$, but the sum of all of them can possibly increase as fast as $O(s)^{l-1}.$ In the case $l=2$ we show that indeed 
$$b(\Omega_p^s)=\tau(p) s+o(s)\quad \textrm{($l=2$)}.$$
The leading order coefficient $\tau(p)$ can be \emph{analytically} computed using the structure constants of the Lie algebra of $G$.

Using a dilation procedure, reminiscent to the rescaling for Gromov-Hausdorff limits, we interpret these results as giving some local information on the geometry of $G$ (e.g. we derive for $l=2$ the rate of growth of the number of geodesics with bounded energy as $p$ approaches $e$ along a vertical direction).

\end{abstract}

\section{Introduction}
How many geodesics are there between two fixed points on a Riemannian manifold?\\
The classical way to give an answer to this question, due to M. Morse, is to consider the space $\Omega$ of \emph{all} curves (defined on the same interval) joining the two points, together with the function $J:\Omega \to \R$ that associates to each curve $\gamma$ the number $\frac{1}{2}\int_I\|\dot{\gamma}\|^2$ (the \emph{Energy}). Geodesics joining the two points are thus regarded as \emph{critical points} of the Energy, and the celebrated \emph{Morse inequalities} state that, in the generic case, the number of such geodesics with Energy bounded by $s$ is \emph{at least} the number of ``holes'' of the topological space $\{J\leq s\}.$
%$$b\left(\{J\leq c\}\right)\leq\textrm{Card}\left(\textrm{Crit}(J)\cap \{J\leq c\}\right)$$

In this Riemannian framework, if the two points $p_0$ and $p$ are close enough and we only allow curves with energy bounded by a small constant $c>0$, there is only one geodesic joining them and the space of curves connecting $p_0$ with $p$ with energy bounded by $c$ is contractible (bounding the energy ensures that the curves we consider are contained in a small neighborhood of $p_0$). %The reader should notice since now the different scaling we have used for the \emph{distance} and  the \emph{Energy}: curves with Energy bounded by  a small enough $s$, have length bounded $\sqrt{s}>s$ (i.e. the set $\{J\leq s\}$ contains curves that reach a neighborhood of the final point, whose diameter is small but with a different order than the distance of the two points). The importance of this two different scalings will be clear later.\todo{\footnote{\todo{Is this explanation clear?}}}

In fact Morse inequalities can be used both ways: geometers study the topology of $\{J\leq s\}$ to get information on geodesics, topologists use properties of geodesics in order to understand the path space $\Omega.$ This subject has been widely investigated, see for instance the classical works \cite{Milnor, Bott1, Bott2}

On a \emph{sub-Riemannian} manifold $M$ the situation gets more complicated. Being the velocities of curves constrained to the subbundle $\Delta\subset TM$,  the \emph{admissible}-path space is no longer the same as above and new phenomena can occurr. 

Typically, the space of admissible curves with energy less than $s$ joining $p_0$ and $p$ has finite dimensional homology but its homological
dimension and total Betti number may grow to infinity as $p\to p_0$. We are interested in the asymptotic behavior of these quantities and 
related properties of the space of sub-Riemannian geodesics joining $p_0$ and $p$.

In this paper, we focus on the case of a step-two Carnot group where the situation is well controlled. A step-one Carnot group is just a Euclidean space.
The general step-$k$ Carnot group $G$ is a simply connected step-$k$ nilpotent Lie group equipped with a``dilation''. Such a dilation is a one-paramentic group
of automorphisms that generalizes homotheties of the Euclidean space. A typical example is the group of  $n\times n$ lower triangular matrices with units
on the main diagonal, zeros over the diagonal and dilation
 $\delta_t:\{a_{ij}\}_{i,j=1}^n\mapsto  \{t^{i-j}a_{ij}\}_{i,j=1}^n, t>0$ (Example 1 below is the case $n=3$).

First order elements with respect to the dilation form a vector subspace of the Lie algebra of the group and thus define a left-invariant
vector distribution $\Delta\subset TG$. It is assumed that the first order elements generate the whole Lie algebra (see \cite{Pansu}  and \cite{AgrachevBarilariBoscain} for details
and for the explanation of the role of Carnot groups in sub-Riemannian geometry).

There are no one or two dimensional nonabelian Carnot groups. The only three dimensional nonabelian Carnot group is the Heisenberg group: we discuss its geometry as it will be an enlightening, leading example for the rest of the paper.
\begin{example}[The Heisenberg group]The \emph{Heisenberg group} $G$ is the smooth manifold $\R^3$ with the usual coordinates $(x,y,z)$ and the distribution:
\begin{equation}\label{deltah}\Delta=\textrm{span}\left\{\frac{\partial}{\partial x}-\frac{y}{2}\frac{\partial}{\partial z}, \frac{\partial}{\partial y}+\frac{x}{2}\frac{\partial}{\partial z}\right\}.\end{equation}
This means that one is allowed to move only along curves $\gamma:I=[0, 2\pi]\to G$ whose velocity pointwise belongs to $\Delta$ (these curves are called admissible and they are all defined on the same interval; we pick this interval to be $[0, 2\pi]$ because it will simplify notations later). In fact, for a technical reason, one considers curves whose derivative is defined a.e. and is square integrable, ending up with absolutely continuous curves. The sub-Riemannian structure is given by declaring the above vector fields to be an orthonormal basis.

We assume one of the two points is the origin $e=(0,0,0)$ (the structure is invariant by translation) and call the other $p$; thus we consider:
$$\Omega_p=\{\textrm{admissible curves starting at $e$ and ending at $p$}\}.$$ 
As we will show in a while, there is a natural inclusion:
 $$\Omega_p\hookrightarrow L^2(I, \R^2),$$ and we endow it with the induced topology. Recall that $\Delta_p=\textrm{span}\{X(p), Y(p)\}$, where $X,Y$ are the two vector fields given in (\ref{deltah}); in particular for every admissible curve $\gamma$ we can write $\dot \gamma =u_xX+u_yY$, where $u_x, u_y\in L^2(I).$ This correspondence is one-to-one: because of Cauchy's theorem given $u\in L^2(I, \R^2)$, there is only one solution $\gamma_u$ to the system $\dot \gamma=u, \gamma(0)=e$. Thus we can identify the space of \emph{all} admissible curves starting at $e$ with $L^2(I, \R^2)$ and $\Omega_p$ coincides with the set of curves such that $\gamma_u$ is defined for $t=2\pi$ and $\gamma_u(2\pi)=p$ (thus $u$ represents the ``coordinates'' of $\gamma_u$ and it is usually called its \emph{control}).
 
The energy of an admissible curve is by definition (one half) the square of the $L^2$ norm of its control:
$$J(\gamma_u)=\frac{1}{2}\int_{0}^{2\pi}\|\dot{\gamma}_u(t)\|^2dt=\frac{\|u\|^2}{2}.$$
Critical points of this function restricted to the various $\Omega_p$ (i.e. geodesics between $e$ and $p$) are curves whose projection on the $(x,y)$-plane is an arc of a circle (possibly with infinite radius, i.e. an interval on a straight line); the signed area swept out on the circle by this projection equals the $z$-coordinate of the final point.

We see that if $p$ belongs to the $(x,y)$-plane there is only one geodesic joining it with the origin (this is precisely the segment trough $e$ and $p$), as if we were in a riemannian manifold; if $p$ has both nonzero components in the $(x,y)$ plane and the $z$ axis, the number of geodesics is finite; finally, if $p$ belongs to the $z$-axis there are infinitely many such geodesics (moreover given one we obtain infinitely many others by composing it with a rotation in $\R^3$ along the axis).

The behavior of geodesics ending at a point $(0,0,z)$ (slightly abusing of notation we call this point $z$)  is really new if compared to the classical Riemannian situation and is worth discussing it in more details.

Given a natural number $n\in \mathbb{N}_0$ let us pick a circle in the $(x,y)$-plane of area $|z|/n$ passing through the origin; starting at $e$ and moving clockwise around the circle $n$ times, we obtain  the projection of a geodesic ending at $z$ and whose energy is $2\pi n|z|$; rotating this curve along the $z$-axis of an angle $\theta\in SO(2)\simeq S^1$ gives another such geodesic (thus we have an $S^1$ of geodesics with the same energy and endpoints). Hence geodesics  ending at $z$ (i.e. critical points of $J$ on $\Omega_{z}$) arrange into \emph{infinitely} many \emph{families}, indexed by the values of their energy $\{2\pi n |z|\}_{n\in \mathbb{N}}$, each family being homeomorphic to an $S^1$ in $\Omega_{z}$.

It is not possible, therefore, to study the topology of $\Omega_{z}$ using $f=J|_{\Omega_{z}}$ as a Morse function: simply because its critical points are not isolated. Despite this, as we have noted these critical points appear in nice families and in fact these families are \emph{nondegenerate} in a sense that will be specified later, allowing to use the Morse theory machinery (this generalization is usually called \emph{Morse-Bott} theory).

Before proceeding, let us show how to determine the topology of $\Omega_{z}$, by giving its equations in $L^2(I, \R^2).$ The map that associates  to each $u$ the final point $\gamma_u(2\pi)$ is called the \emph{End-point map}; it is a smooth map and will be denoted by $F:L^2(I, \R^2)\to \R^3.$ A careful analysis of the definition of $\gamma_u$ provides the explicit expression for $F=(F_1, F_2, F_3)$, see (\ref{endpointexplicit}). In particular we can write the equations for $\Omega_{z}$ as:
$$\Omega_{z}=\{F_1(u)=0, F_2(u)=0, F_3(u)=z\}.$$
The first two equations define a linear subspace $H=\{F_1=F_2=0\}$ and restricted to this space $F_3$ is a quadratic form $q=F_3|_H$ with \emph{infinitely many} positive and \emph{infinitely many} negative eigenvalues. In particular $\Omega_{z}$ is homotopy equivalent to an infinite dimensional sphere:
$$\Omega_{z}\sim S^\infty$$
and is contractible.
What kind of information can therefore a Morse theoretical study of $\Omega_{z}$ give? As a starting point, the investigation of critical points of $f=J|_{\Omega_{z}}$ gives information on the structure of geodesics; but we will see there is also more. 

Let us find all the critical manifolds of $f$ with energy bounded by $s$. Given a circle trough the origin of area $|z|/n$, it is the projection of a geodesic between $e$ and $z$ with energy bounded by $s$ if and only if $n\leq\frac{s}{2\pi |z|}.$ Rotating each such geodesic we get a whole critical manifold homeomorphic to $S^1$. In particular:
\begin{equation}\label{critH}\{\textrm{geodesics between $e$ and $z$ with $J\leq s$}\}=\underbrace{S^1\cup\cdots \cup S^1}_{\textrm{$\left\lfloor\frac{s}{2\pi |z|}\right\rfloor$ copies}}\end{equation}

In this context the ``amount'' of geodesics can still be used as a measure of the topology of $\Omega_{z}\cap \{J\leq s\}$, in a form that generalizes Morse inequalities. Specifically, whenever the critical points of $f$ arrange into smooth manifolds (with some nondegeneracy conditions), then the sum of the Betti numbers (i.e. the number of ``holes'') of $\{f\leq s\}$ is bounded by the sum of the Betti numbers of all the critical manifolds with $f\leq s$. The inequalities we get are called \emph{Morse-Bott inequalities}, and in our case they give\footnote{For the rest of the paper, $b(X)$ will indicate the \emph{sum} of the $\mathbb{Z}_2$-Betti numbers of $X$, i.e. $b(X)=\sum_i \textrm{rank}H_i(X; \mathbb{Z}_2)$; this sum might a priori be infinite, depending on the topological space $X$, but in the case of our interest it will always be finite.}
\begin{equation}\label{M-Bh}b(\Omega_{z}\cap\{J\leq s\})\leq b(S^1\cup\cdots \cup S^1)\sim \frac{s}{2\pi |z|}.\end{equation}

Thus, as $s\to \infty$ there are more and more critical manifolds but, surprisingly enough, the above inequality is far from being sharp. One can in fact show that $\Omega_z\cap\{J\leq s\}$ has the homotopy type of a finite-dimensional sphere $S^{d_s}$ in $\Omega_{z}$:
$$\Omega_{z}\cap\{J\leq s\}\sim S^{d_s}\quad \textrm{and}\quad b(\Omega_{z}\cap\{J\leq s\})=2.$$
In fact the dimension $d_s$ of this sphere is increasing at a constant rate linear in $s$ but the sum of its Betti numbers is constant: this shows that in (\ref{M-Bh}) using Morse-Bott inequalities we were overcounting.

Hence the picture is the same as of a family of finite-dimensional spheres ``approaching'' an infinite dimensional one; the dimension of the spheres increases with $s$ and the topology eventually vanishes (i.e. $\lim_{s\to \infty} b_i(\Omega_{z}\cap\{J\leq s\})=0$ for every $i\in \mathbb{N}_0$). Nonetheless there is something that ``persists'' at infinity: it is the \emph{sum} of all Betti numbers of $\Omega_{z}\cap \{J\leq s\}.$ What is the meaning of this invariant number?
\begin{figure}

% Generated with LaTeXDraw 2.0.8
% Wed Nov 20 17:07:36 EST 2013
% \usepackage[usenames,dvipsnames]{pstricks}
% \usepackage{epsfig}
% \usepackage{pst-grad} % For gradients
% \usepackage{pst-plot} % For axes
\scalebox{0.6} % Change this value to rescale the drawing.
{
\begin{pspicture}(0,-1.86)(9.821015,1.86)
\psline[linewidth=0.04cm,arrowsize=0.05291667cm 2.0,arrowlength=1.4,arrowinset=0.4]{<-}(0.6610156,1.84)(0.6810156,-1.84)
\psline[linewidth=0.04cm,arrowsize=0.05291667cm 2.0,arrowlength=1.4,arrowinset=0.4]{->}(0.70101565,-1.82)(9.801016,-1.82)
\psline[linewidth=0.04cm](0.6610156,-0.4)(2.0810156,-0.4)
\psline[linewidth=0.04cm](4.881016,-0.4)(6.2610154,-0.4)
\psline[linewidth=0.04cm,linestyle=dashed,dash=0.16cm 0.16cm](2.0810156,-0.4)(2.0810156,-1.76)
\psline[linewidth=0.04cm,linestyle=dashed,dash=0.16cm 0.16cm](4.881016,-0.4)(4.881016,-1.76)
\psline[linewidth=0.04cm,linestyle=dashed,dash=0.16cm 0.16cm](6.2610154,-0.4)(6.2610154,-1.78)
\psline[linewidth=0.04cm,arrowsize=0.05291667cm 2.0,arrowlength=1.4,arrowinset=0.4]{->}(6.6810155,-1.06)(7.861016,-1.06)
\usefont{T1}{ptm}{m}{n}
\rput(2.3924707,1.425){$b_i(\Omega_p\cap\{J\leq s\})$}
\usefont{T1}{ptm}{m}{n}
\rput(9.292471,-1.515){$i$}
\usefont{T1}{ptm}{m}{n}
\rput(7.2224708,-0.795){$s\to \infty$}
\end{pspicture} 
}
\caption{The homology of $\Omega_p\cap\{J\leq s\}$ in the Heisenberg group: the top dimensional homology moves to infinity as $s$ grows and eventually ``disappears''.}
\end{figure}
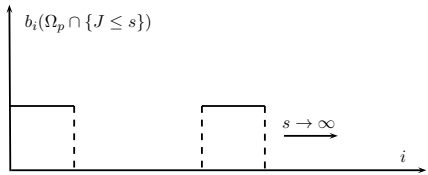

A possible way to answer this question is to consider a local problem  near $e$: assuming $d(e, p_\epsilon)=\epsilon$ how many geodesics are there between $e$ and $p_\epsilon$ with \emph{bounded} energy (say by a constant $c$)? As we have seen, if $p_\epsilon$ is not a verical point, the answer is finitely many (depending on the upper bound on the energy). On the other hand, what happens if $p_\epsilon$ approaches $e$ along the non-riemannian $z$-axis? To answer this question let us take the point:
$$p_{\epsilon}=\left(0, 0, \frac{\epsilon^2}{4\pi} \right) \quad \textrm{such that}\quad d(e, p_\epsilon)=\epsilon.$$ 
Arguing as in (\ref{critH}), we see that the critical manifolds of $J|_{\Omega_{p_\epsilon}}$ with energy bounded by $c$ are $\lfloor\frac{2c}{\epsilon^2}\rfloor$ copies of $S^1$: as we let $p_\epsilon$ closer and closer to $e$ the amount of geodesics keeps on increasing. The equations for $\Omega_{p_\epsilon}\cap \{J\leq c\}$ are:
$$\Omega_{p_\epsilon}\cap \{J\leq c\}=\left\{F_1(u)=F_2(u)=0, \, q(u)=\frac{\epsilon^2}{2\pi},\, \|u\|^2\leq c\right\},$$
and the transformation $u\mapsto u/(\epsilon\sqrt{2})$ gives a homeomorphism between $\Omega_{p_\epsilon}\cap \{J\leq c\}$ and $\Omega_{(0,0,\frac{1}{4\pi})}\cap \{J\leq \frac{2c}{\epsilon^2}\}$. Hence the above discussion applies with $z=(0,0, \frac{1}{4\pi})$ and $s=\frac{2c}{\epsilon^2}$, giving:
$$ \lim_{\epsilon\to 0}b(\Omega_{p_\epsilon}\cap\{J\leq c\})=2.$$
Thus, this persisting number can be interpreted as a local invariant of the Heisenberg group ``in the direction'' of $z$, as we will make more clear later.
 \end{example}
 This example shows that the sub-Riemannian case can be much richer than the riemannian, even at the level of the topology of the set of curves in $\Omega_p$ when $p$ is close to the initial point and the energy is bounded. Let us continue along these lines and ask if we can find,  in the general case, some asymptotic behavior when the distance of the final point become smaller and the energy stays bounded. To this end we need a procedure for taking limit.

 The ``time-one'' picture is the following: we have two points $e$ and $p$ at distance $s$ on a sub-Riemannian manifold $M$, and the space $\Omega_{p}\cap \{J\leq c\}$ of curves reaching a neighborhood of $p$. Now, on a neighborhood $U$ of $e$ there is a well defined family of non-isotropic ``dilations'' centered at $e$ (we assume $p$ is close enough to be in $U$):
$$ \delta_\lambda:U\to U, \quad \lambda>0$$ with the property that $d(e, \delta_\lambda(p))=\lambda d(e, p)$
(for example, in the Heisenberg group this family operates as: $\delta_\lambda(x,y,z)=(\lambda x, \lambda y, \lambda^2 z)).$ 
We are thus led to consider the family of spaces:
\begin{equation}\label{eq:scaling}\Omega_{\delta_{\epsilon}(p)}\cap \{J\leq c\}, \quad \epsilon\to 0\end{equation}
(this is exactly what we did in the Heisenberg case).
This limiting procedure is reminiscent of the one used to construct the so called \emph{sub-Riemannian} tangent space at $e$: as a vector space it is simply $T_eM$, but it comes endowed with much more structure that makes of it a Carnot group. 
%The fact that the study we perform uses this blow-up procedure, catching some infinitesimal behavior near $e$, suggests that one should start by looking at what happens on Carnot groups.

\emph{Step-two} Carnot groups are ``tangent spaces'' to a generic sub-Riemannian manifold (the step-two condition needs some assumptions on the corank of $\Delta$ to be generic).\\
Thus, on one hand our results should be interpreted as an ``infinitesimal'' version of a more general local theory (that the authors plan to discuss in a subsequent paper); on the other hand, Carnot groups are sub-Riemannian manifolds by themselves (plenty of literature has been devoted to them) and we believe this study is of autonomous interest.

The questions we will be interested in are: 
\begin{itemize}
\item[(i)] What is the structure of geodesics connecting two points $e$ and $p$ (i.e. critical points of the energy on $\Omega_p$) in a step-two Carnot group? 
\item[(ii)] What is the growth rate of the ``number'' of critical points at infinity? How does the topology of $\Omega_p\cap\{J\leq s\}$ behave as $s\to \infty$? 
\item[(iii)] Is there some invariant that can be captured for $\Omega_{\delta_{\epsilon}(p)}\cap \{J\leq c\}$ when we perform the limit $\epsilon\to 0$?
\end{itemize}

To start with we need to make our definition of a step-two Carnot group more precise. As a differentiable manifold it will simply be $\R^{d+l}$, together with the distribution $\Delta\subset T\R^{d+l}$ generated by the span of the vector fields defined in coordinates $(x,y)$, $x\in \R^d$ and $y\in \R^l$, by:
$$E_i(x,y)=\frac{\partial}{\partial x_i}(x,y)-\frac{1}{2}\sum_{k=1}^l \sum_{j=1}^{d}a_{ij}^kx_j\frac{\partial}{\partial y_k}(x,y), \quad i=1, \ldots, d.$$
The numbers $\{a_{ij}^k\}_{i,j,k}$ come from skew symmetric matrices $A_k=(a_{ij}^k)\in \sod, \,k=1,\ldots, l;$ these matrices span a $l$-dimensional vector space:
$$W=\textrm{span}\{A_1, \ldots, A_l\}\subset \sod.$$

The sub-Riemannian structure is given by declaring the $E_i$ to be an orthonormal basis of $\Delta$. Two different choices of bases for $W$ will produce \emph{isomorphic} Carnot groups; for this reason we will call $W\subset \sod$ the \emph{Carnot group structure}.\\
Admissible curves will be absolutely continuous curve, whose velocities (a.e. defined) are in $\Delta$. As we did for the Heisenberg group, $\Omega_p$ will denote the set of admissible curves joining the origin $e$ with $p$ and there will be a natural inclusion:
$$\Omega_p\hookrightarrow L^2(I, \R^d).$$
In this way the \emph{Energy} of a curve $\gamma_u$ with velocity $u$ will simply be $J(\gamma_u)=\frac{\|u\|^2}{2}.$

With this notation we have the homeomorphism:
\begin{equation}\label{eq:scaling2}
\Omega_{\delta_\epsilon(p)}\cap\{J\leq c\}=\delta_\epsilon(\Omega_p\cap\{J\leq  c/\epsilon^2\}).
\end{equation}
If $p$ is not a vertical point (and $d>l$), then the number of geodesics joining it with the origin is finite; in other words, the energy of geodesics connecting the origin with $p$  is uniformly bounded
and the  homology of the space $\Omega_{\delta_\epsilon(p)}\cap\{J\leq c\}$ stabilizes as $\epsilon\to 0$ (this is not true for a vertical $p$). 
\begin{propoi}If $p$ is not a vertical point and $d>l$, the cohomology of $\Omega_{\delta_\epsilon(p)}\cap\{J\leq c\}$ stabilizes for $\epsilon$ small enough.

\end{propoi}
Because of this, we are going
to study asymptotics of the homological dimension and total Betti number of $\Omega_{\delta_\epsilon(p)}\cap\{J\leq c\}$ as $\epsilon\to 0$, under the assumption:
\begin{equation}\label{vertical}p\textrm{ is a vertical point}.\end{equation}
Using the $(x,y)$ coordinates as above the ``vertical'' direction is identified with the $y$-space and we denote such space by $\Delta^2\simeq \R^l$; in this way the condition (\ref{vertical}) can be rewritten as $p\in \dd.$

All the above questions can be addressed by studying the Morse-Bott theory of $J|_{\Omega_p}$, by taking appropriate rescalings.
In fact, assuming (\ref{vertical}), we have $\delta_\epsilon(p)=\epsilon^2p$ and the homogeneity of the end-point map on the vertical directions gives the homeomorphism \eqref{eq:scaling2}. In particular the limit \eqref{eq:scaling} can be studied by simply fixing $p$ and letting the energy grow (at a rate $\epsilon^{-2}$).

Our first result concerns the topology of the space $\Omega_p$: first we show that its homotopy type does not depend on $p$: in particular $\Omega_p\sim \Omega_e$ and since the latter is defined by homogeneous equations, then these spaces are contractible. The homotopy equivalence is realized by a simple modification of the construction for ordinary loop spaces. Notice that $\Omega_e$ is a singular space (it is homeomorphic to a cone based at the constant curve $\gamma \equiv e$); despite this, for the generic choice of the Carnot group structure and the final point in the vertical direction, $\Omega_p$ is a smooth (infinite-dimensional) manifold.
\begin{teoi}[The topology of $\Omega_p$]The topological space $\Omega_p$ is homotopy equivalent to $\Omega_e$ and is therefore contractible; moreover for the generic $p\in \dd$ it is a Hilbert manifold.\end{teoi}
In particular, analysis can be performed over $\Omega_p$ and critical points of $J|_{\Omega_p}$
 (which is a smooth function) can be found using the Lagrange multipliers rule. As for Heisenberg, we can write down explicit equations: here the crucial point is that the End-point map sends a curve to $\Delta^2$ if and only if its control has zero mean (i.e. $0=\int_I u\in \R^d$) and the restriction of this map to the space of such controls is \emph{quadratic}.

A careful investigation of the critical points of $f$ shows that they appear in families, i.e. the arrange into critical manifolds. These manifolds are tori $S^1\times\cdots \times S^1$ and are indexed by their Lagrange multipliers. The whole structure can be recovered by $W\subset \sod$ only, as follows.

Consider the infinite union of algebraic sets:
$$\Lambda=\bigcup_{n\in \mathbb{N}}\Lambda_n,\quad \Lambda_n=\{A\in W\, |\, \det(A-in\mathbbm{1})=0\}.$$
Each set $\Lambda_n$ is an hypersurface in $W$: generically it is smooth, but there can be points where $iA$ has multiple integer eigenvalues and the corresponding hypersurface is singular. In fact all these hypersurfaces (except $\Lambda_0$) can be obtained by 
dilations of $\Lambda_1$ (i.e. $\Lambda_n=n\Lambda_1$). Hence $\Lambda$ looks like an ``infinite net'' with two kinds of singularities (they might appear at the same time): one kind comes from the singularities of each $\Lambda_n$ and the other from the intersections $\Lambda_{n_1}\cap\cdots \cap \Lambda_{n_\nu}$, for $\nu\leq l.$

As we will see $\Lambda$ represents the set of all possible Lagrange multipliers, for all possible final points, hence to each point of $\Lambda$ there corresponds a family of geodesics. Notice that, once the basis $\{A_1, \ldots, A_l\}$ of $W$ and coordinates are fixed, the correspondence: 
$$\omega=(\omega_1, \ldots, \omega_l)\mapsto \omega A=\omega_1A_1+\cdots +\omega_l A_l$$
defines a linear isomorphism $(\Delta^2)^*\simeq W$.

The next theorem gives a detailed answer to question (i) above on the structure of geodesics.
\begin{teoi}[The structure of geodesics]Let $u$ be the control associated to a geodesic from $e$ to a point in $\dd$ with Lagrange multiplier $\omega$. Then:
\begin{itemize}
\item[1.] $u(t)=e^{-t \omega A}u_0$ with $u_0=e^{-2\pi \omega A}u_0.$
\item[2.]The final point $q(u)=(q_1(u), \ldots, q_l(u))\in \dd$ of such geodesic is given by:
 $$q_i(u)=\left\langle u_0,L_i(\omega)u_0\right\rangle,\quad L_i(\omega)=\frac{1}{2} \int_{0}^{2\pi}\left(\int_{0}^t e^{\tau \omega A}d\tau\right)A_ie^{-t\omega A}dt.$$
 \item[3.] The energy of such geodesic is given by $J(u)=\omega (F(u))$. 
 \end{itemize}
In addition, for the generic choice of $W\subset \sod$ and of $p\in \dd$:
\begin{itemize}
\item[4.] The set $\Lambda(p)\subset \Lambda$ of Lagrange multipliers for geodesics whose final point is $p$ is a discrete set. 
\item[5.] If $\omega\in \Lambda(p)$ then the integer eigenvalues of the matrix $i\omega A$ are simple and there exist $n_1, \ldots, n_\nu\in \mathbb{N}_0$ such that $\omega$ belongs to $\Lambda_{n_1}\cap\cdots\cap \Lambda_{n_\nu}$ (the number $\nu\leq l$ will be called the number of \emph{resonances} of $\omega$).
\item[6.] If $\omega$ has $\nu$ resonances, the union of all the geodesics from $e$ to $p$ with Lagrange multiplier $\omega$ is a smooth manifold $C_\omega$ homeomorphic to:$$C_\omega \simeq \underbrace{S^1\times \cdots\times S^1}_{\nu \textrm{ times}}.$$
\end{itemize} \end{teoi}

Before proceeding any further, it is interesting to discuss one more example.
\begin{example}[Corank-two distributions]Let us consider the case of $\R^{d+2}$ with Carnot group structure $W=\textrm{span}\{A_1, A_2\}\subset \sod$ having the above genericity property.\\
Let us consider the curve $\Lambda_1\subset W$ first: since the set of matrices with double eigenvalues has codimension three in $\sod$, $\Lambda_1$ is a smooth curve (with possibly many components). In this case the set $\{\omega \in \Lambda_{1}\, |\, \omega(p)\geq0,\,p\in (T_{\omega}\Lambda_1)^{\perp}\}$
consists of Lagrange multipliers whose critical manifold $C_\omega$ is homeomorphic to $S^1;$ all the other Lagrange multipliers whose associated critical manifold is a circle are obtained multiplying these covectors by a positive natural number:
$$\{\omega \in \Lambda(p)\,|\, C_\omega\simeq S^1\}=\bigcup_{n\in \mathbb{N}_0}\{\omega \in \Lambda_{n}\, |\, \omega(p)\geq0,\,p\in (T_{\omega}\Lambda_1)^{\perp}\}.$$
\begin{center}
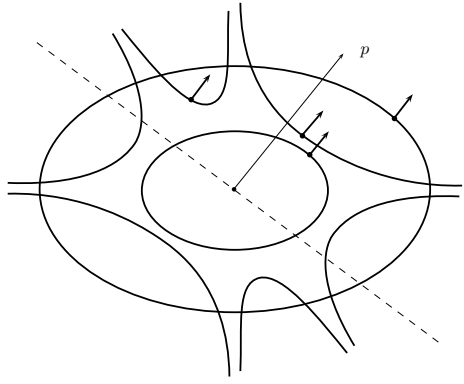
\begin{figure}
% Generated with LaTeXDraw 2.0.8
% Mon Aug 26 20:50:31 EDT 2013
% \usepackage[usenames,dvipsnames]{pstricks}
% \usepackage{epsfig}
% \usepackage{pst-grad} % For gradients
% \usepackage{pst-plot} % For axes
\scalebox{0.6} % Change this value to rescale the drawing.
{
\begin{pspicture}(0,-4.15)(9.98,4.15)
\psellipse[linewidth=0.04,dimen=outer](4.98,-0.02)(2.06,1.33)
\psellipse[linewidth=0.04,dimen=outer](4.98,0.01)(4.3,2.74)
\psbezier[linewidth=0.04](2.5,3.55)(3.08,2.83)(3.9,1.77)(4.42,1.89)(4.94,2.01)(4.8,2.97)(4.84,3.95)
\psbezier[linewidth=0.04](5.1,4.13)(5.12,2.09)(6.28,1.31)(6.7,1.03)(7.12,0.75)(8.28,0.05)(9.96,0.09)
\psbezier[linewidth=0.04](4.86,-4.13)(4.84,-2.09)(3.68,-1.31)(3.26,-1.03)(2.84,-0.75)(1.68,-0.05)(0.0,-0.09)
\psbezier[linewidth=0.04](2.3,3.45)(3.12,2.09)(2.983359,1.3054222)(2.78,1.01)(2.5766408,0.7145778)(1.96,0.07)(0.0,0.15)
\psbezier[linewidth=0.04](7.6,-3.45)(6.78,-2.09)(6.9166408,-1.3054222)(7.12,-1.01)(7.323359,-0.7145778)(7.94,-0.07)(9.9,-0.15)
\psdots[dotsize=0.1](4.96,0.01)
\psline[linewidth=0.02cm,arrowsize=0.05291667cm 2.0,arrowlength=1.4,arrowinset=0.4]{->}(4.94,-0.01)(7.36,3.01)
\usefont{T1}{ptm}{m}{n}
\rput(7.831455,3.055){$p$}
\psline[linewidth=0.02cm,linestyle=dashed,dash=0.16cm 0.16cm](0.64,3.25)(9.44,-3.37)
\psdots[dotsize=0.12](6.46,1.19)
\psdots[dotsize=0.12](6.62,0.77)
\psdots[dotsize=0.12](8.48,1.57)
\psdots[dotsize=0.12](4.02,1.99)
\psbezier[linewidth=0.04](7.44,-3.61)(6.86,-2.89)(6.04,-1.83)(5.52,-1.95)(5.0,-2.07)(5.14,-3.03)(5.1,-4.01)
\psline[linewidth=0.04cm,arrowsize=0.05291667cm 2.0,arrowlength=1.4,arrowinset=0.4]{->}(8.48,1.59)(8.88,2.11)
\psline[linewidth=0.04cm,arrowsize=0.05291667cm 2.0,arrowlength=1.4,arrowinset=0.4]{->}(6.62,0.77)(7.04,1.29)
\psline[linewidth=0.04cm,arrowsize=0.05291667cm 2.0,arrowlength=1.4,arrowinset=0.4]{->}(6.46,1.21)(6.92,1.75)
\psline[linewidth=0.04cm,arrowsize=0.05291667cm 2.0,arrowlength=1.4,arrowinset=0.4]{->}(4.02,1.99)(4.44,2.55)
\end{pspicture} 
}
\caption{The set $\bigcup_{n\in \mathbb{N}_0}\{\omega \in \Lambda_{n}\, |\, \omega(p)\geq0,\,p\in (T_{\omega}\Lambda_1)^{\perp}\}$}
\end{figure}
\end{center}
Let us now consider two natural numbers $n_1, n_2$ such that $\Lambda_{n_1}\cap \Lambda_{n_2}\neq 0$ and let $\omega$ be a point in this intersection such that $\omega(p)\geq 0.$ Let $E(\omega)=\{v\in \R^d\, |\, e^{2\pi\omega A}v=v\}$ and consider the quadratic forms $q_i|_{\omega}:E(\omega)\to \R$ defined by:
$$q_i|_{\omega}:v\mapsto \langle v, L_i(\omega) v\rangle,\quad i=1,2.$$
Let  $q|_{\omega}:E(\omega)\to \R^2$ be the quadratic \emph{map} whose components are the above $q_i|_{\omega}$. The set of Lagrange multipliers whose associated critical manifold is $S^1\times S^1$ coincides with:
$$\{\omega\in \Lambda(p)\, |\, C_\omega=S^1\times S^1\}=\bigcup_{n_1, n_2\in \mathbb{N}_0}\{\omega \in \Lambda_{n_1}\cap \Lambda_{n_2}\, |\, p\in \textrm{im}(q|_\omega)\}.$$
\begin{center}
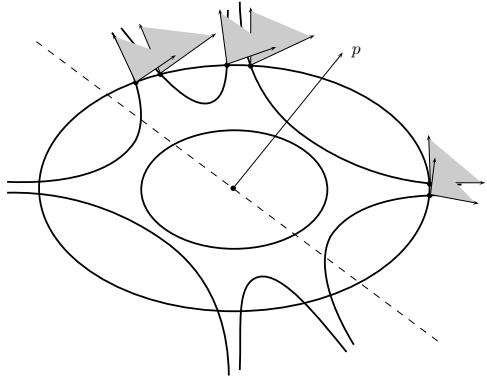
\begin{figure}
% Generated with LaTeXDraw 2.0.8
% Mon Aug 26 20:53:14 EDT 2013
% \usepackage[usenames,dvipsnames]{pstricks}
% \usepackage{epsfig}
% \usepackage{pst-grad} % For gradients
% \usepackage{pst-plot} % For axes
\scalebox{0.6} % Change this value to rescale the drawing.
{
\begin{pspicture}(0,-4.15)(10.49,4.15)
\definecolor{color4290b}{rgb}{0.8,0.8,0.8}
\psellipse[linewidth=0.04,dimen=outer](4.98,-0.02)(2.06,1.33)
\psellipse[linewidth=0.04,dimen=outer](4.98,0.01)(4.3,2.74)
\psbezier[linewidth=0.04](2.5,3.55)(3.08,2.83)(3.9,1.77)(4.42,1.89)(4.94,2.01)(4.8,2.97)(4.84,3.95)
\psbezier[linewidth=0.04](7.44,-3.61)(6.86,-2.89)(6.04,-1.83)(5.52,-1.95)(5.0,-2.07)(5.14,-3.03)(5.1,-4.01)
\psbezier[linewidth=0.04](5.1,4.13)(5.12,2.09)(6.28,1.31)(6.7,1.03)(7.12,0.75)(8.28,0.05)(9.96,0.09)
\psbezier[linewidth=0.04](4.86,-4.13)(4.84,-2.09)(3.68,-1.31)(3.26,-1.03)(2.84,-0.75)(1.68,-0.05)(0.0,-0.09)
\psbezier[linewidth=0.04](2.3,3.45)(3.12,2.09)(2.983359,1.3054222)(2.78,1.01)(2.5766408,0.7145778)(1.96,0.07)(0.0,0.15)
\psbezier[linewidth=0.04](7.6,-3.45)(6.78,-2.09)(6.9166408,-1.3054222)(7.12,-1.01)(7.323359,-0.7145778)(7.94,-0.07)(9.9,-0.15)
\psdots[dotsize=0.12](4.96,0.01)
\psline[linewidth=0.02cm,arrowsize=0.05291667cm 2.0,arrowlength=1.4,arrowinset=0.4]{->}(4.94,-0.01)(7.36,3.01)
\usefont{T1}{ptm}{m}{n}
\rput(7.651455,3.015){$p$}
\psdots[dotsize=0.12](5.34,2.71)
\psdots[dotsize=0.12](4.82,2.73)
\psdots[dotsize=0.12](3.36,2.53)
\psdots[dotsize=0.12](2.82,2.35)
\psdots[dotsize=0.12](9.26,0.11)
\psdots[dotsize=0.12](9.26,-0.15)
\psline[linewidth=0.02cm,linestyle=dashed,dash=0.16cm 0.16cm](0.64,3.25)(9.44,-3.37)
\psline[linewidth=0.02,arrowsize=0.05291667cm 2.0,arrowlength=1.4,arrowinset=0.4,fillstyle=solid,fillcolor=color4290b]{<->}(5.34,4.01)(5.34,2.73)(6.74,3.39)
\psline[linewidth=0.02,arrowsize=0.05291667cm 2.0,arrowlength=1.4,arrowinset=0.4,fillstyle=solid,fillcolor=color4290b]{<->}(9.24,1.19)(9.2775,0.13)(10.48,0.13)
\psline[linewidth=0.02,arrowsize=0.05291667cm 2.0,arrowlength=1.4,arrowinset=0.4,fillstyle=solid,fillcolor=color4290b]{<->}(9.38,0.69)(9.28,-0.13)(10.34,-0.33)
\psline[linewidth=0.02,arrowsize=0.05291667cm 2.0,arrowlength=1.4,arrowinset=0.4,fillstyle=solid,fillcolor=color4290b]{<->}(2.9,3.65)(3.36,2.55)(4.58,3.43)
\psline[linewidth=0.02,arrowsize=0.05291667cm 2.0,arrowlength=1.4,arrowinset=0.4,fillstyle=solid,fillcolor=color4290b]{<->}(4.66,3.97)(4.82,2.75)(5.9,3.13)
\psline[linewidth=0.02,arrowsize=0.05291667cm 2.0,arrowlength=1.4,arrowinset=0.4,fillstyle=solid,fillcolor=color4290b]{<->}(2.32,3.43)(2.82,2.35)(3.7,2.95)
\end{pspicture} 
}
\caption{The set $\bigcup_{n_1, n_2\in \mathbb{N}_0}\{\omega \in \Lambda_{n_1}\cap \Lambda_{n_2}\, |\, p\in \textrm{im}(q|_\omega)\}.$}
\end{figure}
\end{center}\end{example}
Let us move now to question number (ii) above, on the growth rate of the number of critical manifolds of $J|_{\Omega_p}$ as the energy goes to infinity. As we have seen for the Heisenbger group in (\ref{critH}), once $p$ is fixed the number of critical manifolds with energy bounded by $s$ increases linearly in $s$:
$$\textrm{Card}\{\textrm{critical manifolds of $J|_{\Omega_p}$ such that $J\leq s$}\}\sim O(s)\quad\textrm{(Heisenberg)}.$$
The previous example shows that in the corank two case the critical manifolds of $J|_{\Omega_p}$ can be indexed by points on an infinite two-dimensional net and, naively, it is reasonable to guess that they increase quadratically in $s$. Indeed it is a general fact that this number increases at most as a polynomial in $s$ of degree $l=\textrm{corank}(\Delta)$ (see Figure \ref{fig:lattice}).
\begin{teoi}[The growth rate of the number of critical manifolds]For the generic choice of $W\subset \sod$ and of $p\in \dd$ we have:
\begin{equation}\label{growcrit}\textrm{Card}\{\textrm{critical manifolds of $J|_{\Omega_p}$ such that $J\leq s$}\}\leq O(s)^l,\quad l=\dim(\dd).\end{equation}
\end{teoi}

Using the inequality in (\ref{growcrit}) and the fact that the topology of the critical manifold is uniformly bounded by a constant depending on $d$ and $l$ only, one is tempted to answer the second part of question (ii), concerning the topology of $\Omega_p\cap\{J\leq s\}$, using Morse-Bott inequalities. What we would get with this strategy is, for a corank $l$ distribution:
\begin{equation}\label{MBI}b(\Omega_p\cap\{J\leq s\})\leq O(s)^l.\end{equation}
Surprisingly enough, this estimate can be improved up to an $O(s)^{l-1}$ (as we already noticed, in the corank-one case $b(\Omega_p\cap\{J\leq s\})$ is indeed constant). Before discussing the general case, let us continue the corank-two example.

\begin{example}[Corank-two: a topological coarea formula] In this case not only we can show that $b(\Omega_p\cap\{J\leq s\})$ grows as $O(s)$, but also we can analiticaly compute the leading coefficient.

To explain the result, consider a unit circle $S^1 \subset W\simeq (\dd)^*$. As the parameter $\theta$ varies on $S^1$ let us consider the positive eigenvalues $\alpha_1(\theta),\ldots, \alpha_j(\theta)$ of the matrix $i\cos(\theta)A_1+i\sin(\theta)A_2$. The genericity assumption ensures that these numbers can be taken as the value of semialgebraic functions $\alpha_j:S^1\to \mathbb{R}.$ Given $p\in \Delta$ we consider the \emph{rational} functions $\lambda_j:S^1\to \mathbb{R}\cup \{ \infty\}$ given by:
$$\lambda_j:\theta \mapsto \left|\frac{\alpha_j(\theta)}{p_1\cos(\theta)+p_2 \sin(\theta)}\right|\quad \textrm{for}\quad j=1, \ldots, d.$$
Notice that when $\omega$ approaches $p^{\perp}$ these functions might explode, that is why they are rational in $\theta$; on the other hand they are semialgebraic and differentiable almost everywhere and it makes sense to consider the integral:
$$\tau(p)\doteq\frac{1}{2}\int_{S^1}\sum_{j=1}^{d}\left|\frac{\partial\lambda_j}{\partial \theta}(\theta)\right|-\left|\sum_{j=1}^d \frac{\partial\lambda_j}{\partial \theta}(\theta)\right|d\theta.$$
The convergence of the integral follows from the fact that where the derivatives explode, they all have the same sign and the integrand vanishes.

\begin{teoi}[Topological coarea formula]If the corank $l=2$, for a generic choice of $W\subset \sod$ and $p\in \dd$ we have:
\begin{equation}\label{topco}b({\Omega}_{p}\cap\{J\leq s\})=\tau(p) s+o(s).\end{equation}
\end{teoi}
The name ``topological coarea'' reminds of integral geometry: in fact the coefficient $\tau(p)$ is computed by considering the sum of the Betti numbers $b=\sum_{i}b_i$ as an integral over the index set, regarding $i$ as the ``variable'' of integration; rearranging this sum in an appropriate way, in the limit gives exactly the coarea formula for functions of one variable. In fact, looking at the inequalities defining $\Omega_p\cap\{J\leq s\}$, one immediately sees that they are quadratic. There is a general theory (see Appendix \ref{appendixquadrics} or \cite{AgrachevLerarioSystems} for more details) for studying the topology of sets defined by quadratic inequalities: the idea is to consider quadratic forms obtained by taking linear combinations of the equations defining the set, i.e. considering quadratic forms depending on some (homogeneous) parameters (there are as many parameters as the number of inequalities minus one). The main ingredient is  the function on this parameter space that counts the number of positive eigenvalues (the \emph{positive inertia index}) of the corresponding quadratic form. Roughly, the theory says that for each ``change'' in the monotony of  this function, there corresponds a ``hole'' in the set. Thus one can count, for each level of the positive inertia index, how many homology classes are at that level, and this gives another rearrangement of the above sum.
\end{example}
\begin{figure}
% Generated with LaTeXDraw 2.0.8
% Wed Nov 20 18:29:08 EST 2013
% \usepackage[usenames,dvipsnames]{pstricks}
% \usepackage{epsfig}
% \usepackage{pst-grad} % For gradients
% \usepackage{pst-plot} % For axes
\scalebox{0.6} % Change this value to rescale the drawing.
{
\begin{pspicture}(0,-4.699199)(12.081895,4.719199)
\definecolor{color1004b}{rgb}{0.8,0.8,0.8}
\psline[linewidth=0.04cm,linestyle=dashed,dash=0.16cm 0.16cm](2.94,2.5008008)(2.94,-4.679199)
\psline[linewidth=0.04cm,linestyle=dashed,dash=0.16cm 0.16cm](0.0,-2.1791992)(7.82,-2.1191993)
\pspolygon[linewidth=0.04,linecolor=white,fillstyle=solid,fillcolor=color1004b](2.96,-2.1391993)(4.54,2.460801)(8.32,-1.2391992)
\psline[linewidth=0.04cm](2.26,4.640801)(10.28,-3.1391993)
\psline[linewidth=0.04cm](2.96,-2.1791992)(4.96,3.6408007)
\psline[linewidth=0.04cm](2.98,-2.1391993)(12.06,-0.5591992)
\psline[linewidth=0.04cm,arrowsize=0.05291667cm 2.0,arrowlength=1.4,arrowinset=0.4]{->}(3.0,-2.1591992)(4.6,-0.4991992)
\psdots[dotsize=0.12](4.68,1.3808007)
\psdots[dotsize=0.12](4.96,0.78080076)
\psdots[dotsize=0.12](4.92,0.0)
\psdots[dotsize=0.12](4.1,-0.29919922)
\psdots[dotsize=0.12](5.42,-0.39919922)
\psdots[dotsize=0.12](5.2,2.9408007)
\psdots[dotsize=0.12](7.26,2.3608007)
\psdots[dotsize=0.12](6.08,2.200801)
\psdots[dotsize=0.12](7.24,1.1208007)
\psdots[dotsize=0.12](6.1,1.3008008)
\psdots[dotsize=0.12](6.86,3.1808007)
\psdots[dotsize=0.12](9.7,2.6008008)
\psdots[dotsize=0.12](8.64,1.5408008)
\psdots[dotsize=0.12](8.1,2.5408008)
\psdots[dotsize=0.12](9.1,0.10080078)
\psdots[dotsize=0.12](8.12,0.20080078)
\psdots[dotsize=0.12](10.6,1.4608008)
\psdots[dotsize=0.12](10.14,0.10080078)
\psline[linewidth=0.04cm,arrowsize=0.05291667cm 2.0,arrowlength=1.4,arrowinset=0.4]{->}(9.94,-2.4191992)(10.96,-1.5391992)
\usefont{T1}{ptm}{m}{n}
\rput(11.161455,-2.1341991){$s\to \infty$}
\usefont{T1}{ptm}{m}{n}
\rput(3.471455,4.5258007){$\omega(p)=s$}
\usefont{T1}{ptm}{m}{n}
\rput(4.631455,-0.8741992){$p$}
\end{pspicture} 
}
\caption{(Corank two) The number of critical manifolds with energy less than $s$ can grow quadratically in $s$ (this number is proportional to the volume of the shaded region; dots are Lagrange multipliers).}
\label{fig:lattice}
\end{figure}
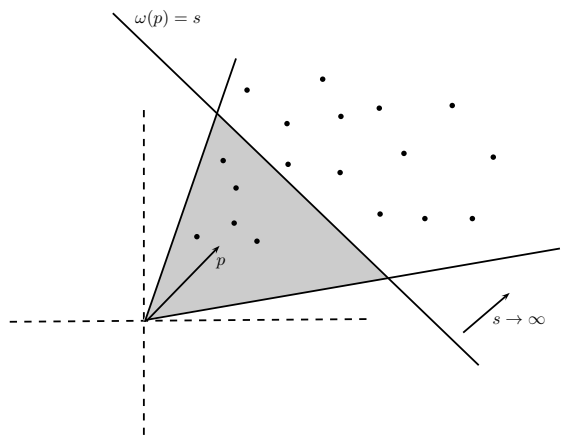

In the general case the upper bound we get is the following (see Figure \ref{fig:bound}).
\begin{teoi}[Bound on the growth rate of the topology]]For the generic choice of $W\subset \sod$ and of $p\in \dd$ we have:
$$b(\Omega_p\cap\{J\leq s\})\leq O(s)^{l-1},\quad l=\dim (\dd).$$

\end{teoi}

Here the idea is to use the function $-J|_{\Omega_p}$, which is again Morse-Bott, whose critical manifolds are the same of $J|_{\Omega_p}$ but with \emph{infinite} index. Thus passing one of these critical manifolds amounts of attaching an infinite dimensional cell (i.e the homotpy of the sublevel set doesn't change). Thus one can start with $\Omega_p\cap\{J\leq s\}$  and ``push'' its topology to its boundary $\Omega_p\cap\{J=s\}$ via $\nabla J |_{\Omega_p}$. This observation will have important consequences for our study, including the above result.

If we look back again at the Heseinberg example, we see that $b_i(\Omega_p\cap\{J\leq s\})$ is zero for $i>O(s),$ simply because $\Omega_p\cap\{J\leq s\}$ has the homotopy type of a sphere whose dimension grows linearly in $s.$ It is in fact a general phenomenon that the maximum nonzero Betti number increases at most linearly in $s$ (with no restrictions on the corank).

\begin{teoi}[The vanishing rate of Betti numbers] For the generic choice of $W\subset \sod$ and of $p\in \dd$ we have:
$$\max\{i\,|\, b_i(\Omega_p\cap \{J\leq s\})\neq 0\}\leq O(s).$$

\end{teoi} 

At this point the study of question (ii) above is complete and one can address question (iii). What we know so far is that the whole $\Omega_p$ is contractible and we have rather precise asymptotics on the behavior of the number of geodesics in $\Omega_p\cap\{J\leq s\}$ and its total Betti number as $s\to \infty$. What we have to do is simply to translate the results using the homeomorphism \eqref{eq:scaling2}:
$$\Omega_{\delta_{\epsilon}(p)}\cap \{J\leq c\}\simeq \Omega_p\cap\{J\leq s\}\quad \textrm{as $s=c/\epsilon^2\to \infty$}.$$

For example one might ask wether the estimate for the number of critical points and the topology of $\Omega_{\delta_\epsilon(p)}\cap \{J\leq c\}$ can be actually attained. In the case $l=2$ we have seen that the leading coefficient $\tau(p)$ of $b(\Omega_{\delta_\epsilon(p)})$ can be analitically computed and in fact for the generic Carnot group structure and $p\in \dd$ it is not difficult to show that $\tau(p)\neq 0$. In particular as we let $\epsilon\to 0$, the topology of $\Omega_{\delta_\epsilon(p)}$ explodes; as a corollary the number of geodesics grows unbounded as well (this is an example of a backward use of Morse-Bott inequalities).

For the general corank $l$ we do not know whether the limit:
$$\tau(p)=\limsup_{\epsilon\to 0}b(\Omega_{\delta_\epsilon(p)}\cap\{J\leq 1\})\epsilon^{2l-2}=\limsup_{\epsilon\to 0}b(\Omega_p\cap \{J\leq1/\epsilon\})\epsilon^{l-1}$$ 
is different from zero, but it is natural to guess so (we only know it is finite).

\begin{figure}

% Generated with LaTeXDraw 2.0.8
% Wed Nov 20 17:24:01 EST 2013
% \usepackage[usenames,dvipsnames]{pstricks}
% \usepackage{epsfig}
% \usepackage{pst-grad} % For gradients
% \usepackage{pst-plot} % For axes
\scalebox{0.6} % Change this value to rescale the drawing.
{
\begin{pspicture}(0,-2.79)(16.001015,2.79)
\psline[linewidth=0.04cm,arrowsize=0.05291667cm 2.0,arrowlength=1.4,arrowinset=0.4]{<-}(0.6010156,2.77)(0.6210156,-2.73)
\psline[linewidth=0.04cm,arrowsize=0.05291667cm 2.0,arrowlength=1.4,arrowinset=0.4]{->}(0.64101565,-2.71)(15.981015,-2.77)
\psline[linewidth=0.04cm](0.6010156,-1.9511149)(1.3609643,-1.9511149)
\psline[linewidth=0.04cm,linestyle=dashed,dash=0.16cm 0.16cm](1.3609643,-1.9511149)(1.3609643,-2.65)
\psline[linewidth=0.04cm](3.2010157,-1.95)(3.9620342,-1.95)
\psline[linewidth=0.04cm,linestyle=dashed,dash=0.16cm 0.16cm](3.2010157,-1.95)(3.2010157,-2.7186956)
\psline[linewidth=0.04cm,arrowsize=0.05291667cm 2.0,arrowlength=1.4,arrowinset=0.4]{->}(8.621016,1.89)(9.801016,1.89)
\usefont{T1}{ptm}{m}{n}
\rput(2.3924707,2.515){$b_i(\Omega_p\cap\{J\leq s\})$}
\usefont{T1}{ptm}{m}{n}
\rput(15.332471,-2.365){$i$}
\usefont{T1}{ptm}{m}{n}
\rput(10.962471,1.855){$s\to \infty$}
\psline[linewidth=0.04cm](4.0022616,-1.1305199)(4.7210155,-1.13)
\psline[linewidth=0.04cm,linestyle=dashed,dash=0.16cm 0.16cm](4.0022616,-1.1305199)(4.0022616,-2.667399)
\psline[linewidth=0.04cm](4.7410154,0.23026311)(5.4410124,0.23026311)
\psline[linewidth=0.04cm,linestyle=dashed,dash=0.16cm 0.16cm](4.7410154,0.23026311)(4.7410154,-2.6673875)
\psline[linewidth=0.04cm,linestyle=dashed,dash=0.16cm 0.16cm](5.4410124,0.23026311)(5.4410124,-2.71)
\psline[linewidth=0.04cm](5.4410157,-0.37)(6.1410155,-0.37)
\psline[linewidth=0.04cm](6.1410155,1.449803)(6.821019,1.449803)
\psline[linewidth=0.04cm,linestyle=dashed,dash=0.16cm 0.16cm](6.1410155,1.449803)(6.1410155,-2.5708723)
\psline[linewidth=0.04cm,linestyle=dashed,dash=0.16cm 0.16cm](6.821019,1.449803)(6.821019,-2.63)
\psline[linewidth=0.04cm](6.821016,0.83)(7.5210156,0.83)
\psline[linewidth=0.04cm](7.5210156,2.2100012)(8.219905,2.2100012)
\psline[linewidth=0.04cm,linestyle=dashed,dash=0.16cm 0.16cm](7.5210156,2.2100012)(7.5210156,-2.6189854)
\psline[linewidth=0.04cm,linestyle=dashed,dash=0.16cm 0.16cm](8.219905,2.2100012)(8.219905,-2.69)
\psline[linewidth=0.04cm](8.221016,0.85)(8.8610115,0.83022463)
\psline[linewidth=0.04cm,linestyle=dashed,dash=0.16cm 0.16cm](8.8610115,0.83022463)(8.8610115,-2.71)
\psline[linewidth=0.04cm](9.501016,0.28973505)(10.160949,0.28973505)
\psline[linewidth=0.04cm,linestyle=dashed,dash=0.16cm 0.16cm](9.501016,0.28973505)(9.501016,-2.6665256)
\psline[linewidth=0.04cm,linestyle=dashed,dash=0.16cm 0.16cm](10.160949,0.28973505)(10.160949,-2.71)
\psline[linewidth=0.04cm](10.1610155,-0.36967766)(10.821015,-0.35)
\psline[linewidth=0.04cm,linestyle=dashed,dash=0.16cm 0.16cm](10.801016,-0.37)(10.821015,-2.73)
\psline[linewidth=0.04cm](11.461016,-0.40933883)(12.142162,-0.40933883)
\psline[linewidth=0.04cm,linestyle=dashed,dash=0.16cm 0.16cm](11.461016,-0.40933883)(11.461016,-2.695961)
\psline[linewidth=0.04cm,linestyle=dashed,dash=0.16cm 0.16cm](12.142162,-0.40933883)(12.142162,-2.7295878)
\psline[linewidth=0.04cm](8.841016,-0.37)(9.461016,-0.37)
\psline[linewidth=0.04cm](10.841016,-1.15)(11.441015,-1.15)
\psline[linewidth=0.04cm](12.1610155,-2.15)(12.821015,-2.15)
\psline[linewidth=0.04cm,linestyle=dashed,dash=0.16cm 0.16cm](12.801016,-2.19)(12.801016,-2.73)
\end{pspicture} 
}
\caption{The Betti numbers of $\Omega_p\cap\{J\leq s\}$. As $s$ goes to infinity the ``wave'' moves to the right and the sum of all the Betti numbers (the area below the wave) can increase as fast as $O(s^{l-1})$, but eventually everything ``disappears''.}
\label{fig:bound}
\end{figure}
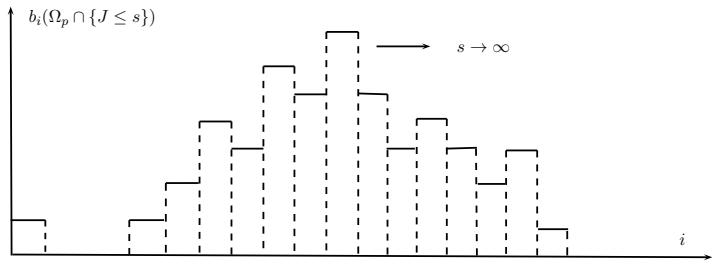

As a concluding remark, we believe these asymptotic Morse inequalities can be useful for generalizations and, in particular, for the next paper that we plan to devote
to the general step-two sub-Riemannian structure. Of course, in general, we cannot guarantee that the functional is Morse-Bott.
The quantity to estimate could be a ``virtual number of geodesics": the minimal number of critical points for a Morsification of the functional. 

\subsection{Structure of the paper}The paper is organized as follows. In section \ref{sec:preliminaries} we discuss some preliminary material and give the main definitions. In Section \ref{sec:geodesics} we study the structure of geodesics: the theorem on the smoothness and the topology of $\Omega_p$ is Theorem \ref{regularvalue} and the theorem on the structure of geodesics is a combination of Lemmas \ref{endpoint}, \ref{nomultiple} and \ref{distinct} and Theorem \ref{structure}. The growth of the number of critical manifolds is computed in Theorem \ref{criticalcount}; the bound on the growth rate of the topology of $\Omega_p\cap\{J\leq s\}$ is proved in Theorem \ref{bettiorder} and the exact asymptotic for the case $l=2$ is the content of Theorem \ref{coarea}; the theorem on the vanishing rate of the Betti numbers is Corollary \ref{maxbetti}. The Appendix contains, for the reader's convenience, some useful results we will use in the main body.

\section{Preliminaries}\label{sec:preliminaries}
\subsection{Step two Carnot Groups and their geometry}
Here we briefly recall the main definitions related to Carnot groups; the reader is referred to \cite{AgrachevBarilariBoscain, Semmes, OnishchickVinberg} for more details.

As a differentiable manifold a \emph{step two Carnot group} is a connected, simply connected Lie group $G$ whose Lie algebra $\g=T_eG$ decomposes as:
$$\g=\Delta \oplus \dd, \quad \textrm{with} \quad [\Delta, \Delta]=\dd,\quad [\g, \Delta]=0;$$
as vector spaces here we have $\Delta \cong \R^{d}$ and $\dd \cong \R^l$.

Whenever a Lie algebra $\g$ as above is given, the existence (and uniqueness) of such a group $G$ is guaranteed by Lie's theorem. We recall that in fact under the above assumption on the structure of $\g,$ the exponential map $\textrm{exp}:\g \to G$ is an analytic diffeomorphism, hence in particular $G\simeq\R^{d+l}.$

The geometric structure on $G$ is given by fixing a scalar product $h$ on $\Delta$ and considering the distribution $\Delta_q=dL_q\Delta$ together with the extension of $h$ by left translation; in this way the triple $(G, \Delta, h)$ defines a \emph{sub-Riemannian manifold}. Notice that the distribution $\Delta$ is by assumption bracket generating, hence it satisfies H\"ormander's condition, in particular meaning that $G$ is \emph{admissible-path-connected} (see below).

Two such Carnot groups $(G_1, \Delta_1, h_1)$ and $(G_2, \Delta_2, h_2)$ are considered to be isomorphic if there exists a Lie algebra isomorphism $L:\g_1\to \g_2$ such that $L\Delta_1=\Delta_2$ and $L|_{\Delta_1}^*h_2=h_1:$ in fact by the simple connectedness assumption the linear map $L$ integrates to a Lie group isomorphism $\phi:G_1\to G_2$ and the global geometric structures are related by $(G_1, \Delta_1, h_1)=(G_1, \phi^{-1}\Delta_2, \phi^*h_2).$

If we fix now an \emph{orthonormal} basis $\{e_1, \ldots, e_d\}$ of $\Delta$ and a basis $\{f_1, \ldots, f_l\}$ of $\dd,$ then the bracket structure can be written as:
\begin{equation}\label{eq:structure}[e_i, e_j]=\sum_{k=1}^la_{ij}^kf_k, \quad \textrm{for all } i,j\in\{1, \ldots,d\}\end{equation}
where each matrix $A_k=(a_{ij}^k)$ belongs to $\sod.$ In particular we can consider the vector space:
$$W=\textrm{span}\{A_1, \ldots, A_l\}\subset \sod.$$
Each vector space $W$ of dimension $l$ in $\sod$ defines a sub-Riemannian structure on $G$ by fixing a basis $\{A_1, \ldots, A_l\}$ for $W$ and declaring that the corresponding matrices define the bracket structure in an orthonormal basis $\{e_1, \ldots, e_d\}$ for $\Delta.$ The isomorphism class of the Carnot group does not depend on the choice of the basis of $W$: indeed, let $\{A'_1,\ldots,A'_l\}$ be another basis of $W$ and $B=(b_{hk})$ the basis-change matrix, such that $A'_h=\sum_{k=1}^{l}b_{hk}A_k$. Now we can build another Carnot group by defining its Lie algebra $\mathfrak{g}'$ with basis $\{e'_1,\ldots,e'_d,f'_1,\ldots,f'_l\}$ such that $\{e'_1,\ldots,e'_d\}$ is orthonormal; the structure constants are given by the entries of the matrices $A'_h$ as in equations (\ref{eq:structure}). The map $\phi$, defined on the basis elements by $e'_i \mapsto e_i$ and $f'_h \mapsto \sum_{k=1}^{l} b_{hk}f_k$, gives an isomorphism between $\mathfrak{g}'$ and $\mathfrak{g}$.
\begin{remark}[The moduli space of Carnot Groups] Once a sub-Riemannian structure is given, changing $\{e_1, \ldots, e_d\}$ to another orthonormal basis $\{Me_1, \ldots, Me_d\}$ (where $M$ is an orthogonal matrix in $O(d)$) changes $W$ into $W'=MWM^T.$ Thus denoting by $G(l, \sod)$ the Grassmannian of $l$-planes in $\sod$, the (naive) moduli space of step two Carnot Groups is represented by the quotient $\mathcal{M}_{l,d}=G(l, \sod)/O(d)$. Since $\mathcal{M}_{l,d}$ is the quotient of a manifold by a Lie group action, the quotient map is open and perturbing $W$ defines a ``genuine" perturbation of the isomorphism class of the corresponding Carnot group; in particular this means that a generic choice of $W$ results in a generic choice of an isomorphism class of Carnot groups.
\end{remark}

Motivated by the above remark, once $W\in G(l, \sod)$ is fixed we consider the Carnot group given by exponentiating $\g=\R^d\oplus \R^l$, whose Lie algebra is given as follows: $\{e_1, \ldots, e_d\}$ is the standard orthonormal basis for $\R^d$, $\{f_1,\ldots,f_l\}$ is the standard basis for $\R^l$ and fixing a basis $\{A_1, \ldots, A_l\}$ for $W$, the Lie brackets are given by equation (\ref{structure}); we will call $W$ the \emph{Carnot group structure}.

The following theorem gives a geometric realization of Carnot groups.
\begin{teo}\label{geometric}
Let $\{A_1, \ldots, A_l\}$ be a basis for $W\subset \sod$ and for $i=1, \ldots, d$ consider the vector fields $E_i$ on $\R^{d+l}$ defined in coordinates $(x,y)$ by:
$$E_i(x,y)=\frac{\partial}{\partial x_i}(x,y)-\frac{1}{2}\sum_{k=1}^l \sum_{j=1}^{d}a_{ij}^kx_j\frac{\partial}{\partial y_k}(x,y).$$
Then the sub-Riemannian manifold $(\R^{d+l}, \Delta=\textrm{span}\{E_1, \ldots, E_l\}, g),$ where $g$ is the standard Euclidean metric, is isomorphic to the Carnot group defined by $W.$
\end{teo}

\subsection{The End-point map}
\emph{Admissible paths} on $G$ are defined to be curves $\gamma:I=[0, 2\pi]\to G$ whose derivative exists almost everywhere, is square integrable and belongs to the distribution $\Delta$. We denote the set of such paths by $\Omega.$ Our choice of the interval $I=[0, 2\pi]$ is motivated by simplicity of notation: we will later need to expand the components of an admissible path into their Fourier series; a different choice of the interval will produce a completely equivalent theory.

As earlier noticed, the bracket generating condition implies that any two points in $G$ can be joined by an admissible path.
The set $\Omega$ can be given a \emph{Hilbert manifold structure} as follows. Let $u=(u_1, \ldots, u_d)\in L^2(I, \R^d)$ and consider the Cauchy problem:
$$
\dot{\gamma}(t)=\sum_{i=1}^{d}u_{i}(t)E_{i}(\gamma(t)),\quad \gamma(0)=e.
$$

The reader can assume this ODE problem is set on $\R^{d+l},$ using the above geometric realization Theorem; in this case the identity element $e\in G$ corresponds to the zero of $\R^{d+l}.$
By Caratheodory's Theorem the above Cauchy problem has a local solution $\gamma_{u}$ and we consider the set:
$$\mathcal{U}=\{u\in L^2(I, \R^d)\,|\, \gamma_u \textrm{ is defined for $t=2\pi$}\}.$$
For general subriemannian manifolds the set $\mathcal{U}$ is an open subset of $L^2(I, \R^d)$ (by ODE's continuous dependence theorem) and is called the set of \emph{controls}; in our case the estimates for the final time can be made uniform and we actually have $\mathcal{U}=L^2(I, \mathbb{R}^{d})$. Associating to each $u$ the corresponding path $\gamma_u$ gives thus a local coordinate chart and by slightly abusing of notation in the sequel we will often identify $\Omega$ with $\mathcal{U}$.

Once we are given the Carnot group structure $W=\textrm{span}\{A_1, \ldots, A_l\}$, we can use Theorem \ref{geometric} to write down the above ODE in a more explicit form:
$$\left\{
\begin{array}{l}
\dot{x}=u \\
\dot{y}_{i}=\frac{1}{2}x^{T}A_{i}u
\end{array}
\quad\textrm{and}\quad \gamma(0)=0\right.
$$
In this framework the \emph{End-point map} is the \emph{smooth} map: $$
F: \Omega \longrightarrow G,$$
which associates to each curve $\gamma$ its final point $\gamma(2\pi).$

If we define $A=(A_{1},\ldots,A_{l}),$ we can use again Theorem \ref{geometric} and write the End-point map as:
\begin{equation}\label{endpointexplicit}
F(u)=\left(\int_{0}^{2\pi}u(t)\,dt,\frac{1}{2}\int_{0}^{2\pi}\left\langle\int_{0}^{t}u(\tau) d\tau  , A\,u(t)\right\rangle dt \right);
\end{equation}
(here the brackets denote the sub-Riemannian scalar product on the Lie algebra $\mathfrak{g}$).

In the sequel we will mainly be interested in admissible paths whose endpoints lie on $\Delta^2$ (see Proposition \ref{propo:bound} below).
Being $G$ of step two, we know that $\Delta^{2}$ is an abelian subalgebra of $\mathfrak{g};$ therefore we can identify $\Delta^{2}$ with the submanifold $\textrm{exp}(\Delta^2) \subset G$ and using Theorem \ref{geometric} we can write this identification as: 
$$\xi_{1}f_{1} + \ldots + \xi_{l}f_{l}\mapsto (\underbrace{0,\ldots,0}_{x},\underbrace{\xi_{1},\ldots,\xi_{l}}_{y}),$$
(here as above $\{f_1, \ldots, f_l\}$ is a basis of $\Delta^2$).

We study now the structure of the set of admissible paths whose endpoints are on $\Delta^2$. It turns out that in the local coordinates given by the controls $\mathcal{U}$ it coincides with the kernel of the differential of $F$ at $0 \in \mathcal{U}$
$$H=\ker D_0F,$$
as described by the following proposition.
\begin{propo}\label{propoH}
The following properties hold:
\begin{enumerate}[(a)]
\item $H=\left\{u\in L^2(I, \R^d)\, |\,  \int_I u\, dt =0 \right\}$;
\item $u \in H \Leftrightarrow F(u) \in \emph{\textrm{exp}}(\Delta^2)$;
\item $F|_{H}=\emph{He}_{0}F$.
\end{enumerate}
\end{propo}
\begin{proof}
For point (a) we compute the differential $D_{0}F$: by taking a variation $\varepsilon v$ of the constant curve $\gamma\equiv 0$ we easily see that
$$
D_{0}Fv=\left. \frac{d}{d\varepsilon}\right|_{\varepsilon=0}\left(\varepsilon\int_{0}^{2\pi}v(t)\,dt,\frac{1}{2}\varepsilon^{2}\int_{0}^{2\pi}\left\langle\int_{0}^{t}v(\tau) d\tau , A\,v(t)\right\rangle dt  \right) = \left(\int_{0}^{2\pi}v(t)\,dt,0  \right) \in \mathfrak{g}.
$$
which proves property (a).\\
Point (b) is a direct consequence of equation (\ref{endpointexplicit}). \\
For point (c) we notice that the Hessian $\he_{0}F$ is defined on $H=\ker D_{0}F$ with values in $\coker D_{0}F = \Delta^{2}$, thus has the same range as $F|_{H}$. As in the proof of (a) if we consider the second derivative of a variation and we easily obtain $\he_0F$ has the same expression of $F$ when restricted to $H$.
\end{proof}
We denote by $q$ the Hessian of $F$ at zero, i.e. the quadratic map:
$$q \doteq F|_H:H\to \R^l.$$
Every component $q_{i}$ of $q$ is a quadratic form on $H$ and its 
explicit expression is given by:
\begin{equation} \label{explicitqi}
q_{i}(u)=\frac{1}{2} \int_{0}^{2\pi}\left\langle\int_{0}^{t}u(\tau) d\tau , A_{i}\,u(t)\right\rangle dt.
\end{equation}
By polarization we obtain the expression for the associated bilinear form:
\begin{align*}
q_{i}(u,v)=&\frac{1}{4}\left( \int_{0}^{2\pi}\left\langle\int_{0}^{t}u(\tau)+v(\tau) d\tau , A_{i}(u(t)+v(t))\right\rangle dt - \int_{0}^{2\pi}\left\langle\int_{0}^{t}u(\tau) d\tau , A_{i}\,u(t)\right\rangle dt + \right. \\
&\left. -\int_{0}^{2\pi}\left\langle\int_{0}^{t}v(\tau) d\tau , A_{i}\,v(t)\right\rangle dt \right) = \\
& = \frac{1}{4} \left(\int_{0}^{2\pi}\left\langle\int_{0}^{t}u(\tau) d\tau , A_{i}\,v(t)\right\rangle dt + \int_{0}^{2\pi}\left\langle\int_{0}^{t}v(\tau) d\tau , A_{i}\,u(t)\right\rangle dt \right) = \\
& = \frac{1}{4} \left(\int_{0}^{2\pi}\left\langle\int_{0}^{t}u(\tau) d\tau, A_{i}\,v(t)dt\right\rangle - \int_{0}^{2\pi}\left\langle v(t),A_{i}\int_{0}^{t}u(\tau) d\tau \right\rangle dt \right) = \\
& =\frac{1}{2}\int_{0}^{2\pi}\left\langle\int_{0}^{t}u(\tau) d\tau,A_{i}\,v(t)\right\rangle dt, \\
\end{align*}
where the fourth row follows from integration by parts. \\
Moreover to every $q_{i}$ it corresponds a symmetric operator $Q_{i}:H \rightarrow H$ defined by:
$$q_i(u)=\langle u, Q_i u\rangle_H\quad \textrm{for all $u\in H$}.$$
(recall that $\langle u, v\rangle_H=\int_I\langle u, v\rangle dt$). We will use the notation $Q$ for the map $(Q_{1}, \ldots, Q_{l}):H \rightarrow H \otimes \R^{l}$. Also, given a covector $\omega \in \R^{l*}$ we will denote by $\omega q, \omega Q, \omega A$ the compositions of $\omega:\R^l\to \R$ respectively with $q, Q$ and $A.$\\
In the sequel we will need to expand a control $u\in H$ into its Fourier series: we will write $u=\sum_{k \in \mathbb{N}_0}U_{k}\coskt + V_{k} \sinkt$ where $U_{k},V_{k} \in \Delta$; the constant term is zero because of part (a) of Proposition \ref{propoH} (mean zero condition).
\begin{propo} \label{operatorQstructure}
Let $T_{k}$ be the subspace of $H$ with ``wave number" $k$, namely $$T_{k}=\Delta \otimes \emph{\textrm{span}} \{\cos  kt,\sin kt\}.$$ Then we have the following:
\begin{enumerate}[(a)]
\item $H = \oplus_{k \geq 1} T_{k}$, and the sum is orthogonal with respect to the scalar product;
\item For every $\omega \in \R^{l*}$ we have $\omega Q T_{k}\subset T_{k}$ (i.e. each subspace $T_k$ is invariant by $\omega Q$);
\item Consider the orthonormal basis $\{e_i \otimes \coskt, e_{i} \otimes \sinkt \}_{i=1}^d$ for $T_k$; in this basis the matrix associated to $\omega Q|_{T_{k}}$ is:
$$
\frac{1}{k}(\omega P)\doteq\frac{1}{k}
\left(
\begin{array}{cc}
0 & \frac{1}{2}\omega A \\
-\frac{1}{2}\omega A & 0 \\
\end{array}
\right).
$$
\end{enumerate}
\end{propo}
\begin{proof}
Point (a) is just Fourier decomposition theorem; $k\geq 1$ expresses the mean zero condition.\\
For the other two points, let us consider $u \in T_{n}$ and $v \in H$, with Fourier series respectively $u=U\coskt + V \sinkt$ and $v=\sum_{n \geq 1} U_{n} \frac{1}{\sqrt{\pi}}\cos nt + V_{n}\frac{1}{\sqrt{\pi}}\sin nt$. By a direct computation we have
\begin{align*}
\langle u,\omega Q v \rangle_H & = \int_{0}^{2\pi}\left\langle\int_{0}^{t} U \frac{1}{\sqrt{\pi}} \cos k\tau + V\frac{1}{\sqrt{\pi}}\sin k\tau \, d\tau,\frac{1}{2} \omega A v\right\rangle dt= \\
&=\sum_{n \geq 1}\int_{0}^{2\pi}\frac{1}{k}\left\langle U\sinkt - V\coskt, \frac{1}{2}\omega A \left(U_{n}\frac{1}{\sqrt{\pi}} \cos nt  + V_{n} \frac{1}{\sqrt{\pi}}\sin nt\right)\right\rangle dt= \\
& = \frac{1}{k}\int_{0}^{2\pi}-\frac{1}{\pi}(\cos kt)^2\left\langle V,\frac{1}{2}\omega A U_{k}\right\rangle + \frac{1}{\pi}(\sin kt)^2\left\langle U,\frac{1}{2}\omega A V_{k}\right\rangle dt= \\
&=\frac{1}{k}\bigg(-\left\langle V,\frac{1}{2}\omega A U_{k}\right\rangle +\left\langle U,\frac{1}{2} \omega A V_{k}\right\rangle \bigg),
\end{align*}
where the equality between second and third row holds because the only non-zero integrals of products of sines/cosines are
$\int_{0}^{2\pi}\frac{1}{\pi}(\cos kt)^2 \, dt = \int_{0}^{2\pi}\frac{1}{\pi}(\sin kt)^2 \, dt = 1.$
\end{proof}
\begin{remark}
We notice that for every $\omega \in (\Delta^2)^*$ the operator $\omega Q$ is compact. Indeed, it is the limit of a converging series of operators with finite-dimensional image:
$$
S_{n} = \sum_{i=1}^{n} \omega Q |_{T_{i}}.
$$
Let us prove that the \emph{operator} norm of $\omega Q-\omega S_n$ goes to zero. Given a norm one $v=\sum_{k\geq 1} v_k,$ we have:
\begin{align}\nonumber\|(\omega Q - \omega S_{n})v\|^2&=\sum_{k\geq n+1}\|\omega Q|_{T_k}v\|^2\leq\sum_{k\geq n+1}\frac{4\|\omega P\|_{\textrm{op}}^2\|v_k\|^2}{k^2}\\
&\nonumber\leq \frac{4\|\omega P\|_{\textrm{op}}^2}{(n+1)^2}\sum_{k\geq n+1}\|v_k\|^2\leq \frac{4\|\omega P\|_{\textrm{op}}^2}{(n+1)^2}.\end{align}
In particular, taking square roots, $\|(\omega Q - \omega S_{n})v\|\leq \frac{4\|\omega P\|_{\textrm{op}}}{(n+1)},$ i.e.~$\|\omega Q - \omega S_{n}\|_{\textrm{op}}\to0$
\end{remark}
We conclude this chapter by describing the spectrum of the operator $\omega Q$. Given $\omega Q$ we consider as above the skew-symmetric matrix $\omega A,$ which can be put in canonical form as a block matrix of the form $\diag\left(\alpha_1(\omega) J_{2},\ldots,\alpha_m(\omega) J_2, 0_n\right)$, where $J_{2} \in \mathfrak{so}(2)$ is the standard symplectic matrix and $0_n$ is the $n \times n$ zero matrix. More precisely on $\Delta$ we can find an orthonormal basis $\{X_{i},Y_{i},Z_{j},\,i=1,\ldots,m,\,j=1,\ldots,n  \}$ for suitable $m,n \in \mathbb{N}$ satisfying $2m+n=d$ such that:
$$
\omega A X_{i} = -\alpha_{i}(\omega) Y_{i}, \quad \omega A Y_{i} = \alpha_{i}(\omega) X_{i}, \quad \omega A Z_{j}=0.
$$
Let us consider now the operator $\omega Q$ restricted to $T_{k}$. Using the basis for $\Delta$ defined above we get the orthogonal basis $$\left\{\left(\begin{array}{c} X_{i} \\ Y_{i} \end{array} \right),\left(\begin{array}{c} -Y_{i} \\ X_{i} \end{array} \right),\left(\begin{array}{c} X_{i} \\ -Y_{i} \end{array} \right),\left(\begin{array}{c} Y_{i} \\ X_{i} \end{array} \right),\left(\begin{array}{c} Z_{j} \\ 0 \end{array} \right),\left(\begin{array}{c} 0 \\ Z_{j} \end{array} \right)\right\},\ i=1,\ldots,m,\ j=1,\ldots,n,$$ with eigenvalues $\frac{\alpha_{i}(\omega)}{k},\frac{\alpha_{i}(\omega)}{k},-\frac{\alpha_{i}(\omega)}{k},-\frac{\alpha_{i}(\omega)}{k},0,0$ respectively. Thus we have proved:
\begin{propo} \label{eigens}
The non-zero eigenvalues of the operator $\omega Q$ are $\pm \frac{\alpha_{i}(\omega)}{k}$ with multiplicity two, where $\alpha_{i}(\omega)$ are the coefficients of the canonical form of $\omega A$ and $k \in \mathbb{N}_{0}$.
\end{propo}

\section{Geodesics}\label{sec:geodesics}
\subsection{The general structure of geodesics}
In this section we introduce the \emph{admissible path space}; given a point $p\in G$ it is defined as:
$$
\Omega_p =\{\textrm{admissible curves starting at the origin and ending at $p$}\}=F^{-1}(p).$$
\begin{remark}It will be useful for us to set $\Sigma_1=\overline{\Sigma'_1}$. Since the euclidean closure is smaller than the Zariski closure and dimension of a semialgebraic set is preserved after taking its Zariski closure, then $\Sigma_1$ is \emph{closed} and contained in a semialgebraic set of codimension one. Points in the complement of $\Sigma_1$ are an open dense set of regular values of $q=F|_H.$
\end{remark}
Since in the space $\Omega$ of admissible paths we are allowed to compute velocities and their lengths, we define the \emph{Energy} functional:
$$
J:  \Omega  \longrightarrow  \R,$$
by associating (as in Riemannian geometry) to each curve $\gamma$ the integral $\int_{0}^{2\pi}\|\dot{\gamma}(t)\|^{2} dt.$

In the case $p$ is a regular value of the End-point map $F$, then a critical point of $J|_{F^{-1}(p)}$ is called a \emph{normal geodesic}.

Theorem \ref{regularvalue} allows then to study the structure of (normal) geodesics whose endpoint is the generic $p$; in fact they are by definition critical points of the restriction of $J$ to $\Omega_p$ and in the case the latter is a Hilbert manifold are defined by the Lagrange multipliers rule. 
Using the above coordinates $G\simeq \mathbb{R}^d\oplus \R^l$ we can decompose a vector $\lambda\in T^*G$ as $\lambda=\eta+\omega, $ where $\eta$ is the ``horizontal'' part and $\omega \in (\Delta^2)^*$ is the ``vertical'' one. We have the following proposition.
\begin{propo}\label{propo:multiplier}Let $u$ be the control associated to a geodesic with Lagrange multiplier $\lambda\in T^*G$ (i.e. $\lambda d_uF=d_u J$). Then:
$$u(t)=e^{-(\omega A)t}u_0\quad \textrm{and}\quad 2\eta=(e^{-2\pi \omega A}+\mathbbm{1})u_0.$$
\begin{proof}A simple computation using \eqref{endpointexplicit} in $\lambda d_uF=u$ gives:
$$u(t)=\eta-(\omega A)\int_{0}^tu(s)ds+\frac{1}{2}\int_0^{2\pi}(\omega A)u(s)ds.$$
Differentiating the above equation provides $\dot{u}=-(\omega A)u$, which in turn implies $u(t)=e^{-(\omega A)t}u_0.$
Substituting the explicit expression $u(t)=e^{-(\omega A)t}u_0$ into the same equation and evaluating at zero gives $u_0=\eta-\frac{1}{2}u(2\pi)+\frac{1}{2}u_0$.
\end{proof}
\end{propo}

If $p$ is not a vertical point and the rank is sufficiently big ($d>l$), then the number of geodesics joining the origin to $p$ is bounded. In order to prove this statement we need a preliminary lemma.

\begin{lemma}For the generic point $p\in G$, the set of geodesics starting at the origin and ending at $p$ is discrete. 
\end{lemma}
\begin{proof}
We can parametrize geodesics with their initial covector $\omega$ and the initial velocity $u_0$; in this way we obtain a smooth map:
$$f:\R^{l}\times \R^d\to G$$
defined by:
$$(\omega, u_0)\mapsto \left(\int_I e^{-t\omega A}u_0 dt,  \int_I \left\langle \int_{0}^{t}e^{-s\omega A}u_0\, ds, A e^{-t\omega A}u_0\right\rangle\, dt\right).$$
If $p$ is a regular value of $f$ (and the set of such $p$ is a residual set) then $f^{-1}(p)$ is a submanifold of $\R^{d+l}$ of dimension zero (possibly noncompact).
\end{proof}

\begin{propo}\label{propo:bound}
Assume that $d>l$.  Then  for the generic $p \notin \dd$ there is a finite number of geodesics between $e$ and $p$.
\end{propo}
 \begin{proof}
 
We already know by the previous lemma that the set of geodesics (i.e. pairs $(\omega, u_0)$ such that $f(\omega, u_0)=p$) is discrete for the generic $p$. We will exclude that the set of possible Lagrange multipliers is unbounded: this will imply (see below) that the set of initial velocities is bounded as well, hence the set of geodesics to $p$ is a discrete set in compact, i.e. it is finite.

If $\omega$ is a Lagrange Multiplier, we can always choose a basis for $\Delta$ such that the matrix $\omega A$ appears in canonical form: so we get $k$ subspaces of $\Delta$ of dimension two on which the matrix $\omega A$ is of the form $\alpha_i J_2$ (where $J_2$ is the standard $2\times 2$ symplectic matrix on the $i$-th eigenspace and $\alpha_i > 0$), for $i=1,\ldots,k$; for every subspace we take the component $u_{0}^{i}$ of the initial velocity $u_{0}$. Since the eigenspaces are orthogonal, the computations for the end-point of the geodesic can be performed separately: thus we split the horizontal part of the end-point $p_\Delta$ into the components of the eigenspaces of $\omega A$, say $p_{\Delta}^{i}$. These components are:
\begin{align*}
p^{i}_{\Delta} & =\int_{0}^{2\pi}e^{-t\omega A}u^{i}_{0} dt = \left[\frac{1}{\alpha_j} J_2 e^{-t\alpha_i J_2}u_{0}^{i}\right]_{0}^{2\pi} =\frac{1}{\alpha_j}J_2 \left(e^{-2\pi\alpha_j J_2}u_{0j} -  u_{0j}\right),
\end{align*}
 on the eigenspaces of $\omega A$. 
The norm squared of the component $p_{\Delta}^{i}$ is given by:
\begin{align*}
\|p^{i}_{\Delta}\|^2 & = \frac{1}{\alpha_{i}^{2}}\bra J_2 \left(e^{-2\pi\alpha_j J_2}u^{i}_{0} -  u^{i}_{0}\right), J_2 \left(e^{-2\pi\alpha_j J_2}u^{i}_{0} -  u^{i}_{0}\right)\ket = \\
& = \frac{2}{\alpha_{i}^{2}} \left( \|u_{0}^{i}\|^2 - \bra u_{0}^{i},e^{-2\pi\alpha_i J_2}u^{i}_{0}\ket \right) = \frac{2}{\alpha_{i}^{2}}\|u_{0}^{i}\|^2 \left(1 - \cos (2\pi \alpha_i) \right)
\end{align*}
The vertical part is more complicated, but we only need to compute its value in the direction of the Lagrange multiplier $\omega$:
\begin{align*}
\omega(p_{\dd}) & = \frac{1}{2}\int_{0}^{2\pi}\bra \int_{0}^{t}e^{-\omega A s}u_{0}ds,\omega A e^{-\omega A t}u_0\ket dt = \\
& =\frac{1}{2}\int_{0}^{2\pi}\bra \int_{0}^{t}-\omega Ae^{-\omega A s}u_{0}ds, e^{-\omega A t}u_0\ket dt = \\
&=\frac{1}{2}\int_{0}^{2\pi}\bra e^{-\omega A t}u_{0}-u_{0}, e^{-\omega A t}u_0\ket dt = \pi \|u_0\|^2 -\frac{1}{2}\bra u_0,p_{\Delta}\ket.
\end{align*}
The equation
\begin{equation} \label{eq:omegaperp}
\omega(p_{\dd}) = \pi \|u_0\|^2 -\frac{1}{2}\bra u_0,p_{\Delta}\ket
\end{equation}
is important at first because it tells that it is enough to prove that the set of Lagrange multipliers is bounded. Indeed, we have
$$
\|\omega\|\|p_{\dd}\| \geq \pi \|u_0\|^2 -\frac{1}{2}\bra u_0,p_{\Delta}\ket;
$$
if the set of Lagrange multipliers is bounded, the set of initial velocities cannot be unbounded otherwise the second term of the inequality would diverge while being limited by a constant.

In order to prove that the set of Lagrange multipliers of the generic end-point is bounded, we suppose on the contrary that we have a sequence $\omega_n$ of Lagrange multipliers; this sequence is forced to diverge since the set of the initial data $(\omega,u_0)$ for the generic end-point $p$ is discrete.
Up to subsequences, we may assume that the normalized Lagrange multipliers $\hat{\omega}_n = \omega_n / \|\omega_n\|$ converge to a covector $\lambda$; moreover we can assume the rescaled eigenvalues $\hat{\alpha}_{i,n} = \alpha_{i,n}/\|\omega_n\|$ and the corresponding eigenspaces converge. The first one is true since every eigenvalue is bounded by the norm of the matrix which is 1 by definition; the second one follows from the fact that the set of changes of basis is the orthogonal group which is compact.
For every sequence of eigenvalues $\hat{\alpha}_{i,n}$ that doesn't converge to $0$ we find a positive real number $c$ such that $c\,\omega_n \leq \alpha_{i,n}$: if we split the second term of the equation \eqref{eq:omegaperp} into its components given by the eigenspaces of $\omega A$:
$$
\omega(p_{\dd}) = \sum_{i=1}^{k} \pi \|u_{0}^{i}\|^2 - \frac{1}{2}\bra u_{0}^{i},p_{\Delta}^{i}\ket.
$$
Therefore we have:
$$
\|p_{\dd}\| \geq \hat{\omega}_n(p_{\dd}) \geq \pi\frac{\|u_{0,n}^{i}\|^2}{\|\omega_n\|}-\frac{\bra u_{0,n}^{i},p_\Delta\ket}{\|\omega_n\|}
$$
and if $\|u_{0,n}^{i}\|$ diverges, the last term of the previous inequality is asymptotic to its first addendum, which turns out to be bounded. 
Then for the corresponding component of the horizontal part of the end-point we have that:
$$
\|p_{\Delta,n}^{i}\|^2 = \frac{2}{\alpha_{i,n}^{2}}\|u_{0,n}^{i}\|^2(1-\cos(2\pi \alpha_{i,n})) \leq 2\frac{\|u_{0,n}^{i}\|^2}{\|\omega_n\|^2}.
$$
If $\|u_{0,n}^{i}\|$ is bounded the last term of the inequality converges to 0; if it diverges we have that
$$
\|p_{\Delta,n}^{i}\|^2 \leq 2\frac{\|u_{0,n}^{i}\|^2}{\|\omega_n\|^2} \leq \textrm{const.} \frac{1}{\|\omega_n\|}
$$
which again converges to $0$. 
So, in the limit, the components of the horizontal part of the end-point are orthogonal to the eigenspaces of the limit matrix $\lambda A$ with non-zero eigenvalues. 

Now, if the matrix $\lambda A$ is not singular the horizontal part has to be zero contradicting the hypotesis; thus in order to end the argument, we will prove that the horizontal part of the generic  point does not belong to the kernel of any degenerate matrix in the span $W$. 

Let us consider the semialgebraic set of the couples of degenerate matrices together with the vectors in their kernels, namely $$X\doteq\{(\eta,x) \in S^{l-1}(W) \times \Delta \ \mid \ \eta Ax=0,\ \dim \ker \eta A \geq 2\}$$ stratified in subsets with constant dimension of the kernel.  Assuming that the Carnot algebra structure defined by $W$ is transversal to every subset of the stratification described in the Appendix \ref{app:sodstrat} (condition satisfied by the generic $W$), the stratum $W_k$ of the matrices of $k$-dimensional kernel has codimension $k\frac{k-1}{2}$.
Therefore the semialgebraic set $X$ has dimension not greater than $\textrm{max}_{k\geq 2} l-1 - k\frac{k-1}{2}+k \leq l$.  If we project $X$ on $\Delta$ we get the set $\tilde{X}$ of points belonging to the kernel of a degenerate matrix of $W$: this is still a semialgebraic set and the projection does not increase the dimension. 
In particular the set 
$$Y=\{(p_\Delta, p_{\dd})\in G\, |\, p_\Delta\in \ker \eta A\quad \textrm{for some $\eta\neq 0$}\}$$
has dimension less than $2l$, which is smaller than $d+l=\dim (G)$ under the assumption $d>l$. This concludes the proof.\end{proof}

\begin{coro}
If $p \notin \dd$ the homology of $\Omega_{p} \cap \{ J \leq s \}$ stabilizes as $s \to \infty$.
\end{coro}

\subsection{Geodesics ending at $p\in \dd$}
In this section we study in more detail the case $p\in\dd$. To start with, we prove that for the generic choice of $p\in \dd$ the set $\Omega_p$ is a nice object, more precisely we have the following theorem.

\begin{teo}\label{regularvalue}For the generic choice of $W\subset \sod$ and a generic $p\in \dd$ the topological space $\Omega_p$ is a Hilbert manifold.
\end{teo}
\begin{proof}
We will prove that the set of critical values of $q=F|_{F^{-1}(\dd)}$ is contained in a semialgebraic subset of codimension one, from this the conclusion follows.

We first notice that the condition for $p\in \dd$ to be a critical value of $q$ is that there exist $u\in H$ and $\omega\in {(\dd)}^*$ such that $q(u)=p$ and $dq_u=2\omega Qu=0.$ In particular if $u=\sum_{k=1}^{\infty}u_k$, where $u_k\in T_k\simeq \R^{2d}$, then $q(u)=\sum_{k=1}^{\infty}\langle u_k,\frac{1}{k}(\omega P)u_k \rangle$ and the condition that $\omega Qu=0$, by invariance of the spaces $T_k$, reads $\omega P u_k=0$ for all $k\geq 1$.

Let us consider the stratification $\sod=\coprod S_r$, where $S_r$ is the set of matrices with constant rank $r$. Each $S_r$ is smooth and over it we have the smooth bundle $K_r=\{(A, v)\in \sod \times \R^n\,|\, Av=0\}$. Since this stratification is homogeneous, then the generic $W$ is transversal to all strata and we have an induced stratification:
$$W=\coprod_{r=0}^d W_r, \quad W_r=S_r\cap W.$$
Consider now the vector bundle $K|_{W_r}$ over $W_r$ (the restriction of $K_r$); notice that for every $\omega A \in W$ we have $\textrm{ker}(\omega P)=\{(x,y)\in \R^{2d}\,|\, \omega Ax=\omega Ay=0\}.$ In particular a smooth section of $K|_{W_r}$ produces also a smooth section of $\{(\omega, z)\in W_r\times \R^{2d}\,|\, \omega P z=0\}$ over $W_r$.

We notice now the following interesting property: if $p\in \dd$ is a critical point for $q$ with Lagrange multiplier $\omega\in W_r$, then $p\in {(T_\omega W_r)}^{\perp}.$ In fact for every $k\geq 1$ let us consider a smooth curve $\omega(t)$ in $W_r$ with $\omega(0)=\omega$, $\dot{\omega}(0)=\eta\in T_\omega W_j$ and a smooth $z(t)\in \ker (\omega(t)\omega P)$ with $z(0)=u_k$ such that $$\omega(t)\omega Pz(t)=0,$$ (the existence of such a smooth $z(t)$ follows from the above discussion). Then deriving the above equation we get $\dot{\omega}(0)\omega Pz(0)+\omega(0)\omega P\dot{z}(0)=0$ and considering the scalar product with $z$ gives:
$$\langle z, \dot{\omega}(0)\omega P z\rangle=\dot{\omega}(0)(q(u_k))=0$$
which tells $\eta(p)=\sum_{k=1}^{\infty}\eta(q(u_k))$ vanishes for every $\eta$ in $T_\omega W_r$. We consider now the semialgebraic set:
$$\Sigma'_1=\bigcup_{r=0}^d\big( \bigcup_{\omega\in W_r}(T_\omega W_r)^{\perp}\big).$$
Because of the above argument all critical points of $q$ are contained in $\Sigma_1'$ and we want to show this set is of dimension strictly less than $l$.

We first check that $\Sigma'_1$ is indeed semialgebraic (and as a result we compute its dimension); being a finite union, it is enough to prove that each $\bigcup_{\omega \in W_r}(T_\omega W_r)^\perp$ is semialgebraic. To this end consider $T_r=\{(\omega, \eta)\in W_r\times W^*\,|\, (\eta \in T_\omega W_r)^\perp\}$, which is clearly semialgebraic, and the semialgebraic projection $\pi_2:T_r\to W^*$ to the second factor. The image of $\pi_2$ is semialgebraic and coincides with $\bigcup_{\omega \in W_r}(T_\omega W_r)^\perp$.

Now, each $\bigcup_{\omega\in W_r}(T_\omega W_r)^{\perp}$ has dimension less than $\dim (W_r)+l-\dim(W_r)-1\leq l-1$, where the $-1$ comes from the fact that by homogeneity $(T_\frac{\omega}{|\omega|}W_r)^{\perp}=(T_\omega W_r)^\perp$ (notice in particular that the stratum of maximal dimension is open and produces only the zero, since the orthogonal complement of its tangent space is the zero only). 

In particular the dimension of $\Sigma'_1$ is strictly less than $l$ and the generic $p\in \dd$ is a regular value of $q$.\end{proof}

Following up the discussion after Proposition \ref{propo:multiplier}, we see that in the case the final point of $\gamma_u$ is in $\dd$, which we know it is equivalent to $\int_I u=0$, we can apply the Lagrange multiplier rule to the map $q.$ More precisely $u$ is the control associated to a curve which is a geodesic with endpoint $p\in \dd$ if:
$$q(u)=p\quad \textrm{and there exists $\omega$ such that $\omega Q u=u$}.$$
The covector $\omega \in (\dd)^*$ is called the \emph{Lagrange multiplier} associated to $u$. Using this remark we see that $u$ is a geodesics with Lagrange multiplier $\omega$ iff:
$$\langle u, v\rangle_H=\langle \omega Qu, v\rangle_H \quad \textrm{for all $v\in H$}$$
(here the final point is not specified, i.e. we are considering all possible geodesics with Lagrange multiplier $\omega$; the final point is recovered by simply applying the expression given in (\ref{endpointexplicit}) to $u$).
Thus in particular we have that for all $v$ in $H=\{\int v=0\}$:
$$\int_0^{2\pi}\langle u(t), v(t)\rangle dt=-\int_0^{2\pi}\langle \omega AU(t), v(t)\rangle dt=0,$$
where $U(t)=\int_{0}^t u(s) ds.$ The previous condition tells that $u+\omega A U$ is a constant function, or equivalently that $\dot{u}=-\omega A u.$ This implies that $u$ must be of the form:
$$u(t)=e^{-t (\omega A)}u_0,$$
and since $u\in H$, then $u_0$ must be in the \emph{integer} eigenspace of $i \omega A, $ i.e.:
$$u_0=e^{-2\pi \omega A}u_0.$$ 
In particular notice that in this case the complete Lagrange multiplier (i.e. the one arising by using the map $F$, as in Proposition \ref{propo:multiplier}), is $\lambda= (u_0,\omega)$. 

We collect the result for a geodesic ending at $p\in\dd$ in a lemma.
\begin{lemma}\label{endpoint} Let $u$ be the control associated to a geodesic whose final point is in $\dd$ with Lagrange multiplier $\omega$. Then:
$$u(t)=e^{-t \omega A}u_0\quad \textrm{with} \quad u_0=e^{-2\pi \omega A}u_0.$$
\end{lemma}

Motivated by the previous lemma, for every $\omega\in W$ we define:
$$E(\omega)=\{v\in \R^d\,|\, e^{-2\pi \omega A}v=v, \, v\notin \ker(\omega A)\}.$$
Thus $E(\omega)$ is the set of possible initial data for \emph{non constant} geodesics with Lagrange multiplier $\omega$.
In particular we see that in order to have a nonzero initial datum the matrix $i\omega A$ must have nonzero integer eigenvalues, thus the set of all possible Lagrange multipliers coincides with the set:
$$\Lambda=\{\omega \in (\dd)^*\,|\, \det(\omega A-i n \mathbbm{1})=0\quad \textrm{for some $ n\in \mathbb{N}_0$}\}.$$
Notice that $\Lambda$ \emph{is not} an algebraic (or a semialgebraic set): it is indeed given by the infinite union of algebraic sets $\Lambda_n=\{\det(\omega A-i n \mathbbm{1})=0\}$. However $\Lambda$ is locally algebraic: if we intersect it with a ball, then only a finite number of $\Lambda_n$ show up.

We discuss now in more detail the structure of the set:
$$E=\{(\omega, v)\in W\times \R^d\,|\, v\in E(\omega)\}.$$
As for $\Lambda$, this set \emph{is not} semialgebraic, although if we take the ``restriction'' $E|_B$ to a compact semialgebraic set $B$, i.e. we only allow $\omega$ to vary on a compact semialgebraic set $B\subset W$, then $E|_B$ becomes semialgebraic.

First for every $\omega$ let us consider the canonical skew-symmetric form of $\omega A$:
$$M(\omega)^T(\omega A)M(\omega)= \textrm{Diag}(\alpha_1(\omega)J_2, \ldots, \alpha_s(\omega)J_2, 0, \ldots, 0)$$
where $M(\omega)$ is an orthogonal matrix, $\alpha_1(\omega), \ldots, \alpha_s(\omega)$ are the \emph{positive nonzero} eigenvalues of $i\omega A$ and $J_2\in \mathfrak{so}(2)$ is the canonical symplectic matrix. Thus $\R^d$ decomposes as the orthogonal sum $\R^d=V_1\oplus\cdots\oplus V_s\oplus K$, where the $V_i$s are the coordinate two planes and $K$ is the vector space of the last $d-2s$ coordinates. Using this notation we set $V_i(\omega)=M(\omega)V_i$: it is the invariant subspace of $\omega A$ associated to the eigenvalue $\alpha_i(\omega).$ In particular we see that:
$$E(\omega)=\bigoplus_{\alpha_i(\omega)\in \mathbb{N}_0}V_i(\omega).$$

Associating to each $v\in E(\omega)$ the control $e^{t\omega A}v$ defines a linear injection of $E(\omega)$ into $H=\oplus_{k\geq 1} T_k$; in particular the previous curve admits the Fourier series decomposition 
\begin{equation}\label{formula}e^{-t\omega A}v=\sum_{k\geq 1}X_k(v)\coskt - Y_k(v) \sinkt\end{equation}
and we denote by $\phi_k$ the linear map $v\mapsto (X_k(v), Y_k(v))$ (the $k$-th component of the Fourier series of $e^{t\omega A}v$ written in coordinates $T_k\simeq \R^{2d}$).

Let now $v\in V_{i}(\omega)$ with $\alpha_i(\omega)=k\in \mathbb{N}$; in order to get the expression for $\phi_k(v)$ we compute the Taylor series of $e^{\omega At}v$. We have:
\begin{align*}
e^{-t\omega A}v & =\left(\sum_{n\geq0}\frac{(\omega A)^n t^n}{n!}\right)v=\left(\sum_{m\geq 0}\frac{(\omega A)^{2l} t^{2l}}{(2l)!}\right)v-\left(\sum_{m\geq 0}\frac{(\omega A)^{2l+1} t^{2l+1}}{(2l+1)!}\right)v=  \\
 & =\left(\sum_{m\geq 0}\frac{(-1)^l k^{2l} t^{2l}}{(2l)!}\right)v-\left(\sum_{m\geq 0}\frac{(-1)^l k^{2l+1} t^{2l+1}}{(2l+1)!}\right)\frac{\omega A}{k}v= \\
 & =v \cos kt-\frac{\omega A}{k}v \sin kt
\end{align*}
where in the second line we have used the fact that $(\omega A)^2v=-k^2v$ (being $V_i(\omega)$ the space associated to the eigenvalue $\alpha_i(\omega)=k$). This computation implies that:
$$\phi_k(v)=\sqrt{\pi}\begin{pmatrix}
v \\
-\frac{\omega A}{k}v 
\end{pmatrix}.$$

Notice that the same construction can be performed using the linear immersion $v\mapsto e^{t\omega A}v$, which gives the above control with \emph{backward} time; the Lagrange multiplier for the corresponding geodesic is $-\omega$ and the final point is $-q(e^{-t\omega A}v).$

Slightly abusing of notation, we will still denote by $q$ the map obtained by composing the endpoint map with the linear immersion $E(\omega)\hookrightarrow H$.

We recall from the Appendix (see Proposition \ref{stratiappendix}) that the lie algebra $\sod$ of skew-symmetric matrices of size $d$ is stratified by the sets $\Gamma_{k|m_1, \ldots, m_r}$ consisting of those matrices $A$ satisfying: $\dim \ker (A) =k;$ the numbers $m_1, \ldots, m_r$ are natural non-increasing (they are the multiplicities in the positive spectrum of $iA$). Each one of these strata is smooth and has codimension $\sum_{i=1}^r(m_i^2-1)+\frac{k(k-1)}{2}.$ Since (by construction) this stratification is homogeneous, then a generic choice of the Carnot structure $W\subset \sod$ will be transversal to all of the strata and will inherit the stratification (in particular respecting the codimensions and smoothness).

To deal with integer eigenvalues we need to refine this stratification, unfortunately ending up with an infinite number of strata, but still with nice properties. More specifically for every $r\geq 0$ we consider $\vec{n}=(n_1, \ldots, n_r)\in \mathbb{N}^r$ with distinct nonzero components and define the \emph{semialgebraic} set $\Gamma_{k|m_1, \ldots, m_r|\vec{n}}$ as follows: we look at the nonzero components of $\vec{n}$, say $n_{j_1}, \ldots, n_{j_{\nu}}$, and we take those matrices in $\Gamma_{k|m_1, \ldots, m_r}$ such that the eigenvalue with multiplicity $m_{j_1}$ equals $n_{j_1}$, the one with multiplicity $m_{j_2}$ equals $n_{j_2}$, and so on.\\For example $(0,2,2,0)\in \mathbb{N}^4$ is not an admissible $\vec{n}$ (since there are two equal nonzero entries); on the other hand if $\vec{n}=(0, 1,2,0)$ then $\Gamma_{k|m_1, m_2, m_3, m_4|\vec{n}}$ equals the set of all matrices in $\sod$ with multiplicities of the spectrum $ \{m_1, m_2, m_3, m_4\}$ \emph{and} with one eigenvalue equal to $i$ (the imaginary unit) with multiplicity $m_2$ and another equal to $2i$ and with multiplicity $m_3$.\\
This operation of fixing some eigenvalues to some integer numbers increases the codimension by $\nu$ (the number of nonzero components of $\vec{n}$).

Using this notation we see that one can stratify $\Lambda$ as:

$$\Lambda=\coprod_{r} \coprod_{\{k, m_1, \ldots, m_r\}}\,
\coprod_{\{\vec{n}\in \mathbb{N}^r\,\textrm{admissible}\}} \Gamma_{k|m_1, \ldots, m_r|\vec{n}}\cap W.$$
As we already noticed, this stratification is not finite even though each stratum is semialgebraic. Neverthless if we intersect $\Lambda$ with a compact ball $B\subset W$ only a finite number of the above strata appear and we are locally semialgebraic.

The next lemma tells that for the generic choice of $W\subset \sod$ and a generic $p\in \dd$, the Lagrange multipliers have simple integer spectrum.

\begin{lemma}\label{nomultiple} For a generic Carnot group structure $W\subset \sod$, the generic $p\in \dd$ is not the final point of a geodesic with lagrange multiplier $\omega$ such that $\omega A$ has multiple eigenvalues in $i\mathbb{Z}$.
\end{lemma}

\begin{proof}
First we pick the structure $W$ to be transversal to all strata of the first one of the above stratifications, the one using only the multiplicities in the spectrum (and we know such a property is generic). We stratify now the set $\Lambda$ by intersecting it with the different strata $\Gamma_{k|m_1, \ldots, m_r}$; we are interested only in those strata for which there is at least a multiple integer eigenvalues and we refine the stratification to the above infinite one, by indexing with the admissible $\vec{n}\in \mathbb{N}^r$.

Thus we let $\Lambda_{\{m_j\geq 2\}}$ be one stratum $\Gamma_{k|m_1, \ldots, m_r|\vec{n}}$ such that $m_j\geq 2$ for at least one index $j$ with $n_j\neq 0.$ Each $\Lambda_{\{m_j\geq 2\}}$ obtained in this way has codimension:
$$\textrm{codim}_{W}\Lambda_{\{m_j\geq 2\}}=\sum_{i=1}^r(m_i^2-1)+\frac{k(k-1)}{2}+\nu.$$

We consider now as above the set $E=\{(\omega, v) \in \Lambda\times \R^d\,|\, e^{2\pi \omega A}v=v, \, v\notin \ker(\omega A)\}$. Over each stratum $\Lambda_{\{m_j\geq 2\}}$ the set $E|_{\Lambda_{\{m_j\geq 2\}}}$ is a smooth vector bundle (it is the restriction to $\Lambda_{\{m_j\geq 2\}}$, which is smooth, of a smooth vector bundle); moreover $E|_{\Lambda_{\{m_j\geq 2\}}}$ is semialgebraic as well (here the vector $\vec{n}$ is fixed). 

Consider the smooth map:
$$f:E|_{\Lambda_{\{m_j\geq 2\}}}\to \dd$$
defined by $(\omega, v)\mapsto q(v)$, where $q(v)$ is the final point of the geodesic associated to the control $v(t)=e^{t\omega A}v.$ We compute the rank of the differential of $f$ and show that the assumption $m_j\geq 2$ implies this rank is less than $l-1$; in particular the image of $f$ has measure zero. Since the set of final points of geodesics with Lagrange multipliers with multiple eigenvalues in $i\mathbb{Z}$ is the countable union of the images of the different $f$ obtained as $\vec{n}$ varies over $\mathbb{N}^r$, the result follows.

The differential of $f$ restricted to the base $\Lambda_{\{m_j\geq 2\}}$ has rank smaller than the dimension of $\Lambda_{\{m_j\geq 2\}}$, which is $l-\sum_{i=1}^r(m_i^2-1)-\frac{k(k-1)}{2}-\nu$.
For the rank of $f$ restricted to the fibers we argue as follows. For every $\omega\in \Lambda_{\{m_j\geq 2\}}$ we consider the invariant subspaces of $\omega A$; for each natural nonzero eigenvalue $\lambda_j(\omega)$ of $i\omega A$ we find an invariant space $V_j(\omega)$ (the real part of the $\lambda_j(\omega)$-eigenspace of $i \omega A$) of dimension $2\mu_j$, twice the multiplicity of $\lambda_j(\omega)$; let's call $I\subset\{1, \ldots, r\}$ the index set for such spaces $V_j(\omega)$ (notice that $I=\{j_1, \ldots, j_{\nu}\}$).

The restriction of $f$ to each such $V_j(\omega)$ maps $\mu_j$ unit circles (lying on distinct orthogonal planes) to a point, in particular the dimension of the kernel of the differential of $f$ on each $V_{j}(\omega)$ is at least $\mu_j$. Since the dimension of $E(\omega)$ is  $2\sum_{j\in I}\mu_j$, we see that the rank of the differential of $f$ on the fibers is at most $\sum_{j\in I}\mu_j$.\\
In particular we can bound the rank of the differential of $f$ as:
\begin{align*}\textrm{rk} (df)&\leq l-\nu-\sum_{j=1}^r(m_j^2-1)-\frac{k(k-1)}{2}+\sum_{j\in I}\mu_j\\
&\leq l-\nu-\sum_{j\in I}(m_j^2-1-m_j)-\frac{k(k-1)}{2}\\
&\leq l-\nu-\sum_{j\in I, \, m_j\geq 2}(m_j^2-1-m_j)-\sum_{j\in I,\, m_j=1}(m_j^2-1-m_j)\\
&\leq l-\nu-\sum_{j\in I, \, m_j\geq 2}1-\sum_{j\in I,\, m_j=1}(-1)\\
&\leq l-\nu-1+(\nu-1)<l-1.\end{align*}

\end{proof}

We define now the set $\Sigma_2\subset \dd$ to be the union of the various $f(E|_{\Lambda_{k|m_1, \ldots, m_r, \vec{n}}})$ where $\vec{n}=(n_1, \ldots n_r)\in \mathbb{N}^r$ is admissible and $m_j\geq 2$ for at least one index $j$ with $n_j\neq 0$. The above lemma says that $\Sigma_2$ is the countable union of semialgebraic sets of codimension at least $2$ (in particular, for example, $\Sigma_2$ has measure zero).

We study now what happens for a $p\in \dd\backslash (\Sigma_1\cup \Sigma_2)$ (because of the above argument such $p$ is generic). Such a $p$ has Lagrange multipliers $\omega$ with simple spectrum, i.e. $\omega A$ belongs to a stratum $\Lambda_{k|m_1, \ldots, m_r| \vec{n}}$ with all multiplicites equal to $1$; for simplcity of notation we omit the string of multiplicites and denote such stratum simply by $\Lambda_{\vec{n}}$. In other words $\Lambda_{\vec{n}}$ is one of the above strata where all eigenvalues are distinct and we have fixed $\nu$ of them to be equal to $in_{j_1}, \ldots, in_{j_\nu}$ (the nonzero entries of $\vec{n}$).

\begin{lemma}\label{distinct}
Let $\omega\in \Lambda_{\vec{n}}$ and $n_1, \ldots, n_\nu$ be the nonzero eigenvalues of $i\omega A$ in $\mathbb{N}$; for $j=1, \ldots, \nu$ let also $V_{j}(\omega)$ be the (two dimensional) invariant subspace of $\omega A$ associated to $n_j$. Then $E(\omega)$ splits as the direct orthogonal sum:
$$E(\omega)=\bigoplus_{j=1}^{\nu}V_{j}(\omega).$$
Moreover the image of $q|_{V_{j}(\omega)}$ is a half line $l_j^+(\omega)$ and:
$$\emph{\textrm{im}}(q|_{E(\omega)})=\emph{\textrm{cone}}\{l_1^+(\omega), \ldots, l_{\nu}^+(\omega)\}.$$
\end{lemma}

\begin{proof}
Recall that the space $E(\omega)$ is defined to be $\{v\in \R^d\,|\, e^{-2\pi \omega A}v=v, \, v\notin \ker (\omega A)\}$; the map that associates to a vector $v\in E(\omega)$ the curve $e^{-t \omega A}v$ defines an embedding of $E(\omega)$ into $H$ and if $v\in V_{j}(\omega)$ then the resulting control must be a linear combination of $\sin (n_j t)$ and $\cos( n_j t)$; in particular $V_{n_j}(\omega)\subset T_{n_j}$. Since the $T_k$ are pairwise orthogonal for each operator $Q_1, \ldots, Q_l,$ then decomposing $v\in E(\omega)$ into its pieces $v=v_{1}+\cdots+ v_\nu$ with $v_j\in V_{j}(\omega)$, we get:
$$q(v)=q(v_1)+\cdots+q(v_\nu),$$
which proves the image of $q|_{E(\omega)}$ is the cone spanned by $\{q(V_{1}(\omega)), \ldots, q(V_{\nu}(\omega))\};$ the orthogonality of the $V_{j}$ follows from the one of the $T_{n_j}.$

It remains to prove that the image of $q|_{V_{j}(\omega)}$ is a half line. By assumption $\omega$ belongs to a smooth stratum of codimension $\nu$ in $W$ andrecalling the definition of  $\Lambda_{n_j}= \{\det(\omega A-in_j\mathbbm{1})=0\}$, we have that:
$$\Lambda_{\vec{n}}=\bigcap_{j=1}^r\Lambda_{n_j}\cap \Gamma_{k|1, \ldots, 1}.$$
The bundle $\coprod_{\omega\in \Lambda_{n_j}}V_{j}(\omega)$ is smooth (being the restriction of a smooth bundle). In particular for every $\eta\in T_\omega\Lambda_{n_j}$ there are curves $\omega(t)\in \Lambda_{n_j}$ and $v(t)\in V_{j}(\omega(t))$ such that $\omega(0)=\omega$, $\dot{\omega}(0)=\eta$ and $v(0)=v$. Deriving the equation $\omega(t)\omega Pv(t)=v(t)$ and taking inner product with $v$ we get:
$$0=\left\langle v,\frac{\dot{\omega}(0)\omega P}{k}v\right\rangle=\eta(q(v))$$
which tells the final point of the geodesic $u$ associated to $e^{\omega A t}v$ is orthogonal to $T_\omega \Lambda_{n_j}$, hence is it contained in a line. On the other hand since $\omega$ is the Lagrange multiplier for the geodesic $u$, we get $\omega(q(u))=J(u)>0,$ which concludes the proof.
\end{proof}

Everything now is ready for the proof of the Theorem that describes the structure of geodesics.

\begin{teo}\label{structure}For a generic Carnot group structure $W\subset \sod$ and a generic $p\in \dd$ the set $\Lambda(p)$ of Lagrange multipliers of geodesic whose final point is $p$ is discrete. Moreover every $\eta\in \Lambda(p)$ belongs to some $\Lambda_{\vec{n}}$ and the set of all geodesics whose endpoint is $p$ with Lagrange multiplier $\eta$ is a compact manifold of dimension $\nu\leq l$ ($\nu$ is the number of nonzero entries of $\vec{n}$) diffeomorphic to the torus $ \underbrace{S^1 \times \cdots \times S^1}_{\nu \textrm{ times}}$.

\end{teo}

\begin{proof}
By Lemma \ref{nomultiple} we know that for the generic choice of $W\subset \sod$ the generic $p\in \dd$ is a final point only of geodesics with Lagrange multipliers in $\Lambda_{\vec{n}}$ for some $\vec{n}\in \mathbb{N}^{\lfloor d/2\rfloor}$.\\
Moreover for every $\eta$ in $\Lambda(p)$ the set of all geodesics with Lagrange multipliers $\eta$ and final point $p$ is the preimage of $p$ under the map $q:E|_{\Lambda(p)}\to \dd.$

For every admissible $\vec{n}\in \mathbb{N}^{\lfloor d/2\rfloor}$ we consider the semialgebraic set $F_{\vec{n}}$ defined by:
$$F_{\vec{n}}=\{(\omega, p)\in \Lambda_{\vec{n}}\times \dd\,|\, p\in \textrm{im}(q|_{E(\omega)})\}$$
together with the semialgebraic map $g:F_{\vec{n}}\to \dd$ defined by $(\omega, p)\mapsto p$. Since each point $(\omega, p)$ in $F_{\vec{n}}$ has $\omega$ in $\Lambda_{\vec{n}}$, then by Lemma \ref{distinct} the dimension of $F_{\vec{n}}$ is at most $l$ . In fact $\omega$ varies on a set of dimension $l-\nu$ and the image of $q|_{E(\omega)}$ is a cone of dimension at most $\nu.$  Since $F_{\vec{n}}$ is semialgebraic we stratify it as $F_{\vec{n}}=\coprod_{j=1}^s F_{{\vec{n}}, j}$, where each stratum is smooth semialgebraic of dimension at most $l$ (in fact here the index $s$ depends on $\vec{n}$ as well, but we omit this dependence to simplify notations). Notice that if $(\omega, p)$ belongs to a stratum of maximal dimension $l$, then the cone $q(E(\omega))$ must have maximal dimension $\nu$ and $p$ must be in its interior.

The restriction $g_{{\vec{n}}, j}=g|_{F_{{\vec{n}}, j}}$ is smooth semialgebraic, thus by the semialgebraic Sard's lemma (see \cite{BCR}) the set $C_{{\vec{n}}, j}$ of its critical values is a semialgebraic set of dimension at most $l-1$. If $p$ is not one of these critical values then $g_{\vec{n}, j}^{-1}(p)$ consists of isolated points if $\dim(F_{\vec{n}, j})=l,$ and is empty otherwise.

We set $\Sigma_3$ to be the union of the critical values of $g_{\vec{n}, j}$ ($\vec{n}$ varies over $\mathbb{N}^{\lfloor d/2\rfloor}$ and $j$ is the stratifying index for $F_{\vec{n}}$ as above); such a union, being a countable union of semialgebraic set of dimension at most $l-1$ has measure zero, hence points belonging to its complement are generic.

On the other hand $\Lambda(p)$ equals the union of the projections on $\Lambda$ of the various $g_{\vec{n}, j}^{-1}(p)$. If we intersect $\Lambda$ with a compact ball $B$, we hit only a finite number of strata $\Lambda_{\vec{n}}$ and $\Lambda(p)\cap B$ is dsicrete; thus for a generic $p$ the set $\Lambda(p)$ is discrete set (possibly infinite). 

From Lemma \ref{distinct} we recall that $q(E(\omega))$ is the cone spanned by the half-lines $l^{+}_{j}(\omega)=q(E_j(\omega))$: moreover $p$ is the sum of nonzero vectors belonging to these half-lines, $p=p_1+\ldots+p_\nu$ with $p_j \in l^{+}_{j}(\omega)$. Since the spaces $E_j(\omega)$ are orthogonal with respect to the operators $J,Q_1,\ldots,Q_l$, the condition $q(u)=p$ with $u \in E(\omega)$ can be split up as $q(u_j)=p_j$ with $u_j \in E_j(\omega)$. The condition $q(u_j)=p_j$ is equivalent to $\omega q(u_j)=\omega(p_j)$ since for every covector $\eta$ orthogonal to $p_j$ the condition $\eta q(u_j)=\eta(p_j)$ is automatically satisfied; on the other hand by the Lagrange Multiplier condition we have that $\omega q(u_j)=J(u_j)$, so that the condition is a positive definite one, $J(u_j)=\omega (p_j)$. This implies that every component $u_j \in E_j(\omega)$ of the geodesic $u$ going to $p$ is constrained on a circle $S^1 \in E_j(\omega)$, from which follows that the critical manifold $C_\omega$ is a $\nu$-dimensional torus. 
\end{proof}

Up to now the genericity assumptions for $p$ come from Theorem \ref{regularvalue}, Lemma \ref{nomultiple} and Theorem \ref{structure}: specifically we require $p\notin \Sigma_1\cup \Sigma_2\cup \Sigma_3$ (where $\Sigma_3$ is defined in the proof of Theorem \ref{structure}).

% MORSE-BOTT THEORY --------------------------------------------------

\subsection{Morse-Bott theory}
Since critical points of the Energy functional $J$ are far from being isolated (they arrange themselves into compact manifolds) we cannot apply Morse Theory in its standard version. What we need is a generalization of it called \emph{Morse-Bott Theory}: it allows to prove the same results as for the ordinary theory if in the definitions nondegenerate critical points are replaced by \emph{nondegenerate critical manifolds} (see \cite{Bott1, Klingenberg}; the basic definitions and results are recalled in the Appendix \ref{app:morsebott}). 

\begin{teo} \label{morsebottpalaissmale}
For a generic Carnot Group structure $W \subset \sod$ and a generic point $p \in \dd$, the Energy functional $J$ restricted to $\Omega_{p}$ is a Morse-Bott function. 
\end{teo}

\begin{proof}
We assume all the genericity conditions of the previous theorems to be satisfied. We need to prove that we can possibly restrict the set of ``good'' final points $p$ to a smaller (but still dense) set for which $J|_{\Omega_p}$ is Morse-Bott.

Part (a) of the definition of a Morse-Bott function given in the Appendix \ref{app:morsebott}, immediately follows from Theorem \ref{structure}.

We proceed to prove part (b). Let us take a Lagrange multiplier $\eta$, its corresponding critical manifold $C_{\eta}$ and a point $u \in C_{\eta}$. Since we already know that the critical manifold is compact, it remains to show that the Hessian of the energy $J$ is non-degenerate outside the tangent space to the critical manifold $C_{\eta}$. Since both $q$ and $J$ are quadratic, they coincide with their second derivatives. By the Lagrange multiplier rule we get the expression for the Hessian:
$$
\he_{u}(q)=(d^{2}J-\eta D^{2}q_{u})|_{T_u\Omega_p}=(J-\eta q)|_{T_u\Omega_p},
$$
where we have $T_u\Omega_p=\ker D_{u}q$. Notice that to the quadratic form defined by the Hessian it corresponds the self-adjoint operator $
\mathbbm{1}-\eta Q.$

The Hessian is degenerate in the direction of $v \in T_{u}\Omega_{p}$  if and only if:
$$
\langle v-\eta Q v,x \rangle=0 \quad \quad \forall x \in T_{u}\Omega_{p},
$$
meaning that $v-\eta Q v$ is orthogonal to $T_{u}\Omega_{p}$. Since the tangent space $T_{u}\Omega_{p}$ is the orthogonal space to $\textrm{span}\{Q_{1}u,\ldots,Q_{l}u\}$, then $v-\eta Q v$ is a linear combination of the vectors $Q_{i}u$, namely $\lambda_{1}Q_{1}u + \ldots +\lambda_{l}Q_{l}u$. 

We can eventually restate the degeneracy condition by the following equations:
\begin{equation} \label{eq:degeneracy}
\left\{ \begin{array}{c} v-\eta Q v = \lambda Qu \\
         \langle v,Qu \rangle=0 
         \end{array}
\right.
\end{equation}
where $\lambda=(\lambda_{1},\ldots,\lambda_{l}) \in (\dd)^{*}$ as above.

If we take $\lambda = 0$ we see that the degeneracy condition is satisfied by the vectors in:
$$
E(\eta) \cap T_{u}\Omega_{p} = T_{u}C_{\eta},
$$
and we have to prove that for the generic choice of $p$ the degeneracy equation (\ref{eq:degeneracy}) does not admit other solutions.

Let us consider the smooth manifold $\Lambda_{\vec{n}}$ of Lagrange multipliers containing $\eta$ (the definition of $\Lambda_{\vec{n}}$ is given before Lemma \ref{distinct}). 
Let us call as before $E_{\vec{n}}$ the fiber bundle with base space $\Lambda_{\vec{n}}$ and fiber $E(\omega)$ with $\omega \in\Lambda_{\vec{n}}$ (see the above discussion).

The tangent space to $E_{\vec{n}}$ at $(u,\eta)$ is determined as follows: take a curve $(u(t),\eta(t))$ in $E_{\vec{n}}$ based on $(u,\eta)$ and compute its tangent vector in $t=0$. Differentiating the condition $
\eta(t)Qu(t)=u(t),$
we get:
$$
\dot{u}-\eta Q\dot{u}=\dot{\eta}Qu,
$$
which is the same condition as for the degeneracy of the Hessian (the first equation in (\ref{eq:degeneracy})).

We consider now the smooth semialgebraic map: $$f:E_{\vec{n}}\to \dd\quad\textrm{given by}\quad (\omega, v)\mapsto q(v).$$ 
The set of regular values of $f$ is a dense subset of $\Delta^2$ (it is the complement of a semialgebraic set of codimension at least one): this subset is the good one we want to restrict to. In other words we consider $\Sigma_4$ to be the union of the set of critical values of the various $f:E_{\vec{n}}\to \dd$ as $\vec{n}$; the complement of $\Sigma_4$ contains generic points\footnote{Thus at this stage $p\in \dd\backslash (\Sigma_1\cup\Sigma_2\cup\Sigma_3\cup \Sigma_4).$}. 
On the preimage of a ``good'' $p$ we know that the differential of $f$ is surjective with rank $l$; moreover:
$$
\textrm{dim}\, E_{\vec{n}}=\textrm{dim}\,\Lambda_{\vec{n}}+\textrm{dim}\,E(\eta)=l-\nu+2\nu=l+\nu.
$$
Looking at the dimensions of the domain and the range of $d_{u}f$ we get that the kernel has dimension $\nu$. On one hand, the kernel of $d_{u}f$ is the vector space satisfyng both the equations (\ref{eq:degeneracy}) for the degeneracy of the Hessian; on the other hand we have the inclusion:
$$
E(\eta) \cap T_{u}\Omega_{p} \subset \ker d_{u}f.
$$
Since both these spaces have dimension $\nu$, they must be equal. It follows that with all the above generic restrictions on $p,$ the only directions of degeneracy for the Hessian are in $T_{u}C_{\eta}$.

\begin{remark} \label{finiteindex}
We know that the operator $\eta Q$ is compact, and that the eigenvalues are of the form $\pm\frac{\alpha_{i}}{k}$ with $k \in \mathbb{N}$ non-zero and $i=1,\ldots,s < \infty$. Then the eigenvalues of the Hessian of the energy on a critical point are $1\pm\frac{\alpha_{i}}{k}$; it follows that the number of negative eigenvalues is always finite, so that the index of every critical manifold is finite.
\end{remark}

It remains to prove property (c) of the definition (the Palais-Smale property). Let us consider a sequence $\{u_{k}\}$ in $\Omega_{p}$ with energy $\|u_{k}\|^{2}$ bounded by $E$ and such that $\nabla \psi_{u_{k}} \rightarrow 0$, where $\psi\doteq J|_{\Omega_{p}}$.

The gradient $\nabla \psi_{u}$ is the orthogonal projection of $\nabla J_{u}=u$ on the space $\textrm{span}\{Q_{1}u,\ldots,Q_{l}u\}^{\perp},$ then if we define $\pi_{u}$ to be the orthogonal projection on the space $\textrm{span}\{Q_{1}u,\ldots,Q_{l}u\}$ we have
$$
\nabla \psi_{u}=u-\pi_{u}u\,.
$$
From $\{u_{k}\}$ we can extract a subsequence (we keep  calling it $\{u_{k}\}$) such that $\|u_{k}\|^{2} \rightarrow L$. Now we can compute
$$
\|\nabla \psi_{u}\|^{2}=\langle u,u \rangle -2\langle u, \pi_{u}u \rangle +  \langle \pi_{u}u, \pi_{u}u \rangle = \langle u,u \rangle -\langle \pi_{u}u, \pi_{u}u \rangle\,,
$$
where the second equality follows from $\langle u, \pi_{u}u \rangle=\langle \pi_{u}u, \pi_{u}u \rangle$ being $\pi_{u}$ an orthogonal projection. \\
Since $\langle u_{k},u_{k} \rangle \rightarrow L\,,$ and $\|\nabla \psi_{u_{k}}\|^{2}= \langle u_{k},u_{k} \rangle -\langle \pi_{u_{k}}u_{k}, \pi_{u_{k}}u_{k} \rangle \rightarrow 0\,, $ it follows that $$\langle \pi_{u_{k}}u_{k}, \pi_{u_{k}}u_{k} \rangle \rightarrow L\,.$$ 
Now, every $\pi_{u_{k}}u_{k}$ is a linear combination of the vectors $Q_{i}u_{k}$, namely
$$
\pi_{u_{k}}u_{k}=\sum_{i=1}^{l}\eta_{k}^{i}Q_{i}u_{k}\,,
$$
and by computing its norm we get
\begin{equation} \label{eq:converge}
\sum_{i,j=1}^{l}\eta_{k}^{i}\eta_{k}^{j}\langle Q_{i}u_{k},Q_{j}u_{k} \rangle \rightarrow L\,.
\end{equation}
Since the operators $Q_{i}$ are compact, they map the bounded sequence $\{u_{k}\}$ to a sequence $\{Q_{i}u_{k}\}$ with limit points, so we can iteratively extract converging subsequences (again we keep calling them $\{u_{k}\}$) and we have $
Q_{i}u_{k} \rightarrow v_{i}.$
In this way the equation (\ref{eq:converge}) becomes
$$
\sum_{i,j=1}^{l}\eta_{k}^{i}\eta_{k}^{j} \langle v_{i},v_{j} \rangle \rightarrow L\,,
$$
where the coefficients $\langle v_{i},v_{j} \rangle$ give the scalar product of the whole Hilbert space $H$ restricted to the finite dimensional subspace $V=\textrm{span}\{v_{i},\ldots,v_{l}\}$. Now we have a bounded sequence of vectors $\eta_{n}=\sum_{i=1}^{l}\eta_{n}^{i}v_{i}$ in $\mathbb{R}^{l}$ from which we can extract a converging sequence with limit $\eta$.

So the sequence $\pi_{u_{k}}u_{k}=\sum_{i=1}^{l}\eta_{k}^{i}Q_{i}u_{k}$ tends to $v=\sum_{i=1}^{l}\eta^{i}v_{i}$, and since
$$
0=\lim_{k \rightarrow \infty}\|\nabla \psi_{u_{k}}\|^{2}=\lim_{k \rightarrow \infty}\|u_{k}-\pi_{u_{k}}u_{k}\|^{2}=\lim_{k \rightarrow \infty}\|u_{k}-v\|^{2}\,,
$$
also the sequence $\{u_{k}\}$ tends to $v$. Moreover, since $v_{i}=\lim_{k \rightarrow \infty}Q_{i}u_{k}=Q_{i}v$ we have that $v=\eta Qv$ which is the condition for $v$ to be a critical point of the Energy.
\end{proof}

\section{Admissible-path space and its topology}\label{sec:paths}
\subsection{Paths with bounded energy}

Despite Theorem \ref{regularvalue} shows that (a priori) only some of the $\Omega_p$ are Hilbert manifolds, they are in fact all homotopy equivalent to each other; the argument is a simple modification of the standard one for loop spaces and appeared first in \cite{Ge}; we recall it here for convenience of the reader.

\begin{teo}\label{homeqad}For every $p_1, p_2\in G$ the spaces $\Omega_{p_1}$ and $\Omega_{p_2}$ are homotopy equivalent.

\end{teo}
\begin{proof}
It is sufficient to prove that for every $p\in G$ the space $\Omega_p$ is homotopy equivalent to $\Omega_e.$
To this end let $\gamma_0\in \Omega_p$ be a fixed \emph{admissible} path and define the map:
$$A:\Omega_e\to \Omega_p$$
by concatenation of loops in $\Omega_e$ with $\gamma$: $A(\gamma)= \gamma_0\gamma$ (velocities have to be rescaled). Let also $\hat\gamma_0$ be $\gamma_0$ with backward time (it connects $p$ to $e$); then define $$B:\Omega_p\to \Omega_e$$ by conatenation with $\hat\gamma_0:$ $B(\gamma)=\hat{\gamma}_0\gamma.$ Let now $\gamma_\epsilon, \,\epsilon \in [0,1]$ be the paths:
$$\gamma_\epsilon(t)=\gamma_0(\epsilon(1-t))$$ and $L_\epsilon: \gamma \to \gamma_\epsilon \hat{\gamma}_\epsilon \gamma$. the maps $L_\epsilon$ give a homotopy between is the identity $L_0=\textrm{id}:\Omega_e\to \Omega_e$ and $L_1=AB$. In a similar way $BA$ is homotopic to the identity on $\Omega_p$ and the two spaces are homotopy equivalent.
\end{proof}

As a corollary we see that $\Omega_p$ is contractible (in particular all its nonzero Betti numbers vanish).
\begin{coro}For every $p\in G$ the topological space $\Omega_p$ is contractible.
\end{coro}
\begin{proof}
By the above Theorem it is enough to show that $\Omega_e$ is contractible, and this is obvious since it is given by \emph{homogeneous} equations.
\end{proof}
We start now the study of Morse-Bott theory of the function $J|_{\Omega_p}=f$ on $\Omega_p$. In the sequel we will use the notation:
$$\Omega_p^s\doteq \Omega_p\cap \{J\leq s\}.$$
Proposition \ref{propo:bound} tells us that in the case $p$ is not a vertical point the number of geodesics joining $e$ to $p$ is finite; in particular if $s>0$ is large enough $f$ does not have critical points on $\{f\geq s\}$ and the topology of $\Omega_p^s$ stabilizes. 
\begin{propo}
If $p$ is not a vertical point, then for every $s$ large enough and $t>0$ the inclusion:
$$\Omega_p^s\hookrightarrow \Omega_p^{s+t}\quad \textrm{is a homotopy equivalence}.$$
\end{propo}

Thus we focus on the case $p\in \dd$. As we already mentioned, we \emph{will not} use Morse-Bott theory to give a lower bound on $b(\xps)$ by counting critical manifolds; we will instead use it to reduce the problem to the study of intersection of real quadrics. In fact the following proposition shows that $\xps$ is homotopy equivalent to its boundary $\partial \xps$; since the map $q$ is quadratic, the space $\partial \xps$ is an intersection of infinite-dimensional quadrics (it is given by the quadratic equations $\|u\|^2=2s$ and $q(u)=p$).

\begin{propo}\label{boundary}
For a generic choice of the Carnot group structure $W\subset \sod$ and a generic point $p\in \dd$, for almost every $s$ the following isomorphism holds:
$$H_{*}(\xps)\simeq H_{*}(\partial \xps).$$
\end{propo}
\begin{proof}
We first notice that the generic $s$ is not a critical value for the energy. Let us consider now the Morse-Bott function $g=-J$ and let us denote by $X^a$ the set $\{g\leq a\}$. The critical manifolds of $g$ are the same as for $J$, except that the \emph{index} of each one of them for $g$ is infinite (since these manifolds have \emph{finite} index for $J$, then they must have \emph{infinite} index for $g$).

After passing a critical value $c$ with corresponding critical manifold $C$, the relative homology of the Lebesgue set is given by (i.e. ``the homology changes by''):
$$
H_*(X^{c+\delta}, X^{c-\delta})\simeq H_{*}(D_C^{-},\partial D_C^{-}).
$$
We recall that $D_C^{-}$ is the unit disk bundle in the fiber bundle over $C$ on which the Hessian of the Morse-Bott function is negative definite; see Theorem \ref{bott} and the subsequent discussion from Appendix \ref{app:morsebott}. Notice that here the choice of the coefficients field $\mathbb{Z}_2$ prevents us from the problem of orientability of this bundle.

This relative homology is zero: since the index of $C$ is infinite, then both $D_c^{-}$ and $\partial D_c^{-}$ retract on $C$ (this follows from the fact that the infinite dimensional sphere is contractible).

We can conclude our proof by observing that even though we pass critical values for $-J$, the homology remains the same of $\partial \xps$ until we get the whole $\xps$.
\end{proof}
Thus we see that, being $\Omega_p$ contractible, each of the Betti numbers $b_i(\Omega_p^s)$  ($i>0$) eventually vanishes as $s\to \infty$. Despite this their sum can still grow: the smallest $i>0$ for which $b_i(\Omega_p^s)\neq 0$ will get bigger and bigger and the amount of topology can increase as well: we are interested in understanding quantitatively this phenomen. \subsection{Asymptotic Morse-Bott inequalities}
Before giving an explicit bound to $b(\xps)$, we will see what \emph{would} this bound \emph{be} if we were to use Morse-Bott inequalities only. The Morse-Bott inequalities bound will follow from the count of the muber of critical manifolds with energy less than $s.$ It turns out that this bound is much worse than the actual one: in fact one has:
$$\textrm{Card}\{\textrm{critical manifolds with energy less than $s$}\}\leq O(s)^l$$
against the actual bound $b(\xps)\leq O\left(s\right)^{l-1}$ (this will be proved in the next section).

The following proposition will be fundamental for the sequel: essentially it allows to turn the direct limits arguments into a quantitative form. Roughly it says that the wave numbers of the controls associated to the geodesics grow at most with the order of their energy.

To deal with this ideas, we introduce the following useful notation: for every $L\in \mathbb{N}$ we define:
$$T^L\doteq \bigoplus_{k\leq L}T_k.$$

\begin{propo} \label{starcity}
For the generic choice of the Carnot group structure $W\subset \sod$ and the generic point $p \in \dd$ there exists a constant $c_{p} > 0 $ such that for every geodesic $\gamma \in \Omega_{p} \cap \{ J \leq s\}$, its associated control belongs to $T^{\left\lfloor s c_{p}\right\rfloor}.$
\end{propo}
In order to prove the previous proposition we first need the following lemma:
\begin{lemma} \label{durazzo}
For the generic choice of the Carnot group structure $W\subset \sod$ and the generic point $p \in \dd$ there exists a constant $c_p > 0$ such that for every Lagrange multiplier $\omega$ associated to $p$, the following inequality holds:
$$
\frac{\langle \omega , p \rangle}{\| \omega \|} \geq \frac{1}{c_p}.
$$
\end{lemma}

\begin{remark}Being the quantity $\langle \omega , p \rangle/\|\omega \| $ the cosinus between $\omega$ and $p$ times the norm of $p$, the lemma says that the Lagrange multipliers for $p$ are contained in a convex acute cone in $W$.

The norm on the space of covectors $\dds \cong W$ is the one induced by the inclusion $W \hookrightarrow \sod$ where $\langle X,Y\rangle = \emph{Trace}(X^{T}Y)$.
\end{remark}

Before giving the proof of the lemma we show how it implies Proposition \ref{starcity}.
\begin{proof}[Proof of proposition \ref{starcity}]
Lemma \ref{durazzo} is equivalent to
$$
\| \omega \| \leq c_p \langle \omega , p \rangle.
$$
The norm of $\omega A \in W$ can be written in terms of its eigenvalues $\alpha_{1},\ldots,\alpha_{s}$: $\|\omega\|=\|\omega A\|=\sqrt{2\alpha_{1}^{2}+\ldots+2\alpha_{s}^{2}}$; it follows that every eigenvalue of $\omega A$ is smaller than the norm of $\omega A$. Since $\omega$ is a Lagrange multiplier, if $u=u_{k_1}+\ldots+u_{k_l}$ then the $k_j$ are integer eigenvalues of $i\omega A.$ Since the energy of a geodesic $u$ associated to $\omega$ is $J(u)=\langle \omega , p \rangle$, we have
\begin{equation} \label{eq:starcity}
k_j \leq \| \omega \| \leq c_p \langle \omega , p \rangle \leq c_p s.
\end{equation}
\end{proof}

Now we go back to the proof of Lemma \ref{durazzo}.
\begin{proof}[Proof of Lemma \ref{durazzo}]
Suppose on the contrary that the constant bounding $\langle \omega , p \rangle/\|\omega\|$ from below doesn't exist, so that we can find a sequence of Lagrange multipliers $\omega_n$ such that, setting $\hat{\omega}_n \doteq \omega_{n}/\|\omega_{n}\|$, the sequence $\hat{\omega}_n (p) \rightarrow 0$. Since the sequence $\hat{\omega}_{n}$ is contained in $S^{l-1}$ which is compact, we can assume (up to subsequences) that it converges, with limit $\lambda$ such that $\langle \lambda,p\rangle=0$ by hypotesis. Up to subsequences we can also assume that every Lagrange multiplier $\omega_{n}$ has the same number of integer eigenvalues (all distinct by Lemma \ref{nomultiple}), say $\nu$. For every Lagrange multiplier $\omega_{n}$ we have the cone of the endpoints of the geodesics  associated to $\omega_{n}$; the Lagrange multiplier $\omega_{n}$ is contained in the intersection of $\nu$ hypersurfaces of matrices with constant eigenvalue equal the imaginary integers $i k_{1}(n),\ldots,i k_{r}(n)$. The direct sum $E(\omega)$ of the associated eigenvalues contains all the geodesics with Lagrange multiplier $\omega$ and $q(E(\omega))$ is the cone spanned by the normal vectors to each of these $\nu$ surfaces (Proposition \ref{distinct}), where the normal vector has to be chosen with positive scalar product with $\omega_{n}$ (since the Energy is positive).

Let us call $\lpj$ the normal vectors to the surfaces with eigenvalue equal to $k_j(n)$ respectively; up to subsequences again we can assume that every direction $\lpj(n)$ converges to some $\lpj$. The point $p$ is contained in the interior of every cone (Theorem \ref{structure}) and it can be written as $p=\sum_{j=1}^{\nu}c_{j}(n)\lpj(n)$ with $c_{j}(n) > 0$ for every $j$ and every $n$. Since $\langle \lambda,p\rangle=0$, we have
$$
0=\langle \lambda,p\rangle=\sum_{j=1}^{\nu}\langle \lambda, \lpj\rangle=\lim_{n \rightarrow \infty}\sum_{j=1}^{\nu}c_{j}(n)\langle\hat{\omega}_{n},\lpj(n)\rangle;
$$
the terms of the sum above must converge to $0$ one by one because they are non-negative.
Not every $c_{j}(n)$ can converge to $0$, otherwise $p$ would be $0$ as well. Therefore at least one of the terms $\langle\hat{\omega}_{n},\lpj(n)\rangle$ converges to $0$; in the limit $p$ is a linear combination of directions $\lpj$ with $j \in I$ for a set of indexes $I$ such that $\langle \lambda, \lpj \rangle=0$ for every $j \in I$.

Now we are going to see what happens to Lagrange multipliers associated to directions $\lpj$ such that $\lim_{n \rightarrow \infty} \langle\hat{\omega}_{n},\lpj(n)\rangle =0$. Let us take the smooth hypersurface $S_{\alpha} \subset W$ of the matrices with an eigenvalue equal to a given $i\alpha \in i\R$ and all the other eigenvalues different from $i\alpha$. Then we take a point $\omega A \in S_\alpha$ and we compute the angle between $\omega A$ and the normal vector to the surface $S_\alpha$ at $\omega A$. The surface $S_\alpha$ can be given as a zero locus of the real valued function $s_\alpha(\eta)\doteq\det (\eta A-i \alpha\Id)$ in the even-dimensional case and $i \det (\eta A-i \alpha\Id)$ in the odd-dimensional case (here for simplicity we discuss the case $d$ is even, but the proof for the odd-dimensional case is analogous).

We may assume that the matrices $A^{1},\ldots,A^{l}$ form an \emph{orthonormal} basis for $W$; moreover we can choose an orthonormal basis for $\Delta$ such that the matrix $\omega A$ is written in canonical form, i.e.~$\omega A=\diag(\alpha J_{2},\alpha_2 J_2,\ldots,\alpha_{d/2}J_2)$.
Now we compute the differential of $s_\alpha$ at $\omega A$:
$$
(ds_\alpha)_{\omega A}=\sum_{i=1,j=1}^{d}\sum_{k=1}^{l}\frac{\partial s_\alpha}{\partial m_{ij}}\frac{\partial m_{ij}}{\partial \eta_k}d\eta_k=\sum_{i=1,j=1}^{d}\sum_{k=1}^{l}\textrm{adj}(\omega A - i \alpha \Id)^{ij}a_{ij}^{k}\,d\eta_k
$$
where $m_{ij}$ are the variables for the entries of the matrices, $\eta_k$ are the coordinates on $W$ given by the components of the covectors in $\dds$, $\textrm{adj}(\omega A - i \alpha \Id)^{ij}$ is the $ij$ entry of the adjugate matrix of $\omega A - i \alpha \Id$ and $a_{ij}^{k}$ are the entries of the matrix $A^{k}$. The matrix $\omega A - i \alpha \Id$ takes the form
$$
\omega A - i \alpha \Id = \diag\left(
\left( \begin{array}{cc}
-i \alpha & \alpha \\
-\alpha & -i \alpha \\
\end{array}\right)
,
\left( \begin{array}{cc}
-i \alpha & \alpha_2 \\
-\alpha_2 & -i \alpha \\
\end{array}\right), \ldots,
\left( \begin{array}{cc}
-i \alpha & \alpha_{d/2} \\
-\alpha_{d/2} & -i \alpha \\
\end{array}\right)
\right);
$$
setting $\beta \doteq \prod_{i=2}^{d/2}(a_{i}^{2}-a^{2})$, the adjugate matrix is
$$
\textrm{adj}(\omega A - i \alpha \Id) = \diag\left(
\left( \begin{array}{cc}
-i \alpha \beta & \alpha \beta \\
- \alpha \beta & -i \alpha \beta \\
\end{array}\right),
\left(\begin{array}{cc}
0 & 0 \\
0 & 0 \\
\end{array}\right),\ldots,
\left(\begin{array}{cc}
0 & 0 \\
0 & 0 \\
\end{array}\right)
\right),
$$
and so we get
$$
(ds_\alpha)_{\omega A} = 2\alpha \beta a_{12}^{k}\, d\eta_k.
$$
Therefore we have $$\langle (ds_\alpha)_{\omega A}, \omega A \rangle=2\alpha\beta \omega_k a_{12}^{k}=2\alpha^2\beta.$$ 
Since the basis $(A^{1},\ldots,A^{l})$ is orthonormal for $W$ and so is $(d\lambda_1,\ldots,d\lambda_l)$ for $W^*$, the norm $\|(ds_\alpha)_{\omega A}\|$ is easily computed: $$\|(ds_\alpha)_{\omega A}\| = 2|\alpha \beta|\sqrt{(a_{12}^{1})^2+\ldots+a_{12}^{l})^2} = 2|\alpha \beta| \|a_{12}\|$$ where $a_{12}\doteq (a_{12}^{1},\ldots,a_{12}^{l}).$
Now we can compute the cosinus of the angle $\theta$ between $\omega$ and the normal to $S_\alpha$ at $\omega A$:
$$
\cos\theta = \frac{\langle (ds_\alpha)_{\omega A}, \omega A \rangle}{\|(ds_\alpha)_{\omega A}\|\,\|\omega A\|} = \frac{2\alpha^2\beta}{2|\alpha \beta| \, \|a_{12}\|\,\|\omega\|} = \pm \frac{\alpha}{\|a_{12}\|\,\|\omega\|}.
$$

Let us go back to the directions $\lpj$ such that $\langle \lambda , \lpj \rangle = \lim_{n \rightarrow \infty}\langle \hat{\omega}_n, \lpj(n) \rangle = 0$; recall that the imaginary integers $ik_j(n)$ are the eigenvalue corresponding to the eigenspaces mapped on $\R \lpj(n)$. Notice that the constant $\|a_{12}\|$ in the previous computation depends on the space $W$ and on the basis we choose in order to put $\omega A$ in canonical form, but now the basis is a priori different for every $\omega_n A$; indeed we have that $\|a_{12}\| \leq \|A\|$ where the norm $\|A\|$ is the one induced by the norm on $\Delta$ on its exterior algebra. It follows that
$$
|\langle \hat{\omega}_n, \lpj(n) \rangle| \geq \left| \frac{\alpha_i(n)}{\|A\| \,\|\omega_n\|} \right|,
$$
and since the first term converges to $0$, so does the second one. 

Notice that the vectors $\lpj(c)$, orthogonal to the surfaces with eigenvalue equal to $k_i(n)$, are orthogonal to the surfaces with eigenvalue equal to $\frac{k_i(n)}{\|\omega_n\|}$ as well, since the surfaces are the same up to homothety. The point $p$ is a linear combination of the directions $\lpj$ with $j \in I$ with the property that $\langle \lambda, \lpj \rangle = 0$: these directions are the images of the limit eigenspaces $E_{j}(\lambda)$ of the sequences $E_{j}(\hat{\omega_n})$. Since their eigenvalues are $\frac{k_j(n)}{\|\omega_n\|}$ which converge to $0$, the eigenvalues of the eigenspaces $E_j(\lambda)$ for $j \in I$ are 0. It follows that $p$ belongs to the image by $q$ of the kernel of $\omega Q$, which means that $p$ is a critical value for $q$, and this is absurd by hypotesis. 
\end{proof}

%As a corollary of Lemma \ref{durazzo} we give another proof of Proposition \ref{boundary}.

We prove now the theorem that gives an upper bound for the number of critical manifolds with energy bounded by $s$.

\begin{teo}\label{criticalcount}
For the generic choice of the Carnot group structure $W\subset \sod$ and of $p\in \dd$: 
$$\emph{\textrm{Card}}\{\textrm{critical manifolds with energy less than $s$} \}\leq O(s)^l. $$
\end{teo}
\begin{proof}
Let us start by considering the vector space $P_{l,d}$ of real polynomials of degree $d$ in $l$ variables. For every $\nu=1, \ldots, l$ we set $f=(f_1, \ldots, f_\nu)\in P_{l,d}^\nu$ and :
$$Y_\nu=\{(f, V, \omega)\in P_{l,d}^\nu\times G(l-\nu, l) \times \mathbb{R}^l\,|\,\omega \in Z(f)\backslash \textrm{Sing}(Z(f)),\, V=(T_\omega Z(f))^\perp,\, p\in V\}.$$
Since $Y_\nu$ is semialgebraic, we can consider the semialgebraic projection to the first factor $\pi:Y_\nu\to P_{l,d}^\nu$ and stratify $P_{l,d}^\nu=\coprod_{j=1}^s P_j$ such that $\pi$ is semialgebraically trivial over each stratum (see \cite{BCR}). In particular since there are finitely many strata, there exists a number $\beta_\nu$ such that for every $f\in P_{l,d}^\nu$:
$$b_0(\pi^{-1}(f))\leq \beta_\nu$$
(only a finite number of fibers, up to semialgebraic homeomorphism, appear).
In particular, in the case when there are finitely many $\omega$ with normals $V= (T_\omega Z(f))^\perp$ containing $p$, this construction implies their number is bounded by $\beta_\nu$.

Consider now for every $n\in \mathbb{N}$ the polynomial $f_n\in P_{l,d}$ defined by:
$$f_n(\omega)=i^{d}\det(\omega A-in\mathbbm{1})$$
(the $i^d$ factor has the only scope of turning $f_n$ into a \emph{real} polynomial in the case $d$ is odd).
We know from Theorem \ref{nomultiple} that for the generic choice of $W\subset \sod$ and $p\in \dd$ each Lagrange multiplier belongs to some $Z(f_{k_1}, \ldots, f_{k_\nu})$, where $\nu$ is the number of positive integer eigenvalues of the matrix $i\omega A$. Now if $u=u_1+\cdots+u_{\nu}$ is the control associated to a geodesic with final point $p$ and energy less than $s$, Proposition \ref{starcity} implies:
$$k_j\leq s c_p\quad \textrm{for } j=1, \ldots, \nu$$
for a constant $c_p$ depending only on $W$ and $p$.

Thus the way to get all \emph{possible} Lagrange multipliers associated to geodesics ending at $p$ with energy less than $s$  is by intersecting $\nu$ of the hypersurfaces $Z(f_k)$ for $k=1,\ldots, \lfloor sc_p\rfloor$ and $\nu=1, \ldots, l$ and considering the points in this intersection where the normal space contains $p$. There are 
$$\binom{\lfloor s c_p\rfloor}{\nu}=O(s)^\nu$$
possible ways of choosing the hypersurfaces and the above argument implies each such choice can contribute by at most $\beta_\nu$ Lagrange multipliers.

In particular the set of Lagrange multipliers for $p$ with energy less than $s$ is bounded by:
$$\textrm{Card}\{\textrm{Lagrange multipliers $\omega$ such that $\omega(p)\leq s$}\}\leq \sum_{\nu=1}^l O(s)^\nu=O(s)^l.$$
To each critical manifold with energy less than $s$ there corresponds one and only one Lagrange multiplier whose scalar product with $p$ is less than $s$, hence the conlcusion follows.  
\end{proof}

As a corollary we derive now Morse-Bott inequalities; as already stated the following bound will be improved in the next section.

\begin{coro}[Morse-Bott inequalities]\label{MorseBottinequalities} For the generic choice of $W\subset \sod$ and of the point $p\in \dd$ we have:
$$b(\xps)\leq \sum_{J(C_\omega)\leq s}b(C_\omega)\leq O(s)^l.$$
\end{coro}

\begin{proof} By Lemma \ref{distinct} and Theorem \ref{structure} we know that each critical manifold is an intersection of $l$ quadrics in $\mathbb{R}^{2d}$ (it is the preimage of $p$ under $q|_{E(\omega)}$): in particular the possible homeomorphism types of such manifold are finite and there is a constant $\beta$ such that $b(C_\omega)\leq \beta$ for every critical manifold $C_\omega.$

Since the sum $\sum_{J(C)\leq s}b(C)$ contains at most $O(s)^l$ terms (by Theorem \ref{criticalcount}), the conlcusion follows from Morse-Bott inequalities (see Appendix \ref{app:morsebott}).
\end{proof}
%------------------------------------------------------------------------------------------------

\subsection{Asymptotic total Betti number}
The aim of this section is to refine the bound for $b(\xps)$ given in Corollary \ref{MorseBottinequalities}.

We start by proving some technical results.
\begin{propo} \label{boundedindex}
For the generic choice of $W \in \sod$ and the generic point $p \in \dd$ we have: (a) the critical manifolds of $J$ on $\Omega_p$ with Energy less than $s$ coincide with the critical manifolds of $J$ restricted to $\Omega_{p} \cap T^{\lfloor s c_p\rfloor}$; (b) their index is the same either if they are considered critical manifolds for $J$ or if they are considered critical manifolds for $J$ restricted to $\Omega_{p} \cap T^{\lfloor s c_p \rfloor}$.
\end{propo}

\begin{proof} We have already proved that there exists a constant $c_p$ such that all the critical points with energy less than $s$ are contained in $\Omega_{p} \cap T^{\lfloor s c_p \rfloor}$ (Proposition \ref{starcity}); since the spaces $T_{k}$ are orthogonal with respect to both the quadratic maps $J$ and $q$, the critical points of $J|_{\Omega_p \cap T^{\lfloor s c_p \rfloor}}$ are given by the same equations as for critical points of $J|_{\Omega_p}$ using the Lagrange multipliers rule.

Let $\omega$ be a Lagrange multiplier for $p$ with energy less than $s$. The Hessian of the energy $J$ is $\langle \textrm{Id}-\omega Q \cdot,\cdot \rangle$, so its eigenvalues are $1-\frac{\alpha_{i}(\omega)}{k}$ with $k \in \mathbb{N}_0$ and $\alpha_{i}(\omega)$ eigenvalues of $\omega A$ as usual. Therefore we have negative eigenvalues for every integer $k$ such that $k < \alpha_{i}(\omega)$ for at least one $i$; since $\alpha_i(\omega) \leq \|\omega\|$, it follows from Lemma \ref{durazzo} (as in the proof of Proposition \ref{starcity}) that
$$
k \leq \| \omega \| \leq c_p \langle \omega , p \rangle \leq  s c_p
$$
with the same constant $c_p$ as in Proposition \ref{starcity}.
\end{proof}

The next proposition tells that if we want to compute $b(\xps)$ we can restrict ourselves to the intersection with a finite dimensional subspace of the form $T^L$ (and indeed we have a quantitative control on the dimension).

\begin{propo}\label{vector}
For a generic choice of $W\subset \sod$ and of $p\in \dd$ there exists a constant $r_p>0$ such that for every $m\in \mathbb{N}$:
$$\xps \quad \textrm{deformation retracts to} \quad \xps \cap T^{\lfloor s r_p\rfloor+m}.$$
In particular: $H_*(\xps)\simeq H_*(\xps \cap T^{\lfloor s r_p\rfloor+m}).$
\end{propo}

\begin{proof}
Given $L \in \mathbb{N}$ we can define the function $f_L$, ``distance from $T^L$'' in the following way: every $u \in \xps$ can be uniquely written as $u = \bar{u} + v$ where $\bar{u} \in T^L$ and $v \in \left(T^L \right)^{\bot}$; then we define $f_L(u) \doteq \|v\|^2$. Assume that for a suitable $L$ there are no critical points for $f_L$ outside $\xps \cap T^L$ and the function $f_L$ satisfies the Palais-Smale condition: then we can retract the manifold $f_{L}^{-1}(\left[0,s\right])=\xps$ on the sublevel set $f_{L}^{-1}(0) = \xps \cap T^L$ by Theorem \ref{morsebottdeformation} (the function $f_L$ is bounded on $\omegat$ since $f_L \leq J \leq s$). The manifolds we are considering have boundary, but we can apply the argument above to $\Omega_p \cap \{J < s+ \delta \}$: it deformation retracts to $\xps$ again by Theorem \ref{morsebottdeformation} if we choose $\delta$ small enough not to have new critical values for the Energy; we can indeed choose such $\delta$ small enough not to have new critical values for $f_L$ as well.

We are now going to prove that $f_L$ satisfies the Palais-Smale condition and to find for which $L$ we are sure not to have critical points for $f_L$ in $\xps\setminus( \xps \cap T^L)$.

The gradient of the function $f_L$ at $u=\bar{u}+v$ is $2v$ restricted to $T_u\Omega_p$: this means that there exists $\eta \in \dds$ such that $\nabla_u F_L=2v-2\eta Qu$. For every critical point $u$ for $f_L$ there exists $\eta$ such that $v-\eta Qu=v-\eta Qv - \eta Q\bar{u}=0$: since the space $T^L$ and its orthogonal are invariant with respect to $\eta Q$ we have the equivalent couple of conditions
$$
v=\eta Qv, \quad \eta Q\bar{u}=0;
$$
the first condition tells that $v$ is a geodesic going somewhere, the second one tells that $q(\bar{u})$ is a critical value for $q$.

We need $f_L$ to satisfy the Palais-Smale condition: having the explicit expression of its gradient $\nabla f_L$, we omit this verification whose proof is analogous to the one for $J$ in Theorem \ref{morsebottpalaissmale}.
%The first condition tells that the component $v$ is a geodesic going somewhere with Lagrange multiplier $\eta$.
%; since $v \in \left( T^L\right)^{\bot}$ it has wave numbers greater than $L$, meaning that the matrix $i\eta A$ has integer eigenvalues greater than $L$. Being the norm of $\eta A$ greater than its eigenvalues, we have that $\|\eta\| \geq L$.
%The second condition tells that either $\bar{u}=0$ and so $v$ is a geodesic going to $p$, or $\bar{u} \neq 0$ and $q(\bar{u})$ is a critical value for the map $q$.

We are going to prove that there exists a constant $r_p$ such that if we take $L = \left \lfloor s r_p \right \rfloor$ there are no critical points for $f_L$ in $\xps \setminus (\xps \cap T^L)$.

Assume that such a constant does not exist: then for every term $\rho_n$ of a diverging sequence of positive real numbers, we find $s_n > 0$ and a critical point $u(n)$ for the function $f_{\lfloor s_n \rho_n\rfloor}$ outside $T^{\lfloor\rho_ns_n\rfloor}$ with Energy less than $s_n$. 
%We split $u(n)$ in $T^{\lfloor\rho_n/\epsilon_n\rfloor}$ and its orthogonal as before and we get $u(n) = \bar{u}(n) + v(n)$. If we take the terms of the sequence $\rho_n$ greater than the constant $c_p$ of Lemma \ref{durazzo} (and it is not restrictive to make such a choice since $\rho_n \rightarrow \infty$), every geodesic going to $p$ with energy less than $1/\epsilon_n$ belongs to $T^{\lfloor c_p/\epsilon_n \rfloor} \subset T^{\lfloor\rho_n/\epsilon_n\rfloor}$: then every $\bar{u}(n) \neq 0$ otherwise $v(n)$ would be a geodesic going to $p$ with Energy less than $1/\epsilon_n$ outside $T^{\rho_n/\epsilon_n}$. 
By hypotesis $v(n) \in \left(T^{\lfloor \rho_nsn_n\rfloor}\right)^{\bot}$ so that its Fourier expansion is $v(n)=\sum_{k > \rho_ns_n} v_k(n)$. Recall that $P\doteq Q|_{T_{1}}$: then we have:
$$
\|q(v(n))\|= \left\|\sum_{k > \rho_ns_n} \frac{1}{k}\langle P v_k(n),v_k(n)\rangle \right\| \leq \sum_{k > \rho_ns_n} \frac{1}{k} \|P\| \|u_k(n)\|^2 \leq
$$
$$
\leq \frac{1}{\rho_ns_n}\|P\| \|v(n)\|^2 \leq \frac{1}{\rho_ns_n}\|P\| s_n = \frac{\|P\|}{\rho_n}.
$$
The last term of the chain of inequalities converges to zero: it follows that $q(v(n))$ goes to zero as well. Since $p=q(\bar{u}(n))+q(v(n))$, we get
$$
\lim_{n\rightarrow \infty}q(\bar{u}(n)) = p,
$$
and as we noticed before $\bar{u}(n)$ is a critical point for $q$. This means that we have a sequence of critical values for $q$ converging to $p$, which is impossible since we picked our $p$ into an open subset of regular values (i.e. $p\in \dd\backslash \Sigma_1$). We end the proof by noticing that the condition for having critical points for $f_{\lfloor r_ps_n\rfloor}$ has to be verified on every $T_k$ separately (they are orthogonal and invariant with respect to $Q$): since there are no critical points of $f_{\lfloor r_ps_n\rfloor}$ outside $T^{\lfloor r_ps_n\rfloor}$, there are no critical points of the same function restricted to $T^{\lfloor r_ps_n\rfloor + m}$ (similarly to what happens to the critical points of $J$ in Proposition \ref{boundedindex}) and it follows eventually that $\xps \cap T^{\lfloor r_ps_n\rfloor + m}$ deformation retracts onto $\xps \cap T^{\lfloor r_ps_n\rfloor}$ for every $m \in \mathbb{N}$.
\end{proof}

As a corollary we get the following interesting result, that controls the growth rate of the index of the highest nonzero Betti number of $\xps.$
\begin{coro}\label{maxbetti}For a generic $W\subset \sod$ and $p\in \dd$, there exists a constant $r_p>0$ such that:
$$\max_{i}\{i\,|\,b_i(\xps)\neq 0\}\leq 2d\left\lfloor r_p s\right\rfloor.$$
\end{coro}
\begin{proof}
By Proposition \ref{vector} there exists $c_p>0$ such that $H_*(\xps)\simeq H_*(\xps \cap T^{\lfloor c_ps\rfloor}).$ In particular $\xps$ has the homology of a semialgebraic subset of $\mathbb{R}^{2d\lfloor c_ps \rfloor}$ (namely $\xps\cap T^{\lfloor c_ps\rfloor}$) and its $j$-th Betti number must be zero for $j> 2d\lfloor c_ps\rfloor.$ 
\end{proof}

Everything is now ready for the proof of the main theorem of this section.

\begin{teo}[Strong Morse-Bott inequalities]\label{bettiorder}
For the generic choice of $W\subset \sod$ and of $p\in \dd$ we have:
$$b(\xps)\leq O(s)^{l-1}.$$
\end{teo}

\begin{proof}
First we know from Proposition \ref{vector} that there exists $r_p>0$ such that:
$$b(\xps)=b(\xps\cap T^{\lfloor r_ps\rfloor})\quad \textrm{for all $s>0$ and $m\in \mathbb{N}$}.$$
It means that $\Omega_p^s$ has the same Betti numbers as $$\{x\in  T^{\lfloor r_ps\rfloor}\,| \, q_1(x)=p_1, \ldots, q_l(x)=p_l\}\cap \{\|x\|^2\leq 2s\}.$$ Notice that the dimension of $ T^{\lfloor r_ps\rfloor}$ is a $O(s).$ 

Let us consider the semialgebraic set (a level set of a quadratic map with $l$ components):
$$X=\{x\in  T^{\lfloor r_ps\rfloor}\,|\, q_1(x)=p_1, \ldots, q_l(x)=p_l\}$$   and $\epsilon>0$ small enough such that $X\cap \{2s-\epsilon\leq\|x\|^2\leq2s+ \epsilon\}$ deformation retracts onto $X$ (the existence of such $\epsilon$ is guaranteed by semialgebraic triviality, see \cite{BCR}). Let also $A=X\cap \{\|x\|^2\leq 2s+\epsilon$ (which deformation retracts onto $X\cap \{\|x\|^2\leq 2s\}$ and $B=X\cap \{2s-\epsilon\leq \|x\|^2\}.$ The Mayer-Vietoris exact sequence of the pair $(A, B)$ gives $b(A)+b(B)\leq b(A\cap B)+b(A\cup B)$, which implies:
$$b(\Omega_p^s)=b(A)\leq b(A\cap B)+b(A\cup B).$$
Since $A\cap B$ deformation retracts onto $X\cap \{\|x\|^2=2s\},$ then it is defined by $l$ quadratic equations on a sphere of dimension $\dim(T^{\lfloor r_ps\rfloor})-1=O(s)$; on the other hand $X=A\cup B$ is given by $l$ quadratic equations in a vector space of dimension $\dim(T^{\lfloor r_ps\rfloor})=O(s)$; hence by Proposition \ref{propo:quadrictopology} the total Betti numbers of these spaces are bounded by:
$$b(A\cap B)\leq O(s)^{l-1}\quad\textrm{and}\quad b(A\cup B)\leq O(s)^{l-1}$$
and the conclusion follows.

\end{proof}
%------------------------------------------------------------------------------------------------

\subsection{A topological coarea formula}
In this section we compute the first order asymptotic of $b(\xps)$ in $s$, for the case $l=2.$ It turns out that for a generic choice of the Carnot group structure $W\subset \sod$ and the point $p\in \Delta^2$, the leading term is a \emph{real} number and can be analytically computed using only the data $W=\textrm{span}\{A_1, A_2\}.$

Consider a unit circle $S^1 \subset W.$ For a generic $W$ the eigenvalues of $\omega A$ are distinct and differentiable almost everywhere (the set of matrices in $\sod$ with multiple eigenvalues is a cone with codimension 3,  and the eigenvalues are semialgebraic funtions of the parameter $\omega \in S^1$). Thus there exist semialgebraic functions $\alpha_j:S^1\to \mathbb{R}$ such that the $\alpha_j(\omega),$ for $j=1, \ldots, d,$ are the coefficients of the canonical form of $\omega A$. Given $p\in \Delta$ we consider the \emph{rational} functions $\lambda_j:S^1\to \mathbb{R}\cup \{\infty\}$ given by:
$$\lambda_j:\omega \mapsto \left|\frac{\alpha_j(\omega)}{\langle \omega , p \rangle}\right|\quad \textrm{for}\quad j=1, \ldots, d.$$
Notice that when $\omega$ approaches $p^{\perp}$ these functions might explode, that is why they are rational in $\omega$; on the other hand they are semialgebraic and differentiable almost everywhere and it makes sense to consider the integral:
\begin{equation} \label{eq:tcf}
\tau(p)\doteq\frac{1}{2}\int_{S^1}\sum_{j=1}^{d}\left|\dot{\lambda}_j(\omega)\right|-\left|\sum_{j=1}^d \dot{\lambda}_j(\omega)\right|d\omega.
\end{equation}
The convergence of the integral follows from the fact that where the derivatives of the $\lambda_j$s explode, they all have the same sign and the integrand vanishes.

The next theorem proves that for a fixed $p,$ as a function of $s$:
$$b(\xps)=\tau(p)s+o(s)\quad\textrm{as $s\to \infty$}. $$
\begin{teo}\label{coarea}
If the corank $l=2$, for a generic choice of $W\subset \sod$ and $p\in \dd$ we have:
$$\lim_{s\to \infty}\frac{ b(\xps)}{s}=\tau(p).$$
\end{teo}
\begin{proof}
In order to compute the asymptotic for $b(\xps)$ for $s\to \infty$ we use (\ref{indexformulainfinite}). In fact we have seen that $\xps$ is homotopy equivalent to $\{v\in H\,|\, q(v)= p/s,\, \|v\|^2=1\}$ and the latter can be rewritten as:
$$\left\{v\in H\,|\, q(v)- \frac{ \|v\|^2}{s}p=0\right\}\cap S$$
where $S$ is the infinite dimensional sphere $S=\{\|v\|^2=1\}.$
In particular we can present our set as the intersection of two quadrics in $H$ on the unit sphere $S$; thus we let $\ii^-_s$ for the index function of the quadratic map $q- p \|\cdot\|^2/s$ (we are using the notation of \ref{teo:quadrictopologyinfinite}). In this setting the set $P\subset S^1$ coincides with $\{\omega \in S^1\,|\, \langle \omega , p \rangle<0\}.$ In fact if we let $p=(p_1, p_2)$, here the two quadrics we are considering are $q_1-p_1\|\cdot\|^2/s$ and $q_2- p_2\|\cdot\|^2/s$ and for every $\omega$ the selfadjoint operator on $H$ corresponding to the quadratic form $\omega q$ is $\omega Q- \omega (p)\mathbbm{1}/s.$ In particular the spectrum of $\omega Q-\omega (p)\mathbbm{1}/s$ is obtained by translating the spectrum of $\omega Q$ by $\langle \omega , p \rangle/s$ and since $\omega Q$ is compact and its spectrum is symmetric with respect to the origin, we see that in order to have finitely many negative eigenvalues we need $\langle \omega , p \rangle<0.$

On the other hand the subspaces $T_k$ are invariant by both $\omega Q$ and $ \langle \omega , p \rangle\mathbbm{1}/s,$ thus the index function can be computed as: 
\begin{align*}\ii^-_{s}(\omega)&=\sum_{k\geq 1}\ii^-\left(\omega Q- \frac{\langle \omega , p \rangle}{s}\mathbbm{1}|_{T_k}\right)=\sum_{k\geq 1}\ii^-\bigg(\frac{\omega Q_0}{k}-\frac{ \langle \omega , p \rangle}{s}\mathbbm{1}\bigg)=\\
&=\sum_{j=1}^d \bigg\lfloor \frac{s\alpha_j(\omega)}{\langle \omega , p \rangle}\bigg\rfloor=\sum_{j=1}^d \lfloor s \lambda_j(\omega)\rfloor
\end{align*}
where in the second line we have used the fact that the spectrum of $\frac{\omega Q_0}{k}-\frac{\langle \omega , p \rangle}{s}\mathbbm{1}$ is of the form $\frac{\alpha_j(\omega)}{k}-\frac{\langle \omega , p \rangle}{s}.$ In the sequel we also identify $P\subset S^1$ with a subset of $[0, 2\pi]$ in the standard way.

Denoting now by $\mu(s)$ the number of local maxima of $\ii^-_{s}$ on $P$, we see that formula (\ref{indexformulainfinite}) implies:
\begin{equation}\label{eq:formulaindex}b(\xps)=2\mu(s)+1-b_0(P_0)=2\mu(s)+O(1)\end{equation}
In fact, using the long exact sequence of the pair $(P_{j+1}, P_j)$, we can rewrite $b_0(P_{j+1}, P_j)=b_0(P_{j+1})-b_0(P_j)+b_1(P_{j+1}, P_j);$ substituting these identities into $b(\hat{\Omega}_{\epsilon p})=1+\sum_{j\geq 1}b_0(P_{j+1}, P_j)+ b_1(P_{j+2}, P_{j+1})$ we get $b(\hat{\Omega}_{\epsilon p})=1-b_0(P_0)+2 \sum_{j\geq 1}b_1(P_{j+1}, P_j)$. Since each local maximum of $\ii^-_{\epsilon}$  contributes by $1$ to one of the $b_1(P_{j+1}, P_j)$ and $b(P_0)\leq 1$ (since $P_0$ is convex), then (\ref{eq:formulaindex}) follows.

In order to compute the asymptotic of the number of maxima of $\ii^-_{\epsilon}$ we introduce the following auxiliary data. First we let $\lambda=\sum_{j=1}^d\lambda_j(\omega)$ and notice that this is a semialgebraic function. In particular we can divide $P$ into a finite number of intervals (arcs):
$$P=(\omega_0,\omega_1]\cup[\omega_1,\omega_2]\cup\cdots\cup[\omega_m, \omega_{m+1}]\cup [\omega_{m+1}, \omega_{m+2})$$
such that for every $j,k$ the functions $\alpha_j$ as well as $\alpha$ are monotone on $(\omega_k, \omega_{k+1})$. Labeling $I_k=[\omega_k, \omega_{k+1}]$ we see that also each $\lfloor s\lambda_j \rfloor$ is monotone on $I_k$. On the other hand monotonicity of $\ii^-_{s}$ is granted only where the signs of the derivatives of the $\lambda_j$ all agree. Since for the generic choice of $p$ the functions $\alpha_j$ do not vanish on $\{\omega_0, \omega_{m+2}\}$ (the orthogonal complement of $p$ on $S^1$), $\lambda_j$ approaches infinity when approaching $\omega_0$ or $\omega_{m+2}$; in particular $\ii^-_{s}$ \emph{is} monotone on $I_0$ and $I_{m+1}$ and has no local maxima on them.

For every $j\in \{1, \ldots, m\}$ let us denote respectively by $\mu_j(s)$ and $\sigma_j(s)$ the number of local maxima of $\ii^-_s$ on $I_j$ and the number of subintervals of $I_j$ where $\ii^-_s$ is constant (thus $\sigma_j(s)$ equals the number of  ``jumps" of the integer valued function $\ii^-_{s}$ on $I_j$).

For every interval $I_j=[\omega_j, \omega_{j+1}]$ we see that:
$$|\ii^-_s(\omega_{j+1})-\ii^-_s(\omega_j)|=\sigma_j(s)-2\mu_{j}(s)$$
In particular summing all these equations and using the fact that $\ii^-_{s}$ is monotone on $I_0$ and $I_{m+1}$, combining with (\ref{eq:formulaindex}) we get:
\begin{align*}\frac{ b(\xps)}{s}&=2\frac{ \mu(\epsilon)}{s}+O(1/s)=2\sum_{j=1}^m\frac{\mu_j(s)}{s}+O(1/s)=\\
&=\sum_{j=1}^m\frac{ \sigma_j(s)}{s}-\sum_{j=1}^m \left|\frac{\ii^-_s(\omega_{j+1})- \ii^-_s(\omega_j)}{s}\right|+O(1/s).
\end{align*}
Now we notice that as $s\to \infty$, the function $ \ii^-_{s}/s$ converges uniformly to $\lambda$, thus:
\begin{equation}\label{eq:first}\lim_{s\to \infty}\sum_{j=1}^m \left|\frac{\ii^-_s(\omega_{j+1})- \ii^-_s(\omega_j)}{s}\right|+O(1/s)=\sum_{j=1}^m|\lambda(\omega_{j+1})-\lambda(\omega_{j})|=\int_{\omega_1}^{\omega_{m+1}}\left|\dot{\lambda}(\omega)\right|d\omega.
\end{equation}
It remains to evaluate $\lim_{s}\sum_{j}\frac{ \sigma_j(s)}{s}.$ To this end we let $\sigma_j^i(s)$ be the number of jumps of $\left\lfloor s\lambda_i/\right\rfloor$ on the interval $I_j$. We notice that $\sigma_j^i(s)=\sum_{i=1}^d\sigma_j^i(s)+O(1):$ in fact $\ii^-_{s}$ jumps exactly when one of the $\lfloor s \lambda_j\rfloor$ jumps and these function all jump at different points (except for the points where two eigenvalues are in resonance, but these are in finite number bounded independently of $s$); we also notice that each function $\left\lfloor s\lambda_i\right\rfloor/s$ converges uniformly to $\lambda_i.$ Thus we get:
\begin{align}\nonumber\label{eq:second}\lim_{s\to \infty}\sum_{j=1}^m\frac{\sigma_j(s)}{s}&=\lim_{s\to \infty}\sum_{j=1}^m\left(\sum_{i=1}^d\frac{\sigma_j^i(s)}{s}\right)=\lim_{s\to \infty}\sum_{j=1}^m\left(\sum_{i=1}^d    \left|\frac{\lfloor s \lambda_i(\omega_{j+1})\rfloor}{s} -\frac{\left\lfloor s\lambda_i(\omega_{j})\right\rfloor}{s}\right|\right)=\\ \nonumber
&=\sum_{j=1}^m\left(\sum_{i=1}^d\left|\lambda_i(\omega_{j+1})-\lambda_{i}(\omega_j)\right|\right)=\sum_{j=1}^m\int_{\omega_j}^{\omega_{j+1}}\left(\sum_{i=1}^d\left|\dot{\lambda}_i(\omega)\right|\right)d\omega=\\
&=\int_{\omega_1}^{\omega_{m+1}}\left(\sum_{i=1}^d\left|\dot{\lambda}_i(\omega)\right|\right)d\omega.
\end{align} 
Combininig (\ref{eq:first}) and (\ref{eq:second}) we finally get:
$$\lim_{s\to \infty}\frac{b(\xps)}{s}=\int_{\omega_1}^{\omega_{m+1}}\sum_{i=1}^d\left|\dot{\lambda}_i(\omega)\right|-\left|\sum_{i=1}^d\dot{\lambda}_i(\omega)\right|d\omega=\int_P\sum_{i=1}^d\left|\dot{\lambda}_i(\omega)\right|-\left|\sum_{i=1}^d\dot{\lambda}_i(\omega)\right|d\omega$$
where the last identity follows from the fact that on $I_0$ and $I_{m+1}$ the two functions $\sum_{i=1}^d\left|\dot{\lambda}_i(\omega)\right|$ and $\left|\sum_{i=1}^d\dot{\lambda}_i(\omega)\right|$ are equal. The limit of the statement simply follows by noticing that $\alpha_i(\omega)=\alpha_i(-\omega)$ (i.e. the positive eigenvalues of $i\omega A$ are $\pi$-periodic).
\end{proof}
As a corollary we get the following result: it says that the topology of the set of paths reaching the point $\epsilon p$ with energy $J\leq 1$ explodes; in particular the number of geodesics gets unbounded as well.

\begin{coro}For a generic choice of $W\subset \sod$ and the point $p\in \dd:$
$$\lim_{\epsilon\to 0}b(\Omega_{\epsilon p}\cap \{J\leq 1\})=\infty.$$
\end{coro}
\begin{proof}First notice that $\Omega_{\epsilon p}\cap \{J\leq 1\}$ is homeomorphic to $\Omega_p\cap \{J\leq 1/\epsilon\}$; thus we can apply the above theorem.

In order to prove the limit it is enough to show that for the generic choice of $W$ and $p$ the integral $\tau(p)$ is not zero. Since the integrand function is always nonnegative, it's enough to prove it doesn't vanish identically. Pick two distinct eigenvalues (functions), say $i\alpha_1$ and $i\alpha_2$, for the family $\{\omega A\}_{\omega\in S^1}$. Since these functions are continuous semialgebraic, $i\alpha_1$ has a maximum point $\overline{\omega}$ and we can assume this is not a critical point for $\alpha_2$ also (this is a generic condition). Then in a neighborhood of $\overline{\omega}$ the derivatives of the corresponding $\lambda_1$ and $\lambda_2$ have different signs and the integrand is nonzero. 
\end{proof}

We conclude the section with an example where the topological coarea formula can be computed directly.

\begin{example}[Commuting matrices, corank $l=2$]
Let us fix the corank $l=2$. If the matrices $A_1$ and $A_2$ commute, they can be written simultaneously in their canonical form
$$
A_i=\textrm{diag}\left(v_{1}^{i}J_2,\ldots,v_{k}^{i},0_h \right),
$$
where as usual $J_2$ is the $2\times 2$ sympletic matrix and $0_h$ is the $h \times h$ zero matrix (possibly with $h=0$. Setting $v_{j}\doteq (v_{j}^{1},v_{j}^{2})$ and given $\omega \in \R^{2*}$, the eigenvalues of the matrix $\omega A$ are $\pm \langle\omega,v_j\rangle$. Now we pick a generic $p \in \dd$: having parametrized by $t$ the unit circle in $R^{2*}$, the functions we need in order to compute $\tau(p)$ are $\lambda_j(t) \doteq \left|\frac{\langle \omega(t),v_j \rangle}{\langle \omega(t),p \rangle} \right|$, their derivatives being
$$
\dot{\lambda}_j(t) = \frac{\sgn\left(\langle \omega(t),v_j \rangle\right)}{\sgn(\langle \omega(t),p \rangle)}  \cdot \frac{\langle \dot{\omega}(t),v_j \rangle \langle \omega(t), p  \rangle - \langle \omega(t),v_j \rangle \langle \dot{\omega}(t), p \rangle}{\langle \omega(t),p\rangle^2}.
$$
Since the curve $\omega(t)$ is the arc-length parametrization of the unit circle, for every $t$ the covectors $\omega(t),\dot{\omega}(t)$ form an orthonormal basis for $\dds$; it follows that the term
$$
m_j \doteq \langle \dot{\omega}(t),v_j \rangle \langle \omega(t), p  \rangle - \langle \omega(t),v_j \rangle \langle \dot{\omega}(t), p \rangle
$$
is the determinant of the matrix $\left(\begin{array}{cc} p^1 & a^{1}_{j} \\ p^2 & a^{2}_{j} \end{array}\right)$ which does not depend on $t$.
Since the functions $\lambda_j$ are periodic by $\pi$, we can modify the formula \eqref{eq:tcf} by integrating on the semicircle $\{\langle\omega(t),p\rangle >0\}$ and eliminating the coefficient of $1/2$ before the integral sign. 

These two remarks allow us to simplify the expression of $\dot{\lambda}(t)$:
$$
\dot{\lambda}_j(t)= \sgn(\langle \omega(t),v_j \rangle) \cdot \frac{m_j}{\langle \omega(t),p\rangle^2}.
$$
Now, in order to have the term $\tau(p)>0$, the integrand of
\begin{align*}
\tau(p) & =  \int_{\langle\omega(t),p\rangle > 0}\sum_{j=1}^{d}\big|\dot{\lambda}_j(\omega)\big|-\bigg|\sum_{j=1}^d \dot{\lambda}_j(\omega)\bigg|d\omega = \\
& = \int_{\langle\omega(t),p\rangle > 0}\frac{1}{\langle \omega(t),p\rangle^2}\cdot\left(\sum_{i=1}^{d}|m_j | - \bigg|\sum_{i=1}^{d}\sgn(\langle \omega(t),v_j \rangle)\ m_j  \bigg| \right)d\omega
\end{align*}
must be strictly positive somewhere. This happens if and only if there exist $i$ and $j$ such that for some $t_0$ the terms $\sgn(\langle \omega(t_0),v_i\rangle)m_i$ and $\sgn(\langle \omega(t_0),v_j\rangle)m_j$ have opposite sign. If $m_i \cdot m_j > 0$ we must have $t_0$ such that $\sgn(\langle \omega(t),v_i\rangle)\cdot\sgn(\langle \omega(t),v_j\rangle) < 0$: if such a $t_0$ does not exist, we have $\sgn(\langle \omega(t),v_i\rangle)\cdot\sgn(\langle \omega(t),v_j\rangle) > 0$ for all $t$ such that $\langle\omega(t),p\rangle > 0$, meaning that $v_i$ is proportional to $v_j$; a similar argument holds if $m_i \cdot m_j < 0$.

In order to have $\tau(p)=0$ the only possibility is that every $v_i$ is proportional to every other $v_j$, which is equivalent to say that $A_1$ and $A_2$ are proportional: this implies that $\dim \dd - \dim \Delta =1$ which contradicts the hypothesis.
\end{example}

\appendix

%\section{Semialgebraic geometry}\label{appsemialg}\todo{Put the relevant results here}

\section{Stratifications of $\sod$} \label{app:sodstrat}
Here we construct a useful stratification of $\sod$, generalizing the results from the appendix of \cite{BoscainGauthier}; for general results on stratifications and semialgebraic sets the reader is referred to \cite{BCR}.

We are interested in studying the dimensions of the semi-algebraic sets with generalized eigenvalues of given multiplicities and with given dimension of the kernel.
Every skew-symmetric matrix $A$ can be written in its canonical form as a block matrix, with blocks on the diagonal of the form
$$
\alpha J_2=\begin{pmatrix}
0 & \alpha \\
-\alpha  & 0 
\end{pmatrix},
$$
($J_2$ being the canonical symplectic matrix in $\mathfrak{so}(2)$) and a 0-block of the dimension of the kernel. By \emph{generalized eigenvalues} we mean the entries like $\alpha$.

We introduce the set:
$$
\gammas\subset \sod
$$
defined to be the set of skew-symmetric matrices in $\sod$ with: (a) dimension of the kernel equal to $k$ and (b) multiplicities of the generalized eigenvalues $m_{1},\ldots,m_{r}$ (with $m_1\geq m_2\geq\cdots m_r$).

By acting with $SO(d)$ on a matrix $A \in \gammas$ with eigenvalues $\alpha_{1},\ldots,\alpha_{r}$ (corresponding to the ordered multiplicities), we can put it in the form:
$$
\diag \left(\,\alpha_{1}J_{2m_{1}}, \ldots, \alpha_{r}J_{2m_{r}}, 0_{k}\right)\,,
$$
Let us look at the stabilizer $SO(d)_{A}$ of $A$: first of all it has to fix every eigenspace of $A$, since eigenspaces with different eigenvalues are orthogonal. On the kernel $K$ the stabilizer is the restriction of $SO(d)$ on $K$, i.e.~a copy of $SO(k)$; on the eigenspace with eigenvalue $\alpha_i$ the restriction of $SO(d)$ is a copy of $SO(2m_{i})$, but the stabilizer has to fix the symplectic matrix; it follows that the stabilizer act as a copy of:
$$
SO(2m_{i}) \cap Sp(2m_{i})=U(m_{i}).
$$
where $J_{2n}$ is the symplectic matrix in $\R^{2n}$ and $0_{k}$ is the null matrix on $\R^{k}$.

Now it is possible to compute the codimension of the orbit $\textrm{Ad}(SO(d))A$ of $A$ by the adjoint action $\textrm{Ad}$ of $SO(d)$ on $\sod$ with known stabilizer:
\begin{align*}
&\textrm{codim}\,\textrm{Ad}(SO(d))A =\dim \sod - \dim \textrm{Ad}(SO(d))A = \\
&=\dim \sod - \dim SO(d) + \dim SO(d)_{A}=\dim SO(d)_{A}\,.
\end{align*}
We know that the stabilizer $SO(d)_{A}$ is:
$$
SO(d)_{A}=SO(k) \times U(m_{1}) \times \ldots \times U(m_{r})\,,
$$
and its dimension is:
$$
\dim SO(d)_{A}=\frac{k(k-1)}{2}+\sum_{i=1}^{r}m_{i}^{2}.
$$
Let us now consider the eigenvalues $\alpha_{i}$: as long as they are distinct (so they preserve their multiplicities) they are smooth functions of the matrices \cite{Kato}. On the set $\gammas$ this condition holds true (by definition), hence we have a smooth map:
$$\psi: \gammas \to \R^{r}$$
given by $A  \mapsto  (\alpha_{i}).$
This map is indeed a submersion on the open subset $\mathcal{O}$ of vectors in $\R^{r}$ with distinct entries. The fibers of the map $\psi$ are the orbits of the adjoint action and they are diffeomorphic to a fixed manifold $SO(d) / SO(d)_{A}$; in particular $\gammas$ is a fiber bundle over $\mathcal{O}$ with fibers diffeomorphic to $SO(d) / SO(d)_{A}$. Now we can compute the codimension of $\gammas$ in $\sod$:
\begin{align} \label{eq:gammas}
\textrm{codim}_{\sod} \gammas =\ & \textrm{codim}_{\sod} \textrm{Ad}(SO(d))A - r=\frac{k(k-1)}{2} -r+\sum_{i=1}^{r}m_{i}^{2} \\
& = \frac{k(k-1)}{2} + \sum_{i=1}^{r}\left(m_{i}^2-1 \right).
\end{align}

Since we are interested in the matrices with integer eigenvalues, we will need to stratify $\gammas$ in infinite semialgebraic sets with given integer eigenvalues. Take $\vec{n}=(n_1,\ldots,n_r) \in \mathbb{N}^r$ with non-negative entries such that all the non-zero entries are distinct. Then by $\gammasn$ we will mean the stratum in $\gammas$ with the eigenvalue of multiplicity $m_i$ equal to $\ii n_i$ if and only if $n_i > 0$. The eigenvalues corresponding to zero entries of $\vec{n}$ vary in $\R$. Since by fixing an eigenvalue we drop the dimension of the stratum by 1, we have the following:
\begin{propoa} \label{stratiappendix}
Given $\vec{n} \in \mathbb{N}^r$ with the properties described above and with $\nu$ non-zero entries, the submanifold $\gammasn$ has codimension $\nu$ in $\gammas$, thus its codimension in the set of all matrices $\sod$ is
\begin{equation} \label{eq:codimresonances}
\emph{codim}_{\sod} \gammasn  = \frac{k(k-1)}{2} + \sum_{i=1}^{r}\left(\mu_{i}^{2}-1 \right) + \nu.
\end{equation}
\end{propoa}

\section{Morse-Bott functions} \label{app:morsebott}
In this section we give a short review of Morse-Bott theory; the interested reader is referred to the original paper by Bott \cite{Bott1} and to the books \cite{Klingenberg} and \cite{chang} for more details (especially for the infinite dimensional case).\\
We recall that a function $f:X \to \mathbb{R}$ on the Hilbert manifold $X$ is called a \emph{Morse-Bott function} if: (a) the critical set is the disjoint union of compact smooth manifolds; (b) if $x$ is a critical point belonging to the critical manifold $C$ then $\ker \he_{x}f = T_{x}C;$ (c) for every sequence $\left\{x_{k}\right\}$ on $X$ such that $f(x_{k})$ is bounded and $\|\nabla f_{x_{k}}\|\to 0$, then the sequence $\left\{x_{k}\right\}$ has limit points and every limit point is critical for $f$.\\
Condition (c) is usually referred as \emph{Palais-Smale condition} and is automatically satisfied in the finite dimensional case.
The smooth manifolds of critical points are called \emph{nondegenerate critical manifolds}; notice that admitting only zero-dimensional critical manifolds we get classical Morse functions.\\
The second condition is equivalent to the non-degeneracy of the Hessian on the normal space $N_{x}C$ for $x \in C$. The \emph{index} of the critical manifold $C$ is defined as the maximum of the dimensions of subspaces $V \subset N_{x}C$ where the Hessian is negative-definite; since the Hessian is nondegenerate in the all normal bundle, this number does not depend on the point $x \in C$ and it is denoted by $\textrm{ind}(C)$. \newline
Under this assumptions it is still possible to describe what happens to the topology of the sublevels $X^c \doteq f^{-1}(-\infty,c)$ of the Morse-Bott function $f$ by increasing $c$.

Let us consider a Morse-Bott function $f:X \to \mathbb{R}$ together with a Riemannian metric $g$ on $X$ (the choice of $g$ it's only a technical convenience and in fact the following results do not depend on it). 
The first fundamental theorem of Morse theory describe how the sublevels (don't) change when $c$ increases without passing critical values:
\begin{teoa}
If $f:X\rightarrow \R$ is a Morse-Bott function on $X$ satisfying the Palais-Smale condition and the interval $\left[a,b\right] \subset \R$ does not contain critical values, $X^b$ is diffeomorphic to $X^a$.
\end{teoa}
Roughly speaking $X^b$ is deformed to $X^a$ along the integral curves of the gradient flow $\nabla f$. Moreover if we let $a$ be a critical value for the Morse-Bott function $f$ the sublevels aren't diffeomorphic one to another anymore, but still there exists a deformation.
\begin{teoa} \label{morsebottdeformation}
If $f:X \rightarrow \R$ is a $C^1$ function satisfying the Palais-Smale condition and $\left(a,b\right] \subset \R$ does not contain critical values, $X^a$ is a strong deformation retract of $X^b$.
\end{teoa}
For the proof see Lemma 3.2 in Chapter 1 of \cite{chang}.

It only remains to recall what happens when we pass a critical value for a Morse-Bott function.
Given a critical manifold $C$ we restrict the tangent bundle $TX$ to $C$ and  consider the sub-bundle:
$$
E_C^{-}=\left\{ \textrm{directions where the Hessian of the function $f$ is negative-definite} \right\}\, .
$$
together with the unit disk bundle $D_C^{-} \subset E_C^{-}$. With this notation the following theorem generalizes the classical one (the statement we present here is actually the one in \cite{Klingenberg}).
\begin{teoa}[Bott]\label{bott} 
Let $f:X \to\mathbb{R}$ be a Morse-Bott function, and let $c$ be a critical value. For $\delta>0$ sufficently small the sublevel $X^{c+\delta} = X \cap \left\{f \leq c+\epsilon \right\}$ is homotopic to the sublevel $X^{c-\epsilon}$ with the unit disk bundle $D_C^{-}$ glued along the boundary.
\end{teoa}
Since in our setting we consider homology with $\mathbb{Z}_{2}$ coefficients, it follows from the Thom isomorphism that:
$$
H_{*}(X^{c+\delta},X^{c-\delta})\simeq H_{*}(D_C^{-},\partial D_C^{-}),
$$
where in the last equation we allow the critical manifold $C$ to be nonconnected, in which case we actually have a disjoint union of different bundles (with possibly different rank, corresponding to the possibly different indexes of the components of $C$).\\
Moreover we can state \emph{Morse-Bott inequalities} in terms of the Poincar\'e polynomial of $X^s$ and Morse (Bott) polynomial of $f$, which is defined by
$$
M_{f}^s(t)=\sum_{\{C\,|\,f(C)\leq s\}}P_{C}(t)\, t^{\textrm{ind}(C)}\,,
$$
where the sum is taken amongst the critical manifolds $C$ contained in $\{f\leq s\}$; evaluations of this sum at $t=1$ gives the following.
\begin{propoa} [Morse-Bott inequalities]
$$b(X^s)\leq\sum_{\{C\,|\,f(C)\leq s\}}b(C).$$
\end{propoa}

\section{The cohomology of the intersection of real quadrics}\label{appendixquadrics}

In this section we present useful results from \cite{AgrachevLerarioSystems, LerarioConvex, LerarioComplexity} for the study of the Betti numbers of the intersection of real quadrics on the sphere $S^n$. The motivating example is the case of the zero locus $Y$ of one single nondegenerate quadratic form $q$ on the sphere $S^n$: if $\ii^-(q)$ denotes the negative inertia index of $q$, then:
$$Y\simeq S^{\ii^-(q)-1}\times S^{n-\ii^{-}(q)}.$$
In particular we see that the knowledge of the \emph{index} function on the whole line spanned by $q$ in the space of all quadratic forms determines the topology (since by nondegeneracy $n-\ii^{-}(q)=\ii^{-}(-q)-1$).\\
More generally if we have $l$ quadratic forms $q_1, \ldots, q_l$ in $n+1$ variables, then we consider the function $\eta\mapsto \ii^{-}(\eta q);$ in the generic case this function is the restriction of the negative inertia index function to the span of $q_1, \ldots, q_l$ in the space of all quadratic forms. Although the general theory is more detailed, for our purposes we need explicit computations only in the case $l=2$ and in the general case it will suffice to have quantitative bounds on the topology of: $$Y=\{x\in S^n\,|\, q_1(x)=\cdots=q_l(x)=0\}.$$
In the case $l=2$ we consider a unit circle $S^1$ in $\R^2$ and the restriction $\ii^-|_{S^1}$; also for $j\geq 0$ we let:
$$P_{j}=\{\eta \in S^1\,|\, \ii^{-}(\eta q)\leq j\}.$$
The following formula (\ref{indexformula}) is proved in \cite{LerarioConvex} and relates the Betti numbers of $Y$ to the topology of the sets $P_j$; we denote by $\tilde{b}_j(Y)$ the rank of $\tilde{H}^j(Y;\mathbb{Z}_2)$. The general bound (\ref{bettibound}) for the topology of $Y$ in the case $l\geq 2$ is proved in \cite{LerarioComplexity}. The reader is referred to \cite{AgrachevLerarioSystems, LerarioConvex, LerarioComplexity} for more details.

\begin{propoa}\label{propo:quadrictopology}If $Y$ is the intersection of \emph{two} quadrics on the sphere $S^n$ and $0\leq j\leq n-3$, then:
\begin{equation}\label{indexformula}\tilde{b}_{j}(Y)=\tilde{b}_{n-j-1}(S^{n}\backslash Y)=b_{0}(P_{j+1},P_{j})+b_{1}(P_{j+2},P_{j+1}).\end{equation}
Moreover if $Y$ is defined by $l\geq 2$ quadratic equations on $S^n$, on $\mathbb{R}P^n$ or in $\R^n$, then:
\begin{equation}\label{bettibound}b(Y)\leq O(n)^{l-1}.\end{equation}
\end{propoa}

It is possible to apply the above technique also in the case $Y$ is the intersection of quadrics on the unit sphere in some (infinite dimensional) Hilbert space $H$. The main differences for this infinite dimensional case are the following: $Y$ must be nonsingular; $\check{b}_i$ denotes the rank of the $i$-th \emph{Cech} cohomology group; the negative inertia index might be infinite for some $\eta\in S^1$, but these $\eta$ are already excluded by the condition $\ii^{-}(\eta)\leq j<\infty$. With these modification we have the following result from \cite{AgrachevQuadrics}; formula (\ref{indexformulainfinite}) is the analogue of (\ref{indexformula}), but the condition that $H$ is infinite dimensional allows to remove the restriction on the range for $j$.

\begin{teoa}\label{teo:quadrictopologyinfinite}Let $q_1, \ldots, q_l$ be \emph{continuous} quadratic forms on the Hilbert space $H$ and $Y$ be their (nondegenerate) common zero locus on the unit sphere. Then:
\begin{equation}\label{bettidirect}H_{*}(Y)=\varinjlim_{V\in \mathcal{F}}\{H_*(Y\cap V)\}.\end{equation}
where $\mathcal{F}$ denotes the family of all \emph{finite} dimensional subspace of $H$. Moreover in the case $l=2$:
\begin{equation}\label{indexformulainfinite}\tilde{b}_j(Y)=\check{b}_{0}(P_{j+1},P_{j})+\check{b}_{1}(P_{j+2},P_{j+1})\end{equation}
\end{teoa}
 
 \begin{proof}[Sketch]
Since $Y$ is assumed to be nonsingular, then it has a tubular neighborhood $U$ in $H$ and $H_*(Y)\simeq H_*(U).$ In particular every singular chain in $U$ is homotopic to one whose image is contained in a finite dimensional subspace and (\ref{bettidirect}) follows.
To prove (\ref{indexformulainfinite}) we fix a $j\geq 0$; then using (\ref{indexformula}) we have:
$$\tilde{b}_{j}(Y\cap V)=b_{0}(P_{j+1}(V),P_{j}(V))+b_{1}(P_{j+2}(V),P_{j+1}(V))$$
where $V\subset H$ is a sufficiently big finite dimensional subspace (the condition required on the dimension is $\dim(V)-3\geq j$) and $P_j(V)=\{\eta\in S^1\,|\, \ii^{-}(\eta q|_{V})\leq j\}.$
Now the sets $\{P_j(V)\}_{V\in \mathcal{F}}$ are also partially ordered by inclusion: if $V_1\subset V_2$, then $P_j(V_2)\subset P_j(V_1)$. It is not difficult to show that under the isomorphism $\tilde{H}_j(Y\cap V)\simeq H^0(P_{j+1}(V),P_{j}(V))\oplus H^1(P_{j+2}(V),P_{j+1}(V))$ the inclusion morphism on the homology $H_j(Y\cap V_1)\to H_{j}(Y\cap V_2)$ is induced by the restriction morphism (see \cite{AgrachevQuadrics}):
\begin{equation}\nonumber
\bigoplus_{i=0,1}H^i(P_{j+i+1}(V_1),P_{j+i}(V_1))\to\bigoplus_{i=0,1}H^i(P_{j+i+1}(V_2),P_{j+i}(V_2))
\end{equation}
%$$
%\begin{array}{c}
%    H^0(P_{j+1}(V_1),P_{j}(V_1)\oplus H^1(P_{j+2}(V_1),P_{j+1}(V_1))  \\
%    \downarrow                                                        \\
%    H^0(P_{j+1}(V_2),P_{j}(V_2))\oplus H^1(P_{j+2}(V_2),P_{j+1}(V_2)) \\
%\end{array}
%$$
Since the sets $\{P_j(V)\}_{V\in \mathcal{F}}$ are Euclidean Neighborhood Retracts (being semialgebraic sets), then by the continuity property of Cech cohomology:
$$\tilde{H}_*(Y)=\varprojlim_{V\in \mathcal{F}}\{H^*(P_{j+1}(V),P_{j}(V))\}=\check{H}^*\big(\bigcap_{V\in \mathcal{F}}P_{j+1}(V), \bigcap_{V\in \mathcal{F}}P_{j}(V)\big).$$
Finally $P_j$ equals by construction $\bigcap_{V\in \mathcal{F}}P_{j}(V)$ and the conclusion follows.
\end{proof}

{\bf Acknowledgements.} The first author is supported by Grant of the Russian Federation for the State Support of Researches (Agreement No 14.B25.31.0029)

\end{document}